\documentclass[11pt]{amsart}

\usepackage{amsmath, amsthm, amssymb, mathrsfs}
\usepackage[hidelinks]{hyperref}
\usepackage[parfill]{parskip}    \parskip = 0.2cm    
\usepackage[margin=1in]{geometry} 
\usepackage{tikz}
\usepackage{tikz-cd}
\usetikzlibrary{matrix}
\usepackage{comment} 
\usepackage[activate={true,nocompatibility},final,tracking=true,kerning=true,spacing=true,factor=1100,stretch=10,shrink=10]{microtype}

\usepackage{commutative-diagrams}

\usepackage{mathtools}

\graphicspath{ {images/} }
\usepackage{import}

\numberwithin{equation}{section}

\title[Recognisability for generalised FLC pattern spaces]{\bf Recognisability for generalised hierarchical pattern spaces of finite local complexity}
\date{\today}

\author{James J.\ Walton} 
\address{School of Mathematical Sciences, Mathematical Sciences Building, University Park, Nottingham, NG7 2RD, United Kingdom}
\email{Jamie.Walton@nottingham.ac.uk}
\urladdr{https://www.nottingham.ac.uk/mathematics/people/jamie.walton}

\theoremstyle{plain}
\newtheorem{theorem}{Theorem}[section]
\newtheorem{lemma}[theorem]{Lemma}

\newtheorem{proposition}[theorem]{Proposition}
\newtheorem{corollary}[theorem]{Corollary}
\newtheorem{question}[theorem]{Question}

\newtheorem{innercustomthm}{Theorem}
\newenvironment{customthm}[1]
  {\renewcommand\theinnercustomthm{#1}\innercustomthm}
  {\endinnercustomthm}

\newtheorem{innercustomcor}{Corollary}
\newenvironment{customcor}[1]
  {\renewcommand\theinnercustomcor{#1}\innercustomcor}
  {\endinnercustomcor}

\theoremstyle{definition}
\newtheorem{definition}[theorem]{Definition}
\newtheorem{remark}[theorem]{Remark}
\AtEndEnvironment{remark}{\null\hfill\exend}
\newtheorem{example}[theorem]{Example}
\AtEndEnvironment{example}{\null\hfill\exend}
\newtheorem{notation}[theorem]{Notation}

\newcommand{\exend}{\hfill \ensuremath{\Diamond}}

\newcommand{\R}{\mathbb R}

\newcommand{\Z}{\mathbb Z}
\newcommand{\Q}{\mathbb Q}

\newcommand{\N}{\mathbb N}

\newcommand{\sH}{\mathscr{H}}

\newcommand{\cB}{\mathcal{B}}

\newcommand{\subpatch}{\triangleleft}

\renewcommand{\epsilon}{\varepsilon}

\DeclareMathOperator{\id}{\mathrm{id}}

\newcommand{\sub}{\sigma}

\DeclareMathOperator{\cP}{\mathcal{P}} 
\DeclareMathOperator{\cQ}{\mathcal{Q}} 
\DeclareMathOperator{\cR}{\mathcal{R}}
  
\DeclareMathOperator{\cU}{\mathcal{U}}
\DeclareMathOperator{\cV}{\mathcal{V}} 
\DeclareMathOperator{\cK}{\mathcal{K}}
\DeclareMathOperator{\cT}{\mathcal{T}}
\DeclareMathOperator{\cL}{\mathcal{L}}
\DeclareMathOperator{\cW}{\mathcal{W}}

\newcommand{\xsra}[1]{\overset{#1}{\rightsquigarrow}}

\newcommand{\xsrla}[1]{\overset{#1}{\leftrightsquigarrow}}
\DeclareMathOperator{\LD}{\scalebox{0.8}{\(\xsra{\mathrm{LD}}\)}} 
\DeclareMathOperator{\MLD}{\scalebox{0.8}{\(\xsrla{\mathrm{MLD}}\)}}
\DeclareMathOperator{\LDmap}{\scalebox{0.8}{\(\xrightarrow{\mathrm{LD}}\)}} 

\newcommand{\xsqsubset}[1]{\overset{#1}{\sqsubset}}
\newcommand{\xin}[1]{\overset{#1}{\in}}
\DeclareMathOperator{\LI}{\scalebox{0.8}{\(\xsqsubset{\mathrm{LI}}\)}}
\newcommand{\xsimeq}[1]{\overset{#1}{\simeq}}
\DeclareMathOperator{\LIs}{\scalebox{0.8}{\(\xsimeq{\mathrm{LI}}\)}}
\DeclareMathOperator{\LIin}{\scalebox{0.8}{\(\xin{\mathrm{LI}}\)}}

\newcommand{\xsubset}[1]{\overset{#1}{\subset}}
\DeclareMathOperator{\cpt}{\scalebox{0.8}{\(\xsubset{\mathrm{c}}\)}} 
\DeclareMathOperator{\nhd}{\scalebox{0.8}{\(\xsubset{\odot}\)}}
\DeclareMathOperator{\cptn}{\scalebox{0.8}{\(\xsubset{\mathrm{c-}\odot}\)}}

\newcommand{\lang}{\mathscr{L}}

\newcommand{\UC}{\(\mathrm{(UC)}\)}


\begin{document}

\begin{abstract}
We develop a general framework of Euclidean patterns and pattern spaces of translational finite local complexity (FLC), analogues of translational tiling spaces. The notion of a self affine substitution of tilings is extended to both individual patterns and pattern spaces, which we ask are mapped onto by a local derivation map (the analogue of a sliding block code map from Symbolic Dynamics) from its own expansion by a linear map. We prove recognisability for these, with no minimality requirements. In particular we show that, for each pattern in the space, its substitutional pre-images are translation equivalent. Sizes of all fibres are then determined by relative groups of translational periods. This answers an open question of Cortez and Solomyak, on whether non-periodic tilings necessarily have unique pre-images under substitution: they do, even for a wider notion of pattern and being substitutional. It is shown that there exists a power of the substitution under which any given pattern of the pattern space has multiple pre-images if and only if it has disconnected group of periods. This is equivalent to being periodic in the standard return discrete cases of tilings and Delone sets, but our results also cover examples with non-discrete groups of periods, such as spaces of uniformly discrete but non-relatively dense point sets.
\end{abstract}

\maketitle

\section{Introduction}

Many aperiodically ordered patterns exhibit a form of repeating hierarchy. Just as self-similar fractals have recurring features when zooming in, hierarchical patterns can have recurring structures when zooming out, so are like `upside down fractals'; this goes further than analogy, see for instance \cite{BGN03,WW24}. A famous example is the family of Penrose tilings \cite{Pen79}: on top of any Penrose tiling \(\cP\), one may place a \emph{supertiling} \(L\cP'\) that is an inflation \(L\) (scaling by the golden mean) of a `predecessor' Penrose tiling \(\cP'\) that derives the original by a replacement rule on inflated tiles; see Figure \ref{fig:pen_sub}. That \(\cP'\) is also a Penrose tiling means \(\cP'\) is locally indistinguishable from \(\cP\), that is, its finite patches are all also in \(\cP\) (modulo translation). As described in \cite{Pen79}, they were discovered through consideration of hierarchy, more precisely experimentation with substitution of pentagons, as illustrated in Figure \ref{fig:several subs}. Other important examples, though --- notably, the families of hat and spectre monotilings \cite{SMKG24a,SMKG24b} --- were initially found by other means and motivated by other properties (e.g., forcing non-periodicity with a single tile) yet were \emph{later} found to also be hierarchical. A wealth of hierarchical tilings are now known, constructed by substitution rules on prototile sets, which can be perused on the Tilings Encyclopedia \cite{TilEnc}.

\begin{figure}
\includegraphics[width=\textwidth]{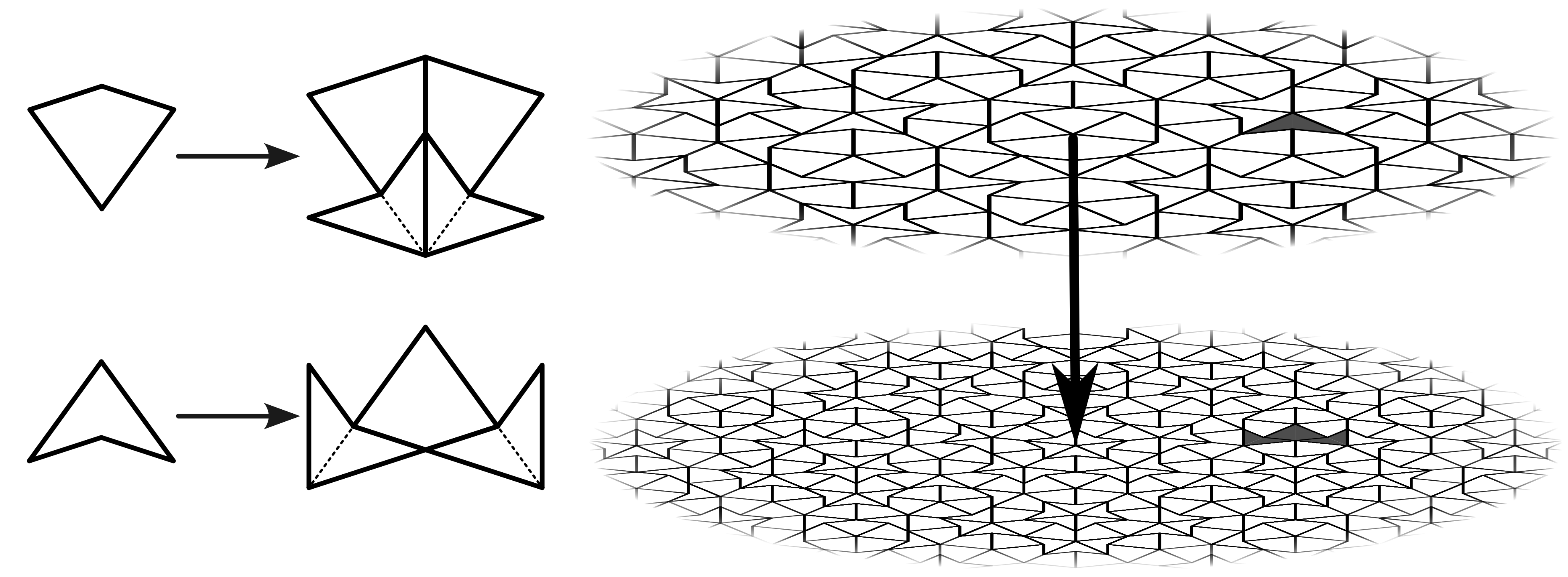}
\caption{The Penrose kite and dart substitution \(\sub\) (left) and a Penrose tiling \(\cP = \sub(\cP') = S(L\cP')\) (lower right) as the subdivision \(S\) of the unique supertiling \(L\cP'\) (upper right), with an example highlighted supertile and its replacement below.
}\label{fig:pen_sub}
\end{figure}

In addition to tilings, aperiodic patterns are often represented by point sets, typically Delone sets, which can be used to model idealised atomic positions of quasicrystals \cite{SBGC84,Hof95}. This is a particularly useful perspective for the cut and project (or `model set') construction \cite{Mey72,Moo00}, which more naturally defines point sets rather than tilings. There is also a standard notion of a Delone (multi)set being `substitutional' \cite{LW03}.

Recognisability results equate a form of injectivity of substitution to aperiodicity, which is integral to the structure of their associated dynamical systems. In Symbolic Dynamics, Moss\'{e}'s Theorem \cite{Mos92,Mos96} establishes recognisability for aperiodic subshifts generated by primitive substitution. This has been expanded to a wide class of non-minimal subshifts \cite{BKM09,BPR23}, and also \(S\)-adic systems \cite{BSTY19,BPRS25}. In the geometric setting, to the author's knowledge, only substitution \emph{tilings} \cite{Thu89} have been directly treated before (although some other repetitive examples can be covered using \cite{Sol07,PS01}). An vital result for the topological and non-commutative theory of substitution tilings \cite{AP98}, is Solomyak's recognisability theorem, which shows aperiodicity is equivalent to injectivity of the substitution map on the translational hull of a primitive stone inflation tiling \cite{Sol98} (stone inflations are tile substitutions that replace inflated tiles with patches of equal support, as for the chair substitution in Figure \ref{fig:several subs}). In this context, recognisability is also called `unique composition', since it is equivalent to there being a unique supertiling --- the tiling that decomposes to the original --- being unique and, moreover, defined by a local rule grouping tiles into supertiles. This has been extended to a reasonably wide class of non-primitive (non-repetitive) admissible stone inflations \cite{CS11}.

The current paper advances recognisability results in the geometric setting, by extending to general kinds of patterns (rather than just tilings), using a broader notion of a pattern (or space of patterns) being substitutional, fully dropping minimality and allowing for a mix of periodic and non-periodic elements, with a formula for the number of pre-images of any pattern under substitution in terms of groups of translational periods. Moreover, a major aim of the paper is to provide a robust foundational framework for the theory of such patterns and pattern spaces. Before stating our main recognisability results, we will briefly outline how this generalised setting is established.

Throughout, \(E \cong \R^d\) will denote the ambient space of our patterns. In Symbolic Dynamics, a \emph{bi-infinite word} over a (typically finite) alphabet \(A\) is a function \(w \colon \Z \to A\) i.e., \(w \in A^\Z\), we label (or `colour') each point of \(\Z\) with labels in \(A\). In this paper, we define a \emph{pattern} to simply be some \(\cP \in A^E\), that is, a labelling of our Euclidean space with labels in \(A\). This starting point conveniently accommodates a broad range of pattern types together under one hood, including tilings, (coloured) point sets (such as Delone sets but also uniformly discrete and non-relatively dense point sets), coverings of \(E\) by overlapping tiles or tilings not fully covering \(E\) and more, see Figures \ref{fig:pen_sub}--\ref{fig:digit sub}.

For hierarchical patterns, it is customary to begin with a substitution rule on a finite set of `atoms' e.g., a set of prototiles of a stone inflation rule. Given the broad range of scenarios alluded to above, this is no longer simplest approach, or even feasible; in any case, it is useful to allow further flexibility. Fortunately, an encompassing property of such setups can be identified, extendible to general pattern spaces: the possibility to always `de-substitute' patterns within the space. For instance, from a symbolic substitution \(\sub\) defining a subshift \(X_\sub\), the substitution map \(\sub \colon X_\sub \to X_\sub\) is `representable' in \(X_\sub\) \cite{BPR23}, meaning that, for all \(w \in X_\sub\), there exists \(w' \in X_\sub\) for which \(w\) is some shift of \(\sub(w')\); even in the \(S\)-adic case, substitutions between levels are always `eventually representable' \cite{BPRS25}. In the geometric setting, under usual conditions (but even without primitivity), substitution is simply surjective on the tiling space, which may be shown by a simple Cantor diagonalisation argument (see \cite{AP98}). Such substitutions are compositions `inflate, replace'. The `inflate' step is usually given by application of an expansive linear map \(L \colon E \to E\). The `replace' step is defined in a local manner, as illustrated on the right in Figure \ref{fig:pen_sub}.

\begin{figure}
\includegraphics[width=0.8\textwidth]{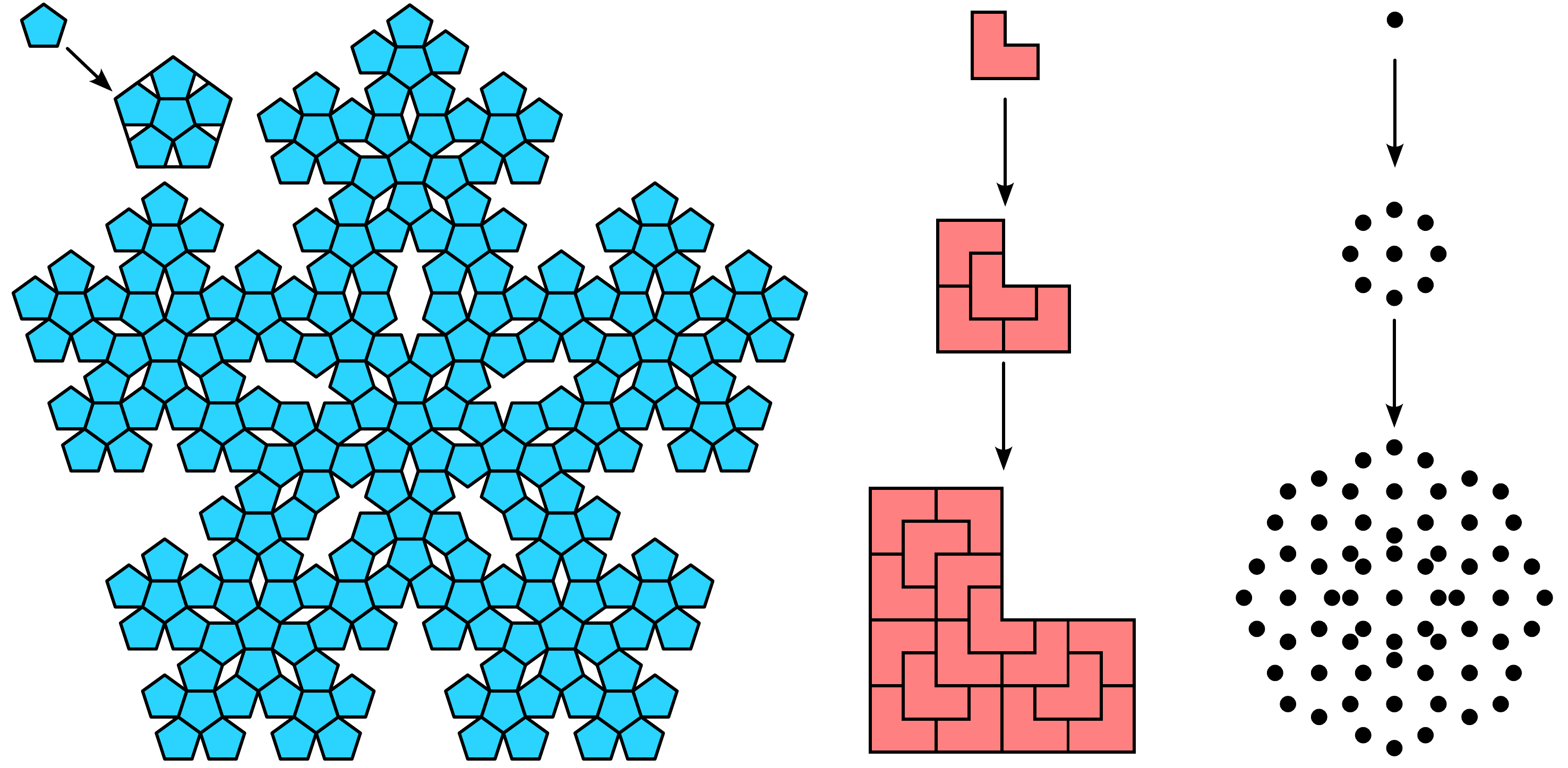}
\caption{Left: the inflation of a pentagon by the square of the golden mean may be partially tiled by six pentagons (the inflated pentagon boundary is shown for reference). The third iteration of the resulting `inflate, replace' rule is also shown. Middle: two iterations of the chair stone inflation rule. Right: two iterations of a point set substitution with inflation \(1+\sqrt{2}\). Each of these may be iterated to define \(L\)-sub pattern spaces (of tilings with gaps, tilings and Delone sets, respectively).}\label{fig:several subs}
\end{figure}

This guides us to the following definitions. A pattern space \(\Omega\) (see Definition \ref{def:pattern space}) will be called \(L\)-sub when it is equipped with a surjective local derivation (LD) map \(S \colon L\Omega \LDmap \Omega\). We describe the requirements on \(\Omega\) in more detail below, but the reader may have in mind a space of tilings \cite{Sad08} generated by a tile substitution rule of finite local complexity (FLC, see \cite{BG13}). The map \(S\) can be thought of as the `replace' operation (as in Figure \ref{fig:pen_sub}), although we also refer to it as `subdivision'. It maps from the pattern space of inflated patterns to the original. See Definition \ref{def:LD map} for the notion of a map of pattern spaces being LD. It is the geometric analogue of a sliding block code factor map from Symbolic Dynamics although note that it cannot be identified with simply being a factor map in this geometric setting (see \cite{CS06}). The requirement that \(S\) is surjective is quite natural: it demands that patterns of \(\Omega\) are all `hierarchical', in that they are all subdivisions of other inflated patterns of \(\Omega\). We define substitution \(\sub \colon \Omega \to \Omega\) as the composition `inflate, replace' i.e., \(\sub \coloneqq S \circ L\).

Throughout, we will use the `local topology' \cite{BG13,Sad08} on pattern spaces \(\Omega\), whereby two patterns are considered `close' if they agree to a large radius about the origin, up to a small translation. We will require that \(\Omega\) is Hausdorff, which is essentially automatic in cases of interest (and which we show may be expressed in simple terms of patterns, see Section \ref{sec:separation properties}). Moreover, we will require that \(\Omega\) is compact. Again, this may be identified with a property expressed in terms of patterns, finite local complexity (FLC), see Definition \ref{def:FLC pattern space}. In the case that \(\Omega\) is a space of (locally finite) tilings, this is the familiar property that there are only a finite number of tiles and ways they can meet, modulo translation, across all tilings in \(\Omega\). We need no further restrictions on \(\Omega\) (notably, minimality is not assumed) for our following main recognisability result.

\begin{customthm}{\ref{thm:recognisability}}
For any compact Hausdorff \(L\)-sub pattern space \(\Omega\), with \(L\) expansive (i.e., with all eigenvalues strictly greater than \(1\) in modulus), substitution \(\sub\) satisfies unique composition modulo translation. That is, for all \(\cP \in \Omega\) and \(\cU\), \(\cV \in \sub^{-1}(\cP)\), we have \(\cU = \cV + x\) for some \(x \in E\).
\end{customthm}

Several consequences quickly follow, such as relating uniqueness of pre-images to non-periodicity. In more detail, for a pattern \(\cP\), we let \(\cK_{\cP}\) denote its group of translational periods i.e., the set of \(x \in E\) for which \(\cP = \cP+x\), and call \(\cP\) non-periodic when \(\cK_{\cP} = \{\mathbf{0}\}\). In Proposition \ref{prop:number of pre-images}, the following is shown to follow from Theorem \ref{thm:recognisability}: for all \(\cP \in \Omega\) and \(\cP' \in \sub^{-1}(\cP)\), we have
\[
\# \sub^{-1}(\cP) = [\cK_{\cP} : L \cK_{\cP'}] < \infty .
\]
In particular, non-periodic patterns have unique pre-images under substitution.

Under the hypotheses of Theorem \ref{thm:recognisability}, a useful property is established in Theorem \ref{thm:finite languages of hierarchical}, that \(\Omega\) contains only finitely many different LI-classes \cite{BG13} (in Symbolic Dynamics terms: only finitely many `languages'). Thus, some power \(\sub^N\) of substitution fixes LI-classes. With the above, this yields Corollary \ref{cor:uniqueness of pre-images}, that if \(\cK_{\cP}\) is connected (in particular, if \(\cP\) is non-periodic) then \(\cP\) has a unique pre-image and, conversely, \(\cP\) has multiple pre-images under \(\sub^N\) when \(\cK_{\cP}\) is disconnected. In the latter case, we say that \(\cP\) is discretely periodic (and discretely non-periodic otherwise). Note that this power cannot be dropped in general: discretely periodic points may have unique pre-images under smaller powers of substitution, even for the simple case of hulls of 1d patterns coming from (non-primitive) constant length substitutions, see Example \ref{exp:subs not LI}.

By the above, we obtain a dichotomy result (see Corollary \ref{cor:injectivity dichotomy}), that substitution \(\sub \colon \Omega \to \Omega\) is injective if and only if each \(\cP \in \Omega\) is discretely non-periodic. In many standard classes of interest (such as hulls of FLC tilings or Delone sets), a return discreteness property holds (see Definition \ref{def:return discrete}), which ensures that all groups of periods are discrete. In this case, discrete periodicity is equivalent to periodicity, so that \(\sub \colon \Omega \to \Omega\) is injective if and only if \(\Omega\) has no periodic elements.

In \cite[Theorem 4.4]{CS11}, Cortez and Solomyak show a recognisability result for tiling spaces generated by FLC, admissible stone inflations satisfying the non-periodic border condition. They note the following open question: more generally, is it still true that non-periodic elements of a tiling space generated by an FLC stone inflation rule necessarily have unique pre-images under substitution? Since such systems are covered by the above results, this indeed holds, even in the wider framework of (compact Hausdorff) expansive \(L\)-sub pattern spaces with the local topology, and a form of converse as explained above also holds.

The above discussion on discrete periodicity shows one aspect of how our results extend some previous ones for tilings to general patterns: we generally do not even assume return discreteness, so the groups of periods \(\cK_{\cP}\) need not be discrete (although, for a compact Hausdorff pattern space, we show that return discreteness holds precisely when all groups of periods are discrete, see Proposition \ref{prop:return discrete from discrete periods}). Consider, for instance, the `empty' pattern \(\cP\), which looks the same at all points of \(E\) (see Example \ref{exp:trivial patterns}). This is as periodic as could be: \(\cK_{\cP} = E\), thus \([E : LE] = 1\), so when \(\cP \in \Omega\) it has unique pre-image (under all powers of \(\sub\)). This example is not so artificial, it is in any pattern space containing a uniformly discrete but non-relatively dense point set.

Beyond this, and dropping minimality, another sense in which our results generalise some previous ones is that we need not restrict to pattern spaces defined by languages `generated' by a substitution on a prototile set or similar, we essentially just need surjectivity of substitution (with linear expansion). In fact, even if we drop surjectivity of substitution \(\sub \colon \Omega \to \Omega\), the results above remain true for the subset \(\sH \subseteq \Omega\) of hierarchical elements, those that have an infinite chain of pre-images. Substitution restricts to a surjective map on \(\sH\), making it an \(L\)-sub pattern space to which our main results apply (see Proposition \ref{prop:restriction to hierarchical elements}; note that \(\sH\) is simply the eventual range of \(\sub\)). Even when starting with a standard tile substitution rule, with an appropriate `full' space of patterns \(\Omega\) (one on which we may define the substitution map \(\sub \colon \Omega \to \Omega\)), the space \(\sH\) typically contains extra patterns to the usual space \(\Omega_\sub\) of tilings admitted by substitution \cite{AP98}, see Section \ref{L-sub pattern spaces from iterated inflate, replace rules} for more details. The particular case of digit (or constant shape) substitutions on \(\Z^d\) is considered in Section \ref{sec:digit subs}. These define a discrete substitution \(\sub \colon A^{\Z^d} \to A^{\Z^d}\) for a finite alphabet \(A\). Substitution extends to a map on the suspension (a tiling space, which is an \(\R^d\)-dynamical system) and our recognisability results above apply to elements in the eventual range of this substitution.

\begin{figure}
\includegraphics[width=0.8\textwidth]{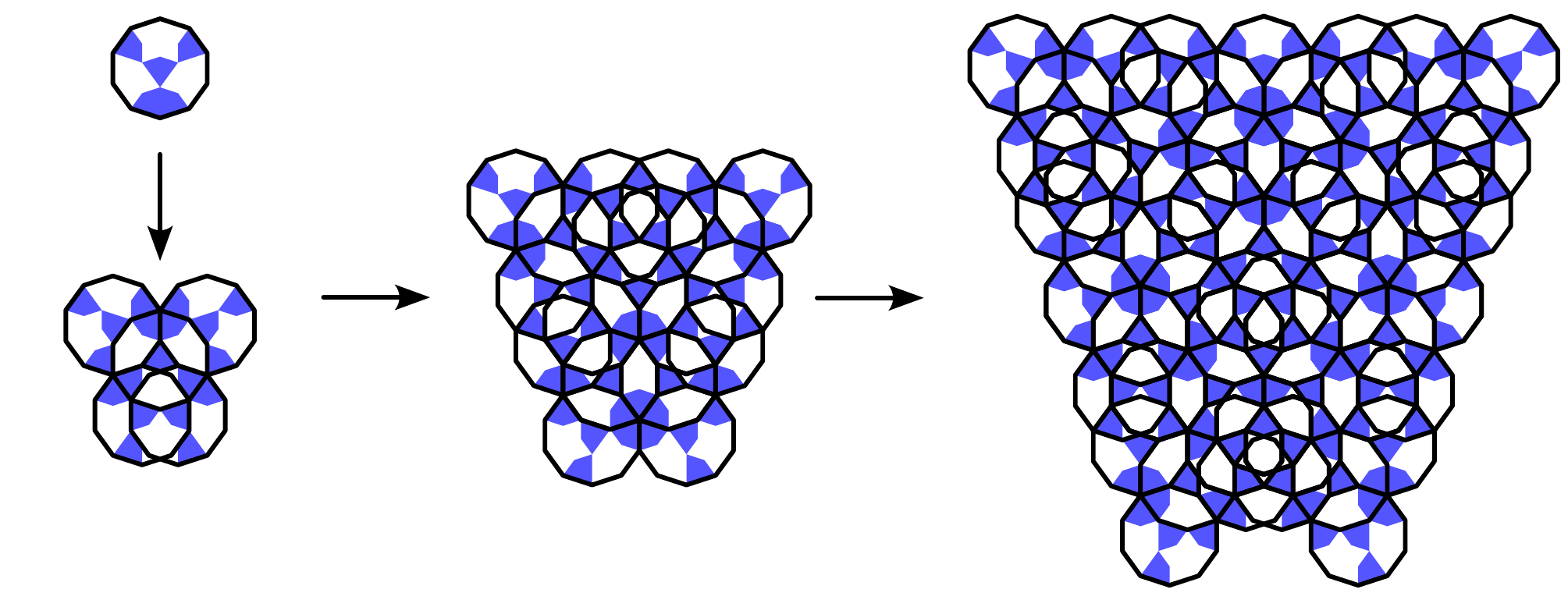}
\caption{Three iterations of a substitution \cite{Jeo03} of the Gummelt decagon \cite{Gum96}, which generates coverings MLD to the Penrose tilings.}\label{fig:gummelt}
\end{figure}

Our definition of a pattern space \(\Omega\) being \(L\)-sub was preceded by consideration of similar notions for individual patterns. A pattern \(\cP\) is called \(L\)-sub if there exists a pattern \(\cP' \LI \cP\) for which \(L\cP' \LD \cP\). Here, \(\cP' \LI \cP\) is read \emph{\(\cP'\) is locally indistinguishable from \(\cP\)} and is equivalent to it being in the orbit closure of \(\cP'\) (or, in symbolic terms, its language being a subset of that of \(\cP\)). The notation \(L\cP' \LD \cP\) is read \emph{\(\cP\) is locally derivable \cite{BSJ91} from \(L\cP'\)} and, again, can be thought of (in symbolic terms) as there being a sliding block code redecorating \(L\cP'\) to \(\cP\), as in Figure \ref{fig:pen_sub}. This useful particular notion of an individual pattern being \(L\)-sub was identified through the collaboration \cite{HKW25}, which considers the question of when a Euclidean cut and project scheme produces substitutional patterns. It is also a slight relaxation of the notion of a pattern being pseudo self-affine i.e., \(L\cP \LD \cP\), see for instance \cite{PS01, Sol07}, which is also known as having a local inflation deflation symmetry (LIDS, see \cite{BSJ91,BG13}). The difference is that this captures the notion of a pattern substituting to \emph{itself} under some substitution rule, whereas \(L\)-sub captures the more relaxed condition of being `generated' by some kind of substitution. Under some basic assumptions (see Theorem \ref{thm:predecessor LI}), any pattern \(\cP\) that is \(L\)-sub has hull (or `orbit closure' \(\Omega_{\cP} = \overline{\cP+E}\)) that is itself naturally \(L\)-sub, so we obtain an associated substitution map \(\sub \colon \Omega_{\cP} \to \Omega_{\cP}\).

For referencing purposes and to unpack some definitions, we give a recognisability dichotomy result for \emph{individual} \(L\)-sub patterns below (rather than pattern spaces) that are return discrete (see Definition \ref{def:return discrete}), which includes standard examples of tilings and Delone sets.

\label{sec:intro_cor}
\begin{customcor}{\ref{cor:unique decomposition for individual patterns}}
Let \(\cP\) be an FLC and return discrete pattern of the Euclidean space \(E\) and \(L \colon E \to E\) a linear expansion. Suppose that \(\cP\) is \(L\)-sub, that is, there is a pattern \(\cP'\) for which \(\cP' \LI \cP\) and \(L\cP' \LD \cP\). Then \(\cP' \LIs \cP\) (that is, we also have \(\cP \LI \cP'\)). We have that \(\cP \LD L\cP'\), equivalently \(\sub \colon \Omega_{\cP} \to \Omega_{\cP}\) is injective, if and only if \(\cP\) is aperiodic, that is, \(\Omega_{\cP}\) contains no periodic elements.
\end{customcor}

The definition above of being \(L\)-sub is reasonably versatile. And being defined in terms of the fundamental relations of local isomorphism/derivability makes it a practical one in other applications \cite{HKW25}, whilst still providing a similar definition (especially in the minimal case) to the familiar one, of being generated by an `inflate, replace' procedure.

Note that use of the local topology on pattern spaces \(\Omega\) (along with compactness, or FLC) is an important restriction of our results: they do not cover pinwheel-type examples \cite{Rad94}, where recognisability is known in the minimal case \cite{HRS05}. Such tilings have infinitely many tile types modulo translation but only finitely many modulo rigid motion. In the symbolic setting, one may also consider appropriately topologised infinite alphabet substitutions \cite{DOP18,MRW25} and it would be interesting to investigate if the techniques here may be extended to any classes of such non-FLC patterns. For \(S\)-adic shifts, a general eventual recognisability result is now known \cite{BPRS25}, see also \cite{BSTY19}. It seems likely that geometric analogues can be formulated. Through the techniques developed in this paper, we are also able to weaken the expansive linearity assumption for substitution to a non-linear expansivity type of condition, something to which we will return in future work.

\subsection{Organisation of paper}

Generalised Euclidean patterns are introduced in Section \ref{sec:Generalised patterns}. Key concepts are patches and languages of patterns. Basic results and notation for patches are given and we introduce local indistinguishability and local derivability of general patterns, fundamental relations that are used throughout the paper (including in the definition of a pattern being \(L\)-sub). A simple but key observation is made in Lemma \ref{lem:close repeat => period}, that a return vector in \(U \subset E\) to a patch containing centres of all \(U\)-patches must be a period of the pattern (informally: relatively short return vectors of `full' patches are periods), which is a suitable replacement (but one applicable to non-repetitive patterns) of a similar property used in Solomyak's proof \cite{Sol98} of unique composition.

In Section \ref{sec:pattern spaces} we introduce pattern spaces, (generalisations of tiling spaces \cite{Sad08}). Fundamental properties of them are established, in many cases extensions of some familiar ones from tiling spaces. This includes the identification of compactness of a pattern space \(\Omega\) with it being FLC, and minimality with repetitivity (or of containing a single LI-class, see Proposition \ref{prop:repetitive <=> minimal}). We explain how several topological properties of a pattern space \(\Omega\) can be understood directly in terms of patterns, including separation properties; for instance, it is shown that a pattern space is compact Hausdorff if and only if it is FLC and `well-separated', see Definitions \ref{def:FLC pattern space} and \ref{def:well-separated}. The notion of an LD map of pattern spaces is given in Definition \ref{def:LD map}. As usual for tiling spaces, bijective LD maps between (compact Hausdorff) pattern spaces have inverses that are also LD (Proposition \ref{prop:inverse of LD is LD}). We also define the `hull' (the orbit closure) \(\Omega_{\cP}\) of an individual pattern \(\cP\), and relate many of these properties to corresponding properties of \(\cP\) (see, for instance, Propositions \ref{prop:FLC and repetitive inherited} and \ref{prop:induced LDs}).

Readers primarily interested in recognisability results are recommended to skip most of the details of Section \ref{sec:pattern spaces} on a first reading, as many of the definitions and results here will be familiar, at least to those accustomed to tiling spaces (or perhaps merely suspensions of subshifts). In Section \ref{sec:Substitutional pattern spaces}, we introduce the notion of a pattern space being substitutional with respect to a linear automorphism \(L\) of \(E\); in most results, \(L\) also needs to be expansive. Some key results and notation needed for recognisability are given in Section \ref{Languages for expansive L-sub pattern spaces}, such as how substitution defines a map between (levels of) the language of an \(L\)-sub pattern space \(\Omega\) and that \(\Omega\) contains only finitely many LI-classes (in the compact expansive case, see Theorem \ref{thm:finite languages of hierarchical}). By restriction to hulls, similar notions for individual \(L\)-sub patterns are considered in Section \ref{sec:Substitutional hulls and L-sub for individual patterns}. An important result for recognisability is Proposition \ref{prop:full patch at origin} which, loosely, bounds how far one needs to look to find all patches of a given size in any pattern of an \(L\)-sub pattern space. This pairs with Lemma \ref{lem:close repeat => period} (mentioned above) in forcing a certain translation to be a period in the proof of recognisability.

Our main recognisability results are stated in Section \ref{sec:recognisability}. Some main consequences (such as a formula for the size of fibres, and injectivity versus non-periodicity dichotomy results) are shown first, from the lynchpin Theorem \ref{thm:recognisability}, whose proof is given afterwards.

Finally, in Section \ref{sec:examples}, we demonstrate our recognisability results on some examples. Most are chosen to clarify how the results go beyond the minimal case, or to examples lacking return discreteness; the latter allows for `unusual hulls', namely, ones that are not torus fibre bundles with totally disconnected fibres \cite{SW03} (hulls could now be spheres or intervals, say, see Example \ref{exp:unusual hulls}). On dropping minimality, Example \ref{exp:subs not LI} shows that even for hulls of 1d tilings defined by constant length non-primitive substitution, LI-classes need not be fixed by substitution and periodic points can (at least under the first power) have unique pre-images. We conclude by explaining how our results cover all hierarchical tilings defined by digit or `constant shape' substitutions over \(\Z^d\).

\subsection*{Acknowledgements}
The author thanks Michael Baake, Edmund Harriss, Henna Koivusalo, Neil Ma\~{n}ibo, Dan Rust and Reem Yassawi for helpful discussions in the writing of this paper. We also thank Marie-Pierre B\'{e}al, Val\'{e}rie Berth\'{e}, Dominique Perrin and Wolfgang Steiner for their kind input and enlightening explanations regarding recognisability in the symbolic setting during a research trip to IRIF in Paris.

\section{Generalised patterns}
\label{sec:Generalised patterns}

Throughout we let \(E \cong \R^d\) (\(d \in \N \coloneqq \{1,2,\ldots\}\)) denote a finite-dimensional vector space over \(\R\), always the ambient space of our patterns. The additive identity of \(E\) is denoted by \(\mathbf{0}\). We fix a norm on \(E\); the exact choice is not important, although we have some requirements of it once the inflation \(L \colon E \to E\) of an \(L\)-sub pattern has been introduced, as explained below Definition \ref{def:expansive}. For \(x \in E\) and \(r \geq 0\), we let \(B(x,r)\) denoted the closed ball of radius \(r\) in \(E\), and \(B_r \coloneqq B(\mathbf{0},r)\).

\begin{notation}\label{not:subsets}
We write \(K \cpt E\) to mean that \(K\) is a compact subset of \(E\), we write \(B \nhd E\) to mean that \(B\) is a neighbourhood of \(\mathbf{0} \in E\) and \(B \cptn E\) if \(B\) is a compact neighbourhood of \(\mathbf{0} \in E\).
\end{notation}

\subsection{Patterns and patches}
Throughout, we will use the following notion of a `pattern':

\begin{definition}\label{def:pattern}
For a set \(A\), a \textbf{pattern} \(\cP\) is a map \(\cP \colon E \to A\) i.e., \(\cP \in A^E\). We let \(\cP[x] \coloneqq \cP(x)\).
\end{definition}

\begin{remark}
An alternative generalised notion of a pattern is given by a `discrete structure' \cite[Definition 5.1]{BG13}, although our definition is even more basic. See also \cite{Nag19}.
\end{remark}

As motivated in the introduction, this definition is analogous to that of a bi-infinite word, which instead is a labelling of \(\Z\) over some `alphabet' \(A\), i.e., an element \(w \in A^\Z\). An alternative motivation (that explains our choice of square brackets) is the observation that, in the translational dynamics of patterns, we require at least some way of comparing when patches of certain sizes agree, modulo translation, at different points in the pattern. In particular, we should at least be able to define a notion of agreement locally, over individual points, when `the pattern agrees at \(x\) and \(y\), modulo translation', giving a relation \(\cP[x] = \cP[y]\). Such a relation alone is enough to derive the relations of agreement for all larger patches:

\begin{definition}\label{def:patches}
Given \(\cP \in A^E\), \(U \subseteq E\) and \(x \in E\), we let \(\cP[x,U] \in A^U\) be the function \(u \mapsto \cP[x+u]\). Generally, when \(U \cpt E\), we call an element \(p \in A^U\) a \textbf{patch}. If \(p = \cP[x,U]\) for some \(x \in E\) and \(U \cpt E\), then we say that \(p\) \textbf{is a \(U\)-patch in \(\cP\)} and write \(p \subpatch \cP\). If \(U = B_r\), we call \(\cP[x,B_r]\) an \textbf{\(r\)-patch} and write \(\cP[x,r] \coloneqq \cP[x,B_r]\).

We denote the set of all \(U\)-patches in \(\cP\) by
\[
\lang_U(\cP) \coloneqq \{p \in A^U \mid p \subpatch \cP\} = \{\cP[x,U] \mid x \in E \} .
\]
The \textbf{language} of \(\cP\) is the set of all patches in \(\cP\) (of all shapes), denoted
\[
\lang(\cP) \coloneqq \bigcup_U \lang_U(\cP) ,
\]
where the union is over all \(U \cpt E\). Note that, by definition, a patch \(p \in A^U\) also stores the information of its `shape' \(U\).
\end{definition}

\begin{definition}\label{def:LI}
For patterns \(\cP\), \(\cQ \in A^E\), we say that \textbf{\(\cP\) is locally indistinguishable from \(\cQ\)}, and write \(\cP \LI \cQ\), if whenever \(p \subpatch \cP\) then \(p \subpatch \cQ\), that is, \(\lang(\cP) \subseteq \lang(\cQ)\). If \(\cP \LI \cQ\), and vice versa i.e., \(\lang(\cP) = \lang(\cQ)\), then we call \(\cP\) and \(\cQ\) \textbf{locally isomorphic}. The \textbf{LI-class} of \(\cP\) is the set of all patterns locally isomorphic to \(\cP\), denoted
\[
\mathrm{LI}(\cP) \coloneqq \{\cQ \in A^E \mid \cP \LIs \cQ\} = \{\cQ \in A^E \mid \lang(\cP) = \lang(\cQ) \}.
\]
\end{definition}

\begin{remark}
In the literature, `locally indistinguishable' and `locally isomorphic' are usually used interchangeably, both meaning \(\lang(\cP) = \lang(\cQ)\) in the above notation. However, the one-sided relation in Definition \ref{def:LI} is useful, as we will see in several definitions e.g., the hull of a pattern.
\end{remark}

Since \(\cP \LI \cQ\) if and only if \(\lang(\cP) \subseteq \lang(\cQ)\), the following is trivial:

\begin{lemma}\label{lem:LI is pre-order}
Local indistinguishability is a preorder. That is, for all patterns we have \(\cP \LI \cP\), and if \(\cP_1 \LI \cP_2 \LI \cP_3\) then \(\cP_1 \LI \cP_3\).
\end{lemma}

As usual, we may `transform' patterns given a map of their underlying ambient space:

\begin{definition}\label{def:transformed pattern}
Let \(\cP \in A^E\) and \(T \colon E \to E\) be a bijection (typically, here, a linear map). We define the transformed pattern \(T\cP\) by
\[
(T\cP)[Tx] = \cP[x] \text{ i.e., } (T\cP)[x] \coloneqq \cP[T^{-1}x] .
\]
In the particular case of a translation \(T = (x \mapsto x+z)\), we denote \(\cP + z \coloneqq T\cP\), that is, \((\cP+z)[x] \coloneqq \cP[x-z]\).
\end{definition}

The following basic properties, of \emph{restriction} (1), \emph{transformation} (2), \emph{shifting} (3) and \emph{glueing} (4) of patches, are obvious but often used, especially (3), which we will refer to as \textbf{shifting by \(z\)}. Recall the definition of a transformed pattern from Definition \ref{def:transformed pattern}.

\begin{lemma}\label{lem:pattern properties}
For patterns \(\cP\), \(\cQ \in A^E\),
\begin{enumerate}
	\item if \(\cP[x,V] = \cQ[y,V]\) and \(U \subseteq V\) then \(\cP[x,U] = \cQ[y,U]\);
	\item if \(T(x) = L(x) + z\) is an affine automorphism then \(\cP[x,V] = \cQ[y,V]\) if and only if \((T\cP)[Tx,LV] = (T\cQ)[Ty,LV]\). For \(z \in E\), we have that
	\[
	\cP[x,V] = \cQ[y,V] \text{ if and only if } \cP[x+z,V-z] = \cQ[y+z,V-z];
	\]
	\item if \(\cP[x,V] = \cQ[y,V]\) and \(U+z \subseteq V\) then \(\cP[x+z,U] = \cQ[y+z,U]\);
	\item if \(V = \bigcup_{i \in \mathcal{I}} U_i\) then \(\cP[x,V] = \cQ[y,V]\) if and only if \(\cP[x,U_i] = \cQ[y,U_i]\) for all \(i \in \mathcal{I}\).
\end{enumerate}
\end{lemma}

\begin{proof}
By definition, \(\cP[x,U]\) is just the restriction of the function \(\cP[x,V]\) to \(U \subseteq V\), so (1) is obvious. For (2), for this proof only, we will denote \(\cP[x,V] \cdot v \coloneqq (\cP[x,V])(v) = \cP[x+v]\). Note that \(T^{-1}(x) = L^{-1}(x-z)\) and \(T(x+v) = L(x+v)+z = T(x) + L(v)\) so, by definition of \(T\cP\),
\[
(T\cP)[Tx,LV] =
\left(
LV \ni Lv \mapsto (T\cP)[Tx+Lv] = (T\cP)[T(x+v)] = \cP[x+v] = \cP[x,V] \cdot v
\right)
\]
Analogously, \((T\cQ)[Ty,LV] = (Lv \mapsto \cQ[y,V] \cdot v)\) so, since \(L\) is invertible, the first claim of (2) follows. For the second, given \(z \in E\), we have
\[
\cP[x+z,V-z] =
\left(
V-z \ni v-z \mapsto \cP[(x+z) + (v-z)] = \cP[x+v] = \cP[x,V] \cdot v
\right)
\]
and analogously for \(\cQ\), so the result follows. For (3), we may just combine (1) and (2). Indeed, by (2), we have \(\cP[x+z,V-z] = \cQ[y+z,V-z]\), which we may then restrict to \(U\) using (1) and \(U \subseteq V-z\). Finally, by definition, \(\cP[x,V] = \cQ[y,V]\) if and only if \(\cP[x+v] = \cQ[y+v]\) for all \(v \in V\), from which (4) quickly follows.
\end{proof}

Transforming a pattern by an affine map has the obvious affect on its language of patches:

\begin{lemma}\label{lem:distortions of languages}
Let \(T \colon E \to E\) be affine, with \(T(x) = L(x) + z\) for \(L \colon E \to E\) a linear automorphism. For \(p \in A^U\), define \(Lp \in A^{LU}\) by \((Lp)(Lu) \coloneqq p(u)\). Then \(\lang_{LU}(T \cP) = L(\lang_U(\cP))\). In particular, \(\lang(T\cP) = L(\lang(\cP))\).
\end{lemma}

\begin{proof}
If \(p \in \lang_U(\cP)\), say \(p = \cP[x,U]\), then by definition \(Lp = (L\cP)[Lx,LU] \in \lang_{LU}(L\cP)\). And any \(p' = (L\cP)[y,LU]\) can be expressed as \(p' = Lp\), taking \(x = L^{-1}(y)\). So \(p \mapsto Lp\) defines surjection \(\lang_U(\cP) \to \lang_{LU}(L\cP)\) (a bijection of course, with inverse given by \(p \mapsto L^{-1}p\)). Since \(L\cP \LIs T\cP\) (translating a pattern does not affect its language), we have \( \lang_{LU}(T\cP) = \lang_{LU}(L\cP) = L(\lang_U(\cP))\), as required. The last claim follows by taking the union over all \(U \cpt E\), which is equal to taking the union over all \(LU \cpt E\).
\end{proof}

\subsection{Examples of patterns}

\subsubsection{Tilings}\label{sec:patterns from tilings}
A \textbf{tile} is a pair \(t = (K,\ell)\) where \(\emptyset \neq K \cpt E\) and \(\ell \in \cL\) for a fixed set \(\cL\) of `labels' (or `colours'). Given a tile \(t = (K,\ell)\), its support is \(\mathrm{supp}(t) \coloneqq K\) and its translation by \(x \in E\) is \(t + x \coloneqq (K+x,\ell)\). A \textbf{tiling} \(T\) is a set of tiles covering \(E\), that is, \(E = \bigcup_{t \in T} \mathrm{supp}(t)\), where distinct tiles intersect on at most their boundaries. Usually, one also asks that the tiling is \emph{locally finite} i.e., each compact subset intersects at most finitely many tiles. However, this is not important here; non-locally finite tilings can even be FLC (in the sense of Definition \ref{def:FLC}).

There are a couple of way one could define an abstract pattern associated to a tiling. Firstly, define \(A \coloneqq \mathscr{L} \cup \{\partial\}\), where \(\partial\) is a new label not in \(A\). We then define \(\cP \in A^E\) by \(\cP[x] \coloneqq \ell\), if \(x\) belongs to the interior of a tile with label \(\ell\), and \(\cP[x] = \partial\) if \(x\) belongs to the boundary of a tile. If all tiles are connected and different translation classes of tiles have distinct labels, it is easy to see how the original tiling is essentially determined (and in a local way) from the abstract pattern. Otherwise, trivial modifications can be introduced to ensure this.

An alternative approach goes as follows. Given a tiling \(T\) and \(U \subseteq E\), let \(T \sqcap U\) denote the set of tiles whose supports have non-trivial intersection with \(U\). Given \(x \in E\), we denote
\begin{equation} \label{eq:pattern from tiling}
\cT[x] \coloneqq (T \sqcap \{x\}) - x = (T - x) \sqcap \{\mathbf{0}\},
\end{equation}
that is, we take all tiles touching the point \(x\) and recentre them by moving \(x\) to the origin. Taking \(B\) to be all such patches (`centred' at the origin), this defines a pattern \(\cT \in B^E\). Thus, \(\cT[x] = \cT[y]\) if and only if \(x\) and \(y\) are incident with tiles that are exactly equal, modulo translation from \(x\) to \(y\). Similarly, \(\cT[x,r] = \cT[y,r]\) if and only if the tiles within radius \(r\) of \(x\) are exactly equal to those within radius \(r\) of \(y\), modulo translation from \(x\) to \(y\). Again, note how the square brackets indicate how a (rather typical) translational equivalence relation is implicitly being considered in this notation.

Note that for the first construction, \(\#A < \infty\), provided \(\# \cL < \infty\). In contrast, we always have \(\#B = \infty\). Indeed, for \(\cT[x] = \cT[y]\) it is not enough that the patches of tiles touching \(x\) and \(y\) agree modulo translation, they need to agree modulo translation taking \(x\) to \(y\) (that is, \(x\) and \(y\) need to mark equivalent points in the patches). So the reader might wonder in which sense replacing the original tiling with an abstract pattern is faithful. In fact, given the mild conditions (or minor alterations otherwise) in the first approach, both lead to essentially equivalent patterns, in the sense that \(\cP \MLD \cT\), see Definition \ref{def:LD}. Crucially, whichever approach we take, associated pattern spaces (see Section \ref{sec:pattern spaces}) will be topologically conjugate to each other, and to tiling spaces defined, for instance, as in \cite{Sad08}.

With similar definitions, one can also easily define patterns associated to (locally finite) coverings of \(E\) by tiles (i.e., one may allow for overlapping tiles, as in Figure \ref{fig:gummelt}) or, by introducing an `empty' label, also allow tiles to not cover all of \(E\) (see Figure \ref{fig:several subs}).

\subsubsection{Coloured point sets}\label{sec:patterns from point sets}
Let \(\Lambda \subset E\) be \textbf{uniformly discrete} i.e., there exists some \(r > 0\) for which every \(r\)-ball in \(E\) intersects at most one point of \(E\). We call \(\Lambda\) a \textbf{point set}, and say \(\Lambda\) is a \textbf{Delone set} if, additionally, it is \textbf{relatively dense}, that is, there exists some \(R > 0\) so that every \(R\)-ball in \(E\) intersects at least one point of \(E\).

Given a set \(\cL\), we can also considered coloured (or `labelled') point sets and Delone sets, by giving a map \(\ell \colon \Lambda \to \cL\). Let \(A \coloneqq \cL \cup \{\emptyset\}\), where \(\emptyset\) is a new label. A coloured point set naturally defines a pattern \(\cP \in A^E\), by \(\cP[x] = \ell(x)\) if \(x \in \Lambda\), and \(\cP[x] = \emptyset\) otherwise. If \(\Lambda\) is uncoloured we can define \(\cP\) analogously (just choosing an arbitrary common label of all points in \(\Lambda\)).

Given \(x \in E\) and \(U \subseteq E\), note that \(\cP[x,U] \in A^U\) is essentially given by restricting the (coloured) point set to the subset \(U+x\), and translating by \(-x\) back over the origin. That is, it may be identified with the `cluster' \((\Lambda-x) \cap U\) (and similarly in the coloured case).

\subsection{Local derivation of patterns}

As well as defining an important relation between patterns, the notion of a local derivation \cite{BG13,BSJ91} will be key in a generalised notion of being substitutional.

\begin{definition}\label{def:LD}
Given a pattern \(\cP \in A^E\), a function \(f \colon E \to X\) (for any set \(X\)) is \textbf{pattern equivariant} or \textbf{locally defined} (\textbf{by \(\cP\)}) if there exists some \(c \geq 0\) for which, whenever \(\cP[x,c] = \cP[y,c]\) for \(x\), \(y \in E\), then \(f(x) = f(y)\). We say that \(\cQ \in B^E\) is \textbf{locally derivable} (\textbf{LD}) \textbf{from} \(\cP\), and write \(\cP \LD \cQ\), if the function \(f(x) \coloneqq \cQ[x]\) is locally defined by \(\cP\). That is, \(\cP \LD \cQ\) if there exists some \(c \geq 0\), called a \textbf{derivation radius}, for which
\[
\cP[x,c] = \cP[y,c] \text{ implies that } \cQ[x] = \cQ[y] .
\]
If \(\cP \LD \cQ\) and \(\cQ \LD \cP\) then \(\cP\) and \(\cQ\) are \textbf{mutually locally derivable} (\textbf{MLD}), written \(\cP \MLD \cQ\).
\end{definition}

\begin{remark}
Note that the `\(0\)' of \(\cQ[x,0] \equiv \cQ[x]\) (and \(\cQ[y] \equiv \cQ[y,0]\)) implicit in the definition is somewhat arbitrary, in the sense we could replace the requirement \(\cQ[x] = \cQ[y]\) with \(\cQ[x,r] = \cQ[y,r]\) for any fixed \(r \geq 0\), by increasing \(c\) by \(r\) using the following trivial lemma.
\end{remark}

\begin{lemma}\label{lem:increasing LD radius}
Let \(\cP \LD \cQ\) have derivation radius \(c\). Then, for all \(U \cpt E\), if \(\cP[x,B_c+U] = \cP[y,B_c+U]\) then \(\cQ[x,U] = \cQ[y,U]\). In particular, for all \(r \geq 0\), if \(\cP[x,c+r] = \cP[y,c+r]\) then \(\cQ[x,r] = \cQ[y,r]\).
\end{lemma}

\begin{proof}
If \(\cP[x,B_c+U] = \cP[y,B_c+U]\) then, for all \(u \in U\), we have \(\cP[x+u,B_c] = \cP[y+u,B_c]\) (by Lemma \ref{lem:pattern properties}, since \(B_c + u \subseteq B_c + U\)). Hence, by the LD property, \(\cQ[x+u] = \cQ[y+u]\) for all \(u \in U\) so that, by definition, \(\cQ[x,U] = \cQ[y,U]\), as required. The second claim follows, taking \(U = B_r\).
\end{proof}

\begin{lemma}\label{lem:LD preorder}
LD for patterns is a preorder. That is, for any pattern, \(\cP \LD \cP\) and for any patterns \(\cP_1\), \(\cP_2\) and \(\cP_3\), if \(\cP_1 \LD \cP_2 \LD \cP_3\) then \(\cP_1 \LD \cP_3\).
\end{lemma}

\begin{proof}
Clearly \(\cP \LD \cP\), with derivation radius \(c = 0\). If \(\cP_1 \LD \cP_2\) has derivation radius \(c_1\) and \(\cP_2 \LD \cP_3\) has derivation radius \(c_2\), then \(\cP_1 \LD \cP_3\) has derivation radius \(c_1 + c_2\), since if \(\cP_1[x,c_1+c_2] = \cP_1[y,c_1+c_2]\), then \(\cP_2[x,c_2] = \cP_2[y,c_2]\) (by Lemma \ref{lem:increasing LD radius}), and thus \(\cP_3[x] = \cP_3[y]\), as required.
\end{proof}

As already noted, a local derivation \(\cP \LD \cQ\) corresponds to the notion of a sliding block code between two bi-infinite words, where the derivation radius plays the role of the `window size' of the coding. To make this even more explicit, observe that \(\cP \LD \cQ\) is equivalent to there being some `window' \(U \cpt E\) and some `coding' function \(f \colon \lang_U(\cP) \to B\) for which \(\cQ[x] = f(\cP[x,U])\).

As will be familiar to readers in Aperiodic Order, the MLD relation is useful in passing between dynamically equivalent representations of patterns; for instance, it lets one work interchangeably between tilings and Delone sets (briefly, by puncturing and a Voronoi cell construction, see \cite{BG13}). Generally, LD is a non-symmetric relation:

\begin{example}
Let \(\cP\) be a pattern defined by a periodic tiling of translates of a \(2 \times 2\) squares and \(\cQ\) similarly but for the tiling given by cutting each square into four \(1 \times 1\) squares (associated patterns may be defined in a variety of MLD ways, as noted in Section \ref{sec:patterns from tilings}). Then \(\cP \LD \cQ\) but not vice versa: cutting the \(2 \times 2\) squares is a locally defined operation, whereas combining \(1 \times 1\) squares is not, as it requires information on parity of tile locations, which is not part of the information of a patch. MLD equivalent patterns necessarily have the same groups of periods (Lemma \ref{lem:periods under LI and LD}), which \(\cP\) and \(\cQ\) do not.
\end{example}

\subsection{Translational periods}

We may define the group of translational periods of a pattern just as we do for tilings and point sets:

\begin{definition}\label{def:periods}
For \(\cP \in A^E\) we let \(\cK = \cK_{\cP}\) denote the Abelian group of (translational) \textbf{periods} of \(\cP\), that is \(\cK \coloneqq \{x \in E \mid \cP = \cP+x\}\). We say that \(\cP\) is \textbf{non-periodic} if \(\cK = \{\mathbf{0}\}\), otherwise it is \textbf{periodic}.
\end{definition}

Local isomorphism and local derivation respect inclusion of periods:

\begin{lemma}\label{lem:periods under LI and LD}
If \(\cP \LI \cQ\) then \(\cK_{\cP} \geq \cK_{\cQ}\). Similarly, if \(\cP \LD \cQ\), then \(\cK_{\cP} \leq \cK_{\cQ}\). In particular, if \(\cP \LIs \cQ\) or \(\cP \MLD \cQ\) then \(\cK_{\cP} = \cK_{\cQ}\).
\end{lemma}

\begin{proof}
Let \(g \in \cK_{\cQ}\). Clearly, this is equivalent to \(\cQ[x] = \cQ[x+g]\) for all \(x \in E\). Let \(U \coloneqq \{\mathbf{0},g\}\). Take an arbitrary \(x \in E\) and its corresponding two-point patch \(p = \cP[x,U]\). If \(\cP \LI \cQ\), we must also have \(p \triangleleft \cQ\), that is, \(p = \cQ[y,U]\) for some \(y \in E\). Thus, \(\cP[x] = \cQ[y] = \cQ[y+g] = \cP[x+g]\). Since \(x \in E\) was arbitrary, \(g \in \cK_{\cP}\), as required.

Now suppose that \(\cP \LD \cQ\), with derivation radius \(c\), and let \(g \in \cK_{\cP}\). Let \(x \in E\) be arbitrary. Since \(g \in \cK_{\cP}\), we have that \(\cP[x,c] = \cP[x+g,c]\), thus \(\cQ[x] = \cQ[x+g]\). Since \(x \in E\) was arbitrary, it follows that \(g \in \cK_{\cQ}\).
\end{proof}

Unlike in the case of tilings and Delone set, groups of periods need not be discrete here. For instance, consider an `empty' pattern, one completely devoid of any distinguishing features, given by \(\cP[x] = a\) for constant \(a \in A\). Then \(\cK_{\cP} = E\). A return discreteness property of \(\cP\) ensures discreteness of \(\cK_{\cP}\), see Definition \ref{def:return discrete}, although we will not generally require this in our results. We will, however, at least need our groups of periods to be topologically closed, which will ensure that the index \([\cK : L\cK]\) of the inflated subgroup of periods is always finite. But this is not a strong requirement, it will follow from weaker separation properties required later. The fact needed for the following definition, on closed Euclidean subgroups, is well-known:

\begin{definition}\label{def:discrete component}
Let \(G \leqslant E\) be a closed subgroup of \(E \cong \R^d\). Then \(G \cong V \oplus K\), where \(V \leqslant G\) is the vector subspace given by the connected component of \(G\) containing the origin, and \(K = G/V \cong \Z^k\) for some \(0 \leq k \leq d\). We say that \(G\) has \textbf{trivial discrete component} if \(K\) is the trivial group, that is, when \(G\) is connected.
\end{definition}

Of course, if \(G\) is discrete then \(V\) above is trivial, so \(G\) has trivial discrete component if and only if \(G = \{\mathbf{0}\}\) is itself trivial. When we allow patterns to have non-discrete groups of periods \(\cK\), having trivial discrete component is typically the suitable replacement of \(\cK\) being trivial (i.e., non-periodicity) for our recognisability results.

\begin{definition}\label{def:discretely periodic}
A pattern \(\cP\) is \textbf{discretely periodic} if \(\cK_{\cP}\) has non-trivial discrete component. Otherwise, we call \(\cP\) \textbf{discretely non-periodic}.
\end{definition}

\begin{lemma}\label{lem:periods of distorted patterns}
For an affine automorphism \(T(x) = L(x) + z\) i.e., with \(L\) linear, \(\cK_{T\cP} = L\cK_{\cP}\).
\end{lemma}

\begin{proof}
We have \(T(x+g) = T(x) + L(g)\). Thus, by the definition of the distorted pattern, \(\cP[x] = \cP[x + g]\) is equivalent to \((T\cP)[T(x)] = (T\cP)[T(x) + L(g)]\). Now, \(\cP[x] = \cP[x+g]\) for arbitrary \(x \in E\) is equivalent to \(g \in \cK_{\cP}\), and similarly \((T\cP)[T(x)] = (T\cP)[T(x) + L(g)]\) for arbitrary \(T(x) \in E\) is equivalent to \(L(g) \in \cK_{T\cP}\). It follows that \(\cK_{T\cP} = L\cK_{\cP}\), as required.
\end{proof}

\begin{definition}\label{def:expansive}
A linear map \(L \colon E \to E\) is \textbf{expansive} if all eigenvalues of \(L\) are strictly larger than \(1\) in modulus.
\end{definition}

It can be shown that \(L\) being expansive is equivalent to there existing some norm on \(E\) and \(\lambda > 1\) with \(\|L(x)\| \geq \lambda \|x\|\) for all \(x \in E\).

\begin{lemma}\label{lem:index of closed subgroups}
Let \(G\) be a closed subgroup of \(E \cong \R^d\) and \(L \colon E \to E\) be an expansion with \(LG \leqslant G\). Then the index \(I \coloneqq [G : LG] < \infty\) and \(I = 1\) if and only if \(G\) has trivial discrete component.
\end{lemma}

\begin{proof}
Let \(V \leqslant G\) be the connected component of the origin in \(G\). Since \(LG \cong G\), the connected component of \(LG\) containing the origin has dimensional equal to that of \(V\), and hence is \(V\) since \(LG \leqslant G\). So \(G\) and \(LG\) have discrete components of equal rank. Then \(I = [G : LG] = [G/V : LG/V]\) is the index of a free Abelian group inside another, of equal rank, so is finite. If \(G\) has trivial discrete component then \(G = V = LG\), so \(I = 1\). Conversely, suppose that \(G\) has non-trivial discrete component, that is \(G^+ \coloneqq G \setminus V \neq \emptyset\). Take any norm on \(E\) and \(\lambda > 1\) as described above. We have that \(G^+ \coloneqq G \setminus V\) is closed, so \(\mu \coloneqq \inf\{\|g\| \mid g \in G^+\} > 0\) is attained by some \(g \in G^+\). Suppose that \(Lh = g\) for some \(h \in G\). If \(h \in V\) then \(g = Lh \in V\), a contradiction. Otherwise, \(h \in G^+\), thus \(\|g\| = \|Lh\| \geq \lambda\|\mu\| > \|\mu\|\), a contradiction. Hence, \(g \notin LG\), as required.
\end{proof}

\subsection{Finite local complexity and repetitivity of patterns}

Our main results will not require minimality (or repetitivity) assumptions, although we will require finite local complexity, with respect to translations. In our generalised setup, this is defined (similarly to in \cite{BG13}) as follows:

\begin{definition}\label{def:FLC}
A pattern \(\cP\) has (\textbf{translational}) \textbf{finite local complexity} (or is \textbf{FLC}) if, for all \(U \cpt E\), there exists some \(K(U) \cpt E\) satisfying the following: for all \(p \in \lang_U(\cP)\), we have \(p = \cP[v,U]\) for some \(v \in K(U)\).
\end{definition}

This definition simply says that all patches of a fixed shape that appear somewhere in the pattern, in fact appear within a bounded distance of the origin. It is not hard to show that if \(\cP\) is FLC and \(\cQ \LI \cP\) or \(\cP \LD \cQ\) then \(\cQ\) is also FLC. We will later re-establish, in our generalised setting, the familiar result  \cite{Rob04,Rud89} that FLC corresponds to compactness of the associated translational hull with the local topology.

\begin{example}
Suppose that \(\cP\) is defined by a locally finite tiling, in one of the ways explained in Section \ref{sec:patterns from tilings}. If there are only finitely many tile types that may only only meet along their boundaries in finitely many different ways, modulo translation, then only finitely many patches up to any given diameter may be constructed. Thus, they must all occur within some bounded distance of the origin, so \(\cP\) is FLC in the above sense. It is clear that the converse holds too. Thus, the above notion of FLC agrees with the standard notion (e.g., see \cite{Sad08}) for locally finite tilings. Similar comments apply to (coloured) Delone sets.
\end{example}

\begin{example}
For \(n \in \N\), define \(I_n \coloneqq [1 - 2^{n-1},1 - 2^{n-2}]\) and the non-locally finite tiling in \(E = \R\) given by \(\{I_n + k \mid n \in \N, k \in \Z\}\). Let \(\cP\) be the associated pattern. It is not hard to see that \(\cK_{\cP} = \Z\). Indeed, \(\cP\) is MLD to the pattern defined by a periodic tiling \(\{ [0,1] + k \mid k \in \Z \}\) of translates of the single tile \([0,1]\). In particular, according to Definition \ref{def:FLC}, \(\cP\) is FLC. This is appropriate, given that it is essentially just a jazzy redecoration of a very basic periodic tiling of unit intervals. Moreover, as we will see, FLC corresponds compactness of the associated hull using this definition, so this seems a more useful convention here than any demand of there only being finitely many tile types, modulo translation, in this non-locally finite situation.
\end{example}

Another familiar property from Aperiodic Order is repetitivity of a pattern, meaning that every patch that appears somewhere in fact appears within a bounded neighbourhood of anywhere. This may be defined for general patterns as follows:

\begin{definition}\label{def:repetitive}
A pattern \(\cP\) is \textbf{repetitive} if, for all \(U \cpt E\), there exists some \(R(U) \cpt E\) satisfying the following: for all \(p \in \lang_U(\cP)\) and \(x \in E\), we have \(p = \cP[x+v,U]\) for some \(v \in R(U)\).
\end{definition}

Repetitivity corresponds to minimality of the associated translational hull, to be introduced in the next section. Note that a repetitive pattern is always FLC, as we may take \(K(U) = R(U)\). In the presence of FLC (although not without, for then there are counter-examples), repetitivity can also be expressed in terms of relative density of repetitions of \emph{individual} patches, rather than all \(U\)-patches for each compact \(U\).

\begin{proposition}\label{prop:FLC rep}
Suppose that \(\cP\) is FLC. It is repetitive if and only if we have the following: for all \(p \in \lang(\cP)\), there exists some \(R = R(p) \cpt E\) for which, for all \(x \in E\), there exists some \(v \in R\) with \(p = \cP[x+v,U]\).
\end{proposition}

\begin{proof}
Of course, in the above statement the shape \(U\) is determined by \(p\). This property for individual patch recurrence follows from full repetitivity, by taking \(R(p) = R(U)\). Conversely, suppose that \(\cP\) is FLC and satisfies this property. Let \(U \cpt E\) be arbitrary and \(K = K(U) \cpt E\) be as in Definition \ref{def:FLC} from FLC. Defining \(q \coloneqq \cP[\mathbf{0},K + U]\), we claim that \(R = R(U) \coloneqq K+R(q)\) satisfies the requirement for repetitivity in Definition \ref{def:repetitive}.

Indeed, let \(p \in \lang_U(\cP)\) and \(x \in E\) be arbitrary. Then \(p = \cP[y,U]\) for some \(y \in K\) and \(q = \cP[\mathbf{0},K+U] = \cP[x+v,K+U]\) for some \(v \in R\). Shifting by \(y\) (Lemma \ref{lem:pattern properties}, using \(y+U \subseteq K+U\)), we have \(p = \cP[\mathbf{0} + y,U] = \cP[x+(v+y),U]\) where \(v+y \in K+R(q) = R\), as required.
\end{proof}

\subsection{Periods from short return vectors}
The following gives a generalised notion of a repetitivity function, which allows one to ask if a given pattern has \(U\)-patches appearing sufficiently densely relative to \(U\):

\begin{definition}
Let \(\mathcal{Z} = \{U \mid U \cpt E\}\) and \(\rho \colon D \to \mathcal{Z}\) with some domain \(D \subseteq \mathcal{Z}\). We say that \(\cP\) is \textbf{\(\rho\)-repetitive} if, for all \(U \in D\), \(p \in \lang_U(\cP)\) and \(x \in E\), there exists some \(v \in \rho(U)\) with \(\cP[x+v,U] = p\). That is, given \(U \in D\), if a \(U\)-patch occurs somewhere, it appears within \(\rho(U)\) of anywhere.
\end{definition}

Note that, by Definition \ref{def:repetitive}, \(\cP\) is repetitive precisely when it is \(\rho\)-repetitive for some function \(\rho\) with domain \(D = \mathcal{Z}\). For flexibility, we allow for a smaller domain in the above definition. Clearly, \(\cP\) is repetitive if and only if it is \(\rho\)-repetitive for some \(\rho\) whose domain is final, in the sense that every \(U \cpt E\) satisfies \(U \subseteq V\) for some \(V \in D\), for example, \(D = \{B_r \mid r \geq 0\}\) or \(D = \{V^k \mid k \in \N_0\}\) when working with an \(L\)-sub pattern (with \(V^k\) given in the later Notation \ref{not:balls}).

\begin{example} \label{exp: repetitivity function}
Take any function \(f \colon \R_{\geq 0} \to \R_{\geq 0}\) and define  the corresponding repetitivity function \(\rho \colon \{B_r \mid r \geq 0\} \to \mathcal{Z}\) by \(\rho(B_r) = B_{f(r)}\). Then \(\cP\) is \(\rho\)-repetitive if (the centre of) every \(r\)-patch appears within radius \(f(r)\) of anywhere, so this is a standard notion of a repetitivity function in this case. A pattern is repetitive if and only if it is \(\rho\)-repetitive for a function \(\rho\) of this form.
\end{example}

A central observation in Solomyak's proof of unique composition is that patches of repetitive tilings that repeat too closely, relative to their size and the repetitivity function, are necessarily displaced by a period. As noted in \cite{Sol98}, this is itself a higher dimensional analogue of a main insight in Moss\'{e}'s result \cite{Mos92,Mos96} in the symbolic setting, of being \(N\)-power free for sufficiently large \(N\). The below is a quantitative statement of these in our setting:

\begin{lemma}\label{lem:repetitive and close repeat => period}
Suppose that \(\cP\) is \(\rho\)-repetitive. Take any \(U \in \mathrm{dom}(\rho)\) with \(\mathbf{0} \in U\). Suppose there exists some \(x \in E\) and \(u \in U\) for which \(\cP[x,\rho(U)] = \cP[x+u,\rho(U)]\). Then \(u \in \cK_{\cP}\).
\end{lemma}

\begin{proof}
Take any \(y \in E\) and denote \(V \coloneqq \rho(U)\). By the definition of \(\rho\)-repetitivity, there exists \(v \in V\) with \(\cP[x+v,U] = \cP[y,U]\). In particular, since \(\mathbf{0} \in U\), we have \(\cP[x+v] = \cP[y]\). Moreover, since we assume \(\cP[x,V] = \cP[x+u,V]\), and \(v \in V\), we have \(\cP[x,\{v\}] = \cP[x+u,\{v\}]\), that is, \(\cP[x+v] = \cP[x+u+v]\). Combining these, \(\cP[x+u+v] = \cP[x+v] = \cP[y]\). From \(\cP[y,U] = \cP[x+v,U]\) and \(u \in U\), we similarly have \(\cP[y+u] = \cP[x+v+u]\) and hence \(\cP[y] = \cP[y+u]\). Since \(y\) was arbitrary, \(u \in \cK_{\cP}\).
\end{proof}

\begin{example}
Consider the function \(\rho\) from Example \ref{exp: repetitivity function}, coming from the standard repetitivity function \(f(r)\) of a repetitive pattern \(\cP\). By the above, if we ever have patches about \(x\) and \(x+u\) agreeing to radius \(f(r)\), for \(\|u\| \leq r\), then \(u \in \cK\).
\end{example}

The above lemma is of no use for non-repetitive patterns. However, we observe that whilst some large patches may appear with relatively short displacements in a non-repetitive, non-periodic pattern, certain patches (those containing all patches of some smaller size) cannot:

\begin{lemma}\label{lem:close repeat => period}
Let \(U \cpt E\) with \(\mathbf{0} \in U\). Suppose there exists \(x \in E\) and \(V \cpt E\) satisfying the following: every \(U\)-patch occurs with centre within \(V\) of a point \(x\). That is, suppose that for all \(p \in \lang_U(\cP)\) there exists some \(v \in V\) with \(\cP[x+v,U] = p\). Then, if \(\cP[x,V] = \cP[x+u,V]\) for some \(u \in U\), we have \(u \in \cK_{\cP}\).
\end{lemma}

\begin{proof}
The proof is identical to that of Lemma \ref{lem:repetitive and close repeat => period}, except now the existence of the \(v \in V\) with \(\cP[x+v,U] = p\) is assumed at a particular \(x \in E\) rather than being derived from repetitivity at all \(x \in E\).
\end{proof}

\section{Pattern spaces}
\label{sec:pattern spaces}

In this section, we define the notion of a (translational) pattern space, analogous to the idea of a \emph{tiling space} with the local topology (see \cite{Sad08, BG13}). Rather than defining them in terms of more abstract topological closure properties, we use the following more direct approach, in terms of the patterns themselves:

\begin{definition}\label{def:pattern space}
Let \(\Omega \subseteq A^E\) be a non-empty set of patterns (so all patterns in \(\Omega\) are over a common ambient space and labelling set). For \(U \cpt E\) and \(U\)-patch \(p \in A^U\), we write \(p \triangleleft \Omega\) if there exists some \(\cP \in \Omega\) with \(p \triangleleft \cP\). The set of all \(U\)-patches of \(\Omega\) is denoted
\[
\lang_U(\Omega) \coloneqq \{p \in A^U \mid p \triangleleft \Omega\} = \bigcup_{\cP \in \Omega} \lang_U(\cP) 
\]
and the set of all patches (of all shapes) in \(\Omega\), its \textbf{language}, is denoted
\[
\lang(\Omega) \coloneqq \bigcup_U \lang_U(\Omega) = \bigcup_{\cP \in \Omega} \lang(\cP),
\]
where the first union is over all \(U \cpt E\). Write \(\cP \LIin \Omega\) if \(\lang(\cP) \subseteq \lang(\Omega)\), that is, if \(p \triangleleft \cP\) implies that \(p \triangleleft \Omega\). We call \(\Omega\) a \textbf{pattern space} if, for all \(\cP \in A^E\), if \(\cP \LIin \Omega\) then \(\cP \in \Omega\).
\end{definition}

Of course, if \(\cP \in \Omega\) then \(\lang(\cP) \subseteq \lang(\Omega)\), so \(\cP \in \Omega\) implies that \(\cP \LIin \Omega\). So being a pattern space means that \(\cP \in \Omega\) if and only if \(\cP \LIin \Omega\). In particular, a pattern space is determined by its language. This definition should remind the reader of that of a subshift \(X \subseteq A^\Z\) where, if a bi-infinite word \(w\) has language contained in \(X\), then \(w \in X\). Equivalently, the standard definition of being a shift space is that \(X\) is topologically closed and shift invariant in the full shift \(A^\Z\), and we will see an analogous result for pattern spaces in Corollary \ref{cor:hulls are orbit closures}. In this geometric setting, though, a pattern space should be considered as a (general dimension analogue of the) \emph{suspension} of a subshift. In particular, it will have a continuous action by \(E \cong \R^d\).

The relation \(\cP \LIin \Omega\) may be interpreted topologically, see Lemma \ref{lem:LIin is closure inclusion}. To be a space, of course we also need to define a geometry on \(\Omega\). Before explaining how, we give a few basic results and illustrative examples (and a non-example) of pattern spaces below.

\begin{example}\label{ex:full pattern space}
As always, assume that \(A \neq \emptyset\). Then taking \(\Omega = A^E\) (that is, every possible pattern), \(\Omega\) is obviously a pattern space, the `full pattern space'. However, it is somewhat pathological compared to something like the full shift: when \(\# A > 1\), it will be non-Hausdorff.
\end{example}

\begin{lemma}\label{lem:pattern spaces closed under LI and translation}
If \(\Omega\) is a pattern space, \(\cP \in \Omega\) and \(\cQ \LI \cP\), then \(\cQ \in \Omega\). In particular, \(\Omega\) is closed under translations by \(E\).
\end{lemma}

\begin{proof}
By definition, \(\cQ \LI \cP\) if and only if \(\lang(\cQ) \subseteq \lang(\cP)\), and \(\cP \in \Omega\) if and only if \(\lang(\cP) \subseteq \lang(\Omega)\). So \(\lang(\cQ) \subseteq \lang(\Omega)\) i.e., \(\cQ \LIin \Omega\) and hence \(\cQ \in \Omega\), as required. For any pattern \(\cP\) and \(x \in E\), we have \(\cP \LIs \cP+x\), so it follows from the first claim that \(\Omega\) is closed under translations.
\end{proof}

\begin{definition}\label{def:hull}
Let \(\cP \in A^E\). We define its (\textbf{translational}) \textbf{hull} to be the pattern space
\[
\Omega_{\cP} \coloneqq \{\cQ \in A^E \mid \cQ \LI \cP \}.
\]
\end{definition}

\begin{lemma}
The hull of a pattern is a pattern space.
\end{lemma}

\begin{proof}
We have \(\lang(\Omega) \coloneqq \bigcup_{\cQ \in \Omega} \lang(\cQ) = \lang(\cP)\) since, for each \(\cQ \in \Omega\), by definition of the hull, \(\cQ \LI \cP\), that is, \(\lang(\cQ) \subseteq \lang(\cP)\) (and we get the reverse containment since \(\cP \in \Omega\), as \(\cP \LI \cP\)). If \(\cQ \LIin \Omega\) then, by definition, \(\lang(\cQ) \subseteq \lang(\Omega) = \lang(\cP)\), so \(\cQ \LI \cP\), hence \(\cQ \in \Omega\), as required.
\end{proof}

Clearly, \(\cP \LI \cQ\) if and only if \(\Omega_{\cP} \subseteq \Omega_{\cQ}\), since these are equivalent to \(\lang(\cP) \subseteq \lang(\cQ)\), giving another justification of this notation; along similar lines, we will see how an LD \(\cP \LD \cQ\) naturally induces a so-called LD map \(\Omega_{\cP} \LDmap \Omega_{\cQ}\) of hulls, in Proposition \ref{prop:induced LDs}. For examples that are not hulls, just being closed under local indistinguishability is generally \emph{not} sufficient to be a pattern space:

\begin{example}\label{exp:pattern space closed under LI but not LIin}
Consider the set \(X\) of patterns defined by tilings of \(E = \R^1\), of unit interval tiles all labelled \(a\), except for at most a finite string of \(n \in \N_0\) consecutive \(b\) tiles.  Then \(X\) is closed under local indistinguishability: whenever \(\cQ \LI \cP\) for some \(\cP \in X\), we also have \(\cQ \in X\) (indeed, \(\cQ\) must be a translate of \(\cP\) in this case, or an all \(a\) tiling, so \(\cQ \in X\)). However, \(X\) is \emph{not} closed under `collective' local indistinguishability, that is, \(\cP \LIin X\) does not imply \(\cP \in X\). Indeed, consider the pattern \(\cQ\) associated to an all \(b\) tiling. Then \(\cQ \LIin X\) but \(\cQ \notin X\). The space \(X\) is not complete with the local uniformity we will introduce (so allowing \(X\) to be a `pattern space' would invalidate Proposition \ref{prop:hull complete}). Including the `limit point' \(\cQ\) defines a (complete) pattern space \(\Omega = X \cup \{\cQ\}\).
\end{example}

\subsection{The local topology of pattern spaces}

Throughout, we use the `local topology' which, loosely speaking, considers two patterns to be close when they agree to a large radius about the origin, up to small translations of each. Although in the context of tiling spaces this is typically defined by a metric, we prefer to define it with a uniformity, since none of the standard arguments really need an explicit (necessarily somewhat arbitrary) distance defined, yet these arguments will retain a similar intuition. It is useful to define a uniformity, rather than merely a topology, since we will sometimes want to speak of completeness, Cauchy sequences and uniform continuity.

Firstly we will define the local uniformity (and associated local topology) and then explain how this geometrises pattern spaces for readers unfamiliar with uniformities (also see, for instance, \cite{Wil04} for an introduction to uniform spaces). Recall from Notation \ref{not:subsets} that \(B \nhd E\) means that \(B\) is a neighbourhood of \(\mathbf{0}\) in \(E\).

\begin{definition}\label{def:local uniformity}
Let \(\Omega \subseteq A^E\). We define \(\mathscr{U}' \subseteq 2^{\Omega \times \Omega}\), a base of the local uniformity, to be the following collection of subsets:
\[
\mathscr{U}' \coloneqq \{\Phi(B,U) \mid B \nhd E, \ U \cpt E\}, \text{ where}
\]
\[
\Phi(B,U) \coloneqq \{(\cP,\cQ) \in \Omega \times \Omega \mid \cP[x,U] = \cQ[y,U] \text{ for some } x,y \in B\} .
\]
The \textbf{local uniformity} \(\mathscr{U} \subseteq 2^{\Omega \times \Omega}\) on \(\Omega\) is defined from this base, as usual, by letting \(e \in \mathscr{U}\) whenever \(e \supseteq e'\) for some \(e' \in \mathscr{U}'\).
\end{definition}

The idea is that each \(e \in \mathscr{U}\) provides a uniform notion of `closeness' across \(\Omega\): we say that \(\cP\) and \(\cQ\) are \textbf{\(e\)-close} if \((\cP,\cQ) \in e\). Note that \(\cP[x,U] = \cQ[y,U]\) (as appearing in the definition of \(\Phi(B,U)\)) is equivalent to \((\cP-x)[\mathbf{0},U] = (\cQ-y)[\mathbf{0},U]\) i.e., the translations \(\cP-x\) and \(\cQ-y\) agree precisely on \(U\). We may consider \(\cP\) and \(\cQ\) to be particularly close if they are \(e\)-close for \(e = \Phi(B,U)\) with very small \(B\) and \(U\) containing a large ball centred at the origin. That is, up to small translations (in the small set \(B\)), they agree to a large distance about the origin (on the large set \(U\)). This is analogous to how one defines the standard tiling metric. We may still define the `ball' of points \(e\)-close to some \(\cQ \in \Omega\), by defining, for \(e \in \mathscr{U}\),
\[
B(\cP,e) \coloneqq \{\cQ \in \Omega \mid (\cP,\cQ) \in e\}  .
\]
This is often denoted by \(e[\cP]\) in the literature; we use the above ball notation as this will be more familiar to readers used to tiling metrics, and it suggests the correct analogy: given \(e \in \mathscr{U}\), there is always some \(e' = \Phi(B_\epsilon, B_r) \in \mathscr{U}'\) with \(e' \subseteq e\), and thus whenever \(\cQ\) is a pattern that agrees with \(\cP\) to radius \(r\), up to shifting \(\cP\) and \(\cQ\) by translations of norm at most \(\epsilon\), then we have \(\cQ \in B(\cP,e)\). Conversely, if \(\cQ\) cannot be made to agree on a sufficiently large ball, after sufficiently small translations (each relative to \(e\)), then \(\cQ \notin B(\cP,e)\).

Recall that a collection \(\mathscr{U}'\) of subsets \(e \subseteq \Omega \times \Omega\) is a \textbf{base of a uniformity} on \(\Omega\) when:
\begin{itemize}
	\item[(a)] if \(e \in \mathscr{U}'\) then \(\Delta \subseteq e\), where \(\Delta = \{(\cP,\cP) \mid \cP \in \Omega\}\);
	\item[(b)] if \(e\), \(e' \in \mathscr{U}'\) then \(f \subseteq e \cap e'\) for some \(f \in \mathscr{U}'\);
	\item[(c)] if \(e \in \mathscr{U}'\) then \(f \circ f \subseteq e\) for some \(f \in \mathscr{U}'\), where for subsets \(u\), \(v \subseteq X \times X\) (i.e., `relations' \(u\) and \(v\) on \(X\)), we denote \(u \circ v \coloneqq \{(x,z) \in X \times X \mid (x,y) \in u, \ (y,z) \in v\}\);
	\item[(d)] if \(e \in \mathscr{U}'\) then \(f^{-1} \subseteq e\) for some \(f \in \mathscr{U}'\), where for a relation \(u \subseteq X \times X\) we denote \(u^{-1} \coloneqq \{(y,x) \mid (x,y) \in u\}\).
\end{itemize}
Such a base defines a uniformity \(\mathscr{U}\) (whose elements are called `entourages') by letting \(e \in \mathscr{U}\) if \(e \supseteq e'\) for some \(e' \in \mathscr{U}'\), see \cite[Definition 35.2]{Wil04} for further details. The local uniformity defines a topology, the \textbf{local topology} on \(\Omega\), as usual. Here, the easiest approach is the definition of a topology by a neighbourhood system: a set \(X \subseteq \Omega\) is a neighbourhood of \(\cP \in \Omega\) if \(B(\cP,e) \subseteq X\) for some \(e \in \mathscr{U}\) (equivalently, some \(e = \Phi(B,U) \in \mathscr{U}'\)). So a subset \(X \subseteq \Omega\) is open if it is a neighbourhood of all of its points, and its closure is \(\overline{X} = \{\cP \in \Omega \mid B(\cP,e) \cap X \neq \emptyset \text{ for all }e \in \mathscr{U}\}\).

\begin{lemma}\label{lem:uniformity well-defined}
The set \(\mathscr{U}'\) is a well-defined uniformity base for any \(\Omega \subseteq A^E\).
\end{lemma}

\begin{proof}
Clearly (a) holds, since for all \(\Phi(B,U) \in \mathscr{U}'\) and \(\cP \in \Omega\) we have \(\cP[\mathbf{0},U] = \cP[\mathbf{0},U]\) and thus \((\cP,\cP) \in \Phi(B,U)\), since \(\mathbf{0} \in B\). Next, let \(e = \Phi(B,U)\), \(e' = \Phi(B',U') \in \mathscr{U}'\). Define \(K \coloneqq B \cap B'\), \(V \coloneqq U \cup U'\) and \(f \coloneqq \Phi(K,V) \in \mathscr{U}'\). If \((\cP,\cQ) \in f\) then \(\cP[x,V] = \cQ[y,V]\) for \(x\), \(y \in K\). Since \(x\), \(y \in K \subseteq B\), \(B'\), and \(U\), \(U' \subseteq V\), we have \((\cP,\cQ) \in e \cap e'\), so \(e \cap e' \supseteq f\), proving (b).

Given \(e = \Phi(B,U) \in \mathscr{U}'\), take any \(B' \cptn E\) (see Notation \ref{not:subsets}) satisfying \(B' - B' \subseteq B\) and let \(U' = U-B'\), which is also compact. Suppose that \((\cP,\cQ)\), \((\cQ,\cR) \in f \coloneqq \Phi(B',U')\). Thus, there exist \(x_1\), \(y_1\), \(x_2\), \(y_2 \in B'\) for which
\[
\cP[x_1,U-B'] = \cQ[y_1,U-B']\ \text{ and } \cQ[x_2,U-B'] = \cR[y_2,U-B'] .
\]
Shifting the first equation by \(-y_1\) and the second by \(-x_2\) (see Lemma \ref{lem:pattern properties} (3)), we have 
\[
\cP[x_1-y_1,U] = \cQ[\mathbf{0},U] = \cR[y_2-x_2,U] .
\]
Since \(x_1-y_1\), \(y_2-x_2 \in B'-B' \subseteq B\), we have \((\cP,\cR) \in \Phi(B,U) = e\) and thus \(f \circ f \subseteq e\), showing (c). Finally, by definition each \(e = \Phi(B,U)\) is already itself symmetric i.e., \((\cP,\cQ) \in e\) if and only if \((\cQ,\cP) \in e\), confirming (d). 
\end{proof}

The uniformity \(\mathscr{U}\) has a countable base, for example by the entourages \(e_n = \Phi(B_{1/n},B_n)\), so a pattern space \(\Omega\) is (pseudo)metrisable, see \cite[Theorem 38.3]{Wil04} (this can also be shown directly, by construction, as one does for the tiling metric), and metrisable when \(\Omega\) is Hausdorff (of which, see Section \ref{sec:separation properties} later). Thus, concepts including closures, continuity and compactness can be replaced with their sequential versions. Pattern spaces, in almost all cases of imaginable interest (namely when \(\Omega\) is a countable union of hulls \(\Omega_{\cP}\) for \(\cP \in \Omega\)), are second-countable; the proof is left as an exercise, as it is not used elsewhere.

The following shows that the `LI-in' relation may be defined topologically:

\begin{lemma}\label{lem:LIin is closure inclusion}
Suppose that \(\Omega \subseteq A^E\), where \(\Omega\) is not necessarily a pattern space. Then \(\cP \LIin \Omega\) if and only if \(\cP \in \overline{\Omega + E}\).
\end{lemma}

\begin{proof}
Let \(\cP \in \overline{\Omega+E}\) and \(p \in \lang(\cP)\), say \(p = \cP[x,U]\). Define \(e \coloneqq \Phi(B_1,(U+x)+B_1) \in \mathscr{U}\). As \(\cP\) is in the closure, there exists \(\cQ \in (\Omega+E) \cap B(\cP,e)\). Since \((\cP,\cQ) \in e\), we have \(\cP[b,(U+x)+B_1] = \cQ[b',(U+x)+B_1]\) for some \(b\), \(b' \in B_1 = B(\mathbf{0},1)\). By Lemma \ref{lem:pattern properties} (2), \(\cP[x+b,U+B_1] = \cQ[x+b',U+B_1]\) and thus, shifting by \(-b \in B_1\), we have \(p = \cQ[x+(b'-b),U]\) so \(p \triangleleft \cQ\). Since \(\cQ \in \Omega + E\), we have \(\cQ+y \in \Omega\) for some \(y \in E\), and of course \(\lang(\cQ) = \lang(\cQ+y)\), hence \(p \triangleleft \cQ+y\) too, so \(p \triangleleft \Omega\). Since the patch \(p \triangleleft \cP\) was arbitrary, we have \(\cP \LIin \Omega\), as required.

For the converse, assume \(\cP \LIin \Omega\) and take an arbitrary \(e = \Phi(B,U) \in \mathscr{U}'\). From \(\cP \LIin \Omega\), for \(p \coloneqq \cP[\mathbf{0},U] \triangleleft \cP\) we have \(p = \cQ[y,U] = (\cQ-y)[\mathbf{0},U]\) for some \(\cQ \in \Omega\) and \(y \in E\). Since \(\mathbf{0} \in B\), this shows that \(\cQ-y \in B(\cP,e) \cap (\Omega-y) \subseteq B(\cP,e) \cap (\Omega + E)\). So the latter is non-empty, and since \(e \in \mathscr{U}'\) was arbitrary we have that \(\cP \in \overline{\Omega+E}\), as required.
\end{proof}

Note that, if \(\Omega \subseteq A^E\) is translation invariant (that is, \(\Omega = \Omega+E\)), then the above says that \(\cP \LIin \Omega\) if and only if \(\cP \in \overline{\Omega}\). This shows that, as usual, hulls are orbit closures, where the closure may be considered as within the non-Hausdorff full pattern space \(A^E\), or if preferred (especially to be Hausdorff), the orbit closure within any sub-pattern space of this:

\begin{corollary}\label{cor:hulls are orbit closures}
A non-empty subspace \(\Omega \subseteq A^E\) is a pattern space if and only if it is topologically closed in \(A^E\) and translation invariant. Moreover, for any pattern space \(\Omega\) and \(\cP\), \(\cQ \in \Omega\), we have \(\cQ \in \overline{\cP+E}\) if and only if \(\cQ \LI \cP\). In particular, we have an equality of uniform spaces \(\Omega_{\cP} = \overline{\cP+E}\).
\end{corollary}

\begin{proof}
All pattern spaces are translation invariant (see Lemma \ref{lem:pattern spaces closed under LI and translation}). Supposing that \(\Omega \subseteq A^E\) is translation invariant, by definition \(\Omega\) is a pattern space if and only if \(\cP \in \Omega\) is equivalent to \(\cP \LIin \Omega\) which, by the above lemma, is equivalent to \(\Omega = \overline{\Omega}\) i.e., \(\Omega\) is closed in \(A^E\).

Since \(\cP+E \subseteq \Omega\) is translation invariant, again by the previous lemma we have that \(\cQ \in \overline{\cP+E}\) if and only if \(\cQ \LIin \cP+E\) i.e., \(\lang(\cQ) \subseteq \lang(\cP+E) = \lang(\cP)\), that is, \(\cQ \LI \cP\). Since, by definition of the hull, we also have \(\cQ \in \Omega_{\cP}\) if and only if \(\cQ \LI \cP\), we obtain \(\Omega_{\cP} = \overline{\cP+E}\). Note that, since pattern spaces are closed, taking the closure \(\overline{\Omega'}\) of a subset \(\Omega' \subseteq A^E\) does not depend on if we take the closure within \(A^E\) or within any other sub-pattern space \(\Omega' \subseteq \Omega\).
\end{proof}

Next we show that translation acts continuously on pattern spaces, making them dynamical systems. Similarly, dilation by an affine map defines a (uniform) homeomorphism of pattern spaces:

\begin{proposition}\label{prop:pattern spaces are DS}
The action of \(E\) on a pattern space \(\Omega\), given by \(x \cdot \cQ \coloneqq \cQ + x\), defines a continuous group action \(E \times \Omega \to \Omega\) of \(E\) on \(\Omega\). For any affine automorphism \(T \colon E \to E\), the map \(T \colon \Omega \to T\Omega = \{T\cQ \mid \cQ \in \Omega\}\), given by \(\cQ \mapsto T\cQ\), is a uniformly continuous homeomorphism of pattern spaces. In the particular case of a hull, this is a homeomorphism of hulls \(T \colon \Omega_{\cP} \to \Omega_{T\cP}\).
\end{proposition}

\begin{proof}
Clearly translation is a group action as \(\Omega\) is closed under translation (Lemma \ref{lem:pattern spaces closed under LI and translation}). To see that it is continuous, let \(x \in E\) and \(\cP \in \Omega\) be arbitrary. To define an arbitrarily small neighbourhood of \(x \cdot \cP = \cP+x \in \Omega\), take an arbitrary \(e = \Phi(B,U) \in \mathscr{U}'\). Take any \(B' \nhd E\) with \(B' + B' \subseteq B\) and define \(e' \coloneqq \Phi(B',U')\) with \(U' \coloneqq U-x\). We will show that, for all \(y \in x+B'\) and \(\cQ \in B(\cP,e')\), we have \(y\cdot \cQ = \cQ + y \in B(\cP+x,e)\), establishing continuity.

Indeed, if \(\cQ \in B(\cP,e')\), then there exist \(b_1\), \(b_2 \in B'\) with \(\cP[b_1,U'] = \cQ[b_2,U']\). For arbitrary \(y = x + b' \in x + B'\) this gives (by definition of the translated pattern) the middle equality of
\[
(\cP+x)[b_1,U'+x] = (\cP+x)[b_1+x,U'] = (\cQ+x)[b_2+x,U'] = (\cQ+x)[b_2,U'+x]  ,
\]
where the outside equalities use Lemma \ref{lem:pattern properties} (2). From \(U = U'+x\) then \(x = y-b'\), we have
\[
(\cP+x)[b_1,U] = (\cQ+x)[b_2,U] = ((\cQ+y)-b')[b_2,U] = (\cQ+y)[b_2+b',U] .
\]
Since \(b_1 \in B' \subset B\) and \(b_2+b' \in B'+B' \subseteq B\), we have that \(\cQ+y \in B(\cP+x,e)\), as required.

Next, let \(T(x) = L(x) + z\) and \(e = \Phi(B,U) \in \mathscr{U}'\). Let \(e' \coloneqq \Phi(L^{-1}(B),T^{-1}(U)) \in \mathscr{U}'\). We claim that, if \(\cQ \in B(\cP,e')\), then \(T\cQ \in B(T\cP,e)\). Indeed, \(\cP[L^{-1}b_1,T^{-1}U] = \cQ[L^{-1}b_2,T^{-1}U]\) for some \(b_1\), \(b_2 \in B\). By Lemma \ref{lem:pattern properties}, \((T\cP)[TL^{-1}b_1,LT^{-1}U] = (T\cQ)[TL^{-1}b_2,LT^{-1}U]\). Now, \(TL^{-1}x = x+z\) and \(LT^{-1}x = x-z\). Thus, \((T\cP)[b_1+z,U-z] = (T\cQ)[b_2+z,U-z]\) so \((T\cP)[b_1,U] = (T\cQ)[b_2,U]\) by Lemma \ref{lem:pattern properties}, hence \(T\cQ \in B(T\cP,e)\), showing that \(\cP \mapsto T\cP \in T\Omega\) is uniformly continuous.

To see that \(T\Omega\) is a pattern space, note that the equalities of \(\lang(T\cP) = L(\lang(\cP)) \subseteq L(\lang(\Omega)) = \lang(T\Omega)\) always hold, by Lemma \ref{lem:distortions of languages}. The inclusion holds if and only if \(\lang(\cP) \subseteq \lang(\Omega)\), that is, \(\cP \LIin \Omega\) (since \(L\) acts as a bijection between languages, with inverse \(L^{-1}\)). An arbitrary pattern may be written in the form \(T\cP\) and the above shows that \(T\cP \LIin T\Omega\) if and only if \(\cP \LIin \Omega\), equivalently \(\cP \in \Omega\) (since \(\Omega\) is a pattern space), which in turn is equivalent (by definition of \(T\Omega\)) to \(T\cP \in T\Omega\), as required. In the case of a hull, \(\lang(T\Omega_{\cP}) = L\lang(\Omega_{\cP}) = L(\lang(\cP)) = \lang(T\cP)\) i.e., \(\cQ \in T\Omega_{\cP}\) if and only if \(\cQ \LI T\cP\), so \(T\Omega_{\cP} = \Omega_{T\cP}\), as required.
\end{proof}

\subsection{FLC and compactness}
In this section we extend a standard result \cite{Rob04}, that compactness is equivalent to FLC, but for general pattern spaces with the local topology. For compactness, we use the equivalent condition of being complete and totally bounded (\cite[Theorem 39.9]{Wil04}). Recall that a uniform space is \textbf{totally bounded} if, for all \(e = \Phi(B,U) \in \mathscr{U}'\), there exists some finite cover \(\{A_1,\ldots,A_n\}\) of \(\Omega\) such that, for each \(i\), we have \(A_i \times A_i \subseteq e\). In other words, it is a covering of \(\Omega\) by finitely many \(e\)-small sets. To extend this beyond hulls (or orbit closures) of single patterns, we define the notion of FLC also for pattern \emph{spaces}:

\begin{definition}\label{def:FLC pattern space}
A pattern space \(\Omega\) is of \textbf{finite local complexity} (or is \textbf{FLC}) if, for each \(U \cpt E\), there are finitely many patterns \(\cP_1\), \(\cP_2\), \ldots, \(\cP_n\) (\(n\) possibly depending on \(U\)) and some \(K(U) \cpt E\) for which, for all \(p \in \lang_U(\Omega)\), there exists some \(i \in \{1,\ldots,n\}\) and \(v \in K(U)\) with \(p = \cP_i[v,U]\).
\end{definition}

The above may be read as: all \(U\)-patches in \(\Omega\) appear within a bounded distance of the origin in one of finitely many patterns. So for a pattern space of tilings (with in-radii and out-radii of all tiles positively bounded from below and above, respectively), this just means that only finitely many `uncentred' patches appear over all tilings in the space (equivalently, there are only finitely many tiles and finitely many ways for them to meet at their boundaries, modulo translation). The value of \(n\) in Definition \ref{def:FLC pattern space} may need to depend on \(U\) e.g., consider Example \ref{exp:pattern space closed under LI but not LIin}. Of course, for hulls, this notion is consistent:

\begin{lemma}\label{lem:pattern FLC <=> hull FLC}
A pattern \(\cP\) is FLC if and only if its hull \(\Omega_{\cP}\) is FLC.
\end{lemma}

\begin{proof}
If \(\cP\) is FLC then, trivially, so is \(\Omega_{\cP}\). Indeed, for \(U \cpt E\), we may take \(n=1\), \(\cP_1 \coloneqq \cP\) and the same subset \(K(U)\) as from Definition \ref{def:FLC pattern space}. Conversely, suppose that \(\Omega_{\cP}\) is FLC and \(U \cpt E\). Let \(\cP_1\), \ldots, \(\cP_n \in \Omega_{\cP}\) and \(K = K(U)\) be as in Definition \ref{def:FLC pattern space}. Since each \(\cP_i \LI \cP\), there exist \(z_i \in E\) for which \(\cP[z_i,K+U] = \cP_i[\mathbf{0},K+U]\). Let \(K' \coloneqq \bigcup_{i=1}^n (K+z_i) \cpt E\), which we claim satisfies the requirement of Definition \ref{def:FLC} for \(U\)-patches. Indeed, given \(p \in \lang_U(\cP)\), by Definition \ref{def:FLC pattern space} there exists some \(i\) and \(y_i \in K\) with \(p = \cP_i[y_i,U]\). Since \(\cP[z_i,K+U] = \cP_i[\mathbf{0},K+U]\), shifting by \(y_i \in K\) we have \(\cP[y_i+z_i,U] = \cP_i[y_i,U] = p\). Since \(y_i + z_i \in K + z_i \subseteq K'\) and \(p = \cP[y_i+z_i,U] \in \lang_U(\cP)\) was arbitrary, we have verified Definition \ref{def:FLC}.
\end{proof}

\begin{lemma}\label{lem:FLC<=>totally bounded}
A pattern space \(\Omega\) is totally bounded if and only if it is FLC.
\end{lemma}

\begin{proof}
Suppose that \(\Omega\) is totally bounded and let \(U \cpt E\) be arbitrary. Fix any \(B \cptn E\), say, \(B = B_1\). By total boundedness, there is a finite covering \(\{A_i\}_{i=1}^n\) of \(\Omega\) with each \(A_i \times A_i \subseteq e \coloneqq \Phi(B,U-B)\). Without loss of generality, each \(A_i \neq \emptyset\), so choose any \(\cP_i \in A_i\) for each \(i\). We claim that taking \(K(U) \coloneqq B-B\) satisfies the requirement for FLC in Definition \ref{def:FLC pattern space}. Indeed, take any \(p \in \lang_U(\Omega)\), say \(p = \cQ[x,U]\). Since the \(A_i\) cover \(\Omega\), we have that \(\cQ-x \in A_i\) for some \(i\), thus \((\cQ-x,\cP_i) \in \Phi(B,U-B)\) so that \((\cQ-x)[b,U-B] = \cP_i[b',U-B]\) for some \(b\), \(b' \in B\). By Lemma \ref{lem:pattern properties}, shifting by \(-b \in -B\),
\[
p = \cQ[x,U] = (\cQ-x)[\mathbf{0},U] = \cP_i[b'-b,U] .
\]
Since \(b'-b \in B-B = K(U)\), we have shown that \(\Omega\) is FLC.

Conversely, suppose that \(\Omega\) is FLC and take an arbitrary \(e = \Phi(B,U) \in \mathscr{U}'\) where, without loss of generality (replacing \(B\) with a subset if needed), \(B\) is compact. Apply Definition \ref{def:FLC pattern space} to \((U+B)\)-patches, defining patterns \(\cP_1\), \ldots, \(\cP_n\) and \(K = K(U+B) \cpt E\). We may cover \(K\) with finitely many translates of \(-B\), say \(K \subseteq \bigcup_{j=1}^m (-B+z_j)\). For \(i \in \{1,\ldots,n\}\) and \(j \in \{1,\ldots,m\}\), we define
\[
A_i^j \coloneqq \{\cQ \in \Omega \mid \cQ[b,U] = \cP_i[z_j,U] \text{ for some } b \in B\} .
\]
We claim that the \(A_i^j\) cover \(\Omega\). Indeed, given \(\cQ \in \Omega\), by definition of \(K\) we have \(\cQ[\mathbf{0},U+B] = \cP_i[y,U+B]\) for some \(y \in K\). By our definition of the \(z_j\) we have \(y = -b + z_j \in -B+z_j\) for some \(j \in \{1,\ldots,m\}\) and \(b \in B\). Thus, shifting by \(-y+z_j = b \in B\), we have \(\cQ[b,U] = \cP_i[z_j,U]\) so that \(\cQ \in A_i^j\), showing the \(A_i^j\) cover \(\Omega\). Moreover, each \(A_i^j \times A_i^j \subseteq e\) since, given \(\cQ_1\), \(\cQ_2 \in A_i^j\), we have \(\cQ_1[b_1,U] = \cP_i[z_j,U] = \cQ_2[b_2,U]\) for \(b_i \in B\), so \((\cQ_1,\cQ_2) \in \Phi(B,U)\), as required.
\end{proof}

\begin{proposition}\label{prop:hull complete}
Any pattern space \(\Omega\) is complete.
\end{proposition}

\begin{proof}
Let \((\cQ_i)_{i \in \mathcal{I}}\) be a Cauchy net in \(\Omega\), for a directed set \(\mathcal{I}\) (in fact, it is sufficient to only consider Cauchy sequences, see the comments below Lemma \ref{lem:uniformity well-defined}, but the argument is essentially identical nonetheless). We wish to show that \((\cQ_i)_i\) is convergent. Given \(n \in \N\), let \(e_n \coloneqq \Phi(B_{2^{-(n+1)}},B_{n+2})) \in \mathscr{U}'\). Since \((\cQ_i)_i\) is Cauchy, for all \(n \in \N\) we may find \(\overline{n} \in \mathcal{I}\) with \((\cQ_u,\cQ_v) \in e_n\) for all \(u\), \(v \geq \overline{n}\). We can obviously arrange that \(\overline{m} \geq \overline{n}\) for \(m \geq n\). Then, in particular, we have \((\cQ_{\overline{n}},\cQ_{\overline{n+1}}) \in e_n\), thus \(\cQ_{\overline{n}}[x_n,B_{n+2}] = \cQ_{\overline{n+1}}[y_n,B_{n+2}]\) for \(x_n\), \(y_n \in B_{2^{-(n+1)}}\). Shifting by \(-y_n \in B_1\),
\begin{equation}\label{eq:completeness}
\cQ_{\overline{n}}[x_n-y_n,B_{n+1}] = \cQ_{\overline{n+1}}[\mathbf{0},B_{n+1}] .
\end{equation}
Define \(z_n \coloneqq x_n - y_n\) and \(z_{\geq n} \coloneqq \sum_{i=n}^\infty z_n\), which exists since each \(\|z_n\| \leq 2^{-n}\). Define \(\cQ \in \Omega\) by setting \(\cQ[\mathbf{0},B_n] = p_n \coloneqq \cQ_{\overline{n}}[z_{\geq n},B_n]\) i.e., \(\cQ[u] \coloneqq \cQ_{\overline{n}}[z_{\geq n}+u]\) once \(n \geq \|u\|\). This is a well-defined element of \(\Omega\). Indeed, the patches \(p_n\) are nested, that is, the restriction of \(p_{n+1}\) to \(B_n\) is \(p_n\). This follows from \(\cQ_{\overline{n}}[z_{\geq n},B_n] = \cQ_{\overline{n}}[z_n + z_{\geq (n+1)},B_n] = \cQ_{\overline{n+1}}[z_{\geq (n+1)},B_n]\), where we shift Equation (\ref{eq:completeness}) by \(z_{\geq n+1}\) in the second equality, using that \(\|z_{\geq n+1}\| \leq 1\). It follows that \(\cQ\) is a well-defined pattern. To see that it is in \(\Omega\), take any patch \(p = \cQ[x,U] \triangleleft \cQ\) and \(n\) sufficiently large so that \(U+x \subseteq B_n\). For \(\cR = \cQ_{\overline{n}} \in \Omega\), we have that \(\cQ[\mathbf{0},B_n] = p_n = \cR[z_{\geq n},B_n]\) and, since \(U+x \subseteq B_n\), we have \(\cQ[\mathbf{0},U+x] = \cR[z_{\geq n},U+x]\), equivalently \(p = \cQ[x,U] = \cR[z_{\geq n}+x,U]\). So \(p \triangleleft \cR \in \Omega\) for the arbitrary \(p \triangleleft \cQ\), hence \(\cQ \LIin \Omega\), confirming \(\cQ \in \Omega\). 

Next, we show that \(\cQ_i \to \cQ\). Indeed, let \(e = \Phi(B,U) \in \mathscr{U}'\) be arbitrary and \(n \geq 1\) be such that \(B_{2^{3-n}} \subseteq B\) and \(B_n \supseteq U\). By definition of \(\overline{n} \in \mathcal{I}\), for \(i \geq \overline{n}\) we have \((\cQ_i,\cQ_{\overline{n}}) \in e_n\), that is, 
\[
\cQ_i[x,B_{n+2}] = \cQ_{\overline{n}}[y,B_{n+2}] 
\]
for \(x\), \(y \in B_{2^{-(n+1)}} \subseteq B_{2^{1-n}}\), where \(z_{\geq n} \in B_{2^{1-n}}\) too, so \(z_{\geq n} - y \in B_{2^{2-n}} \subseteq B_2\). Thus, shifting by this vector,
\[
\cQ_i[x+z_{\geq n}-y,B_n] = \cQ_{\overline{n}}[z_{\geq n},B_n] = p_n = \cQ[\mathbf{0},B_n] .
\]
Since \(x+z_{\geq n}-y \in B_{2^{3-n}} \subseteq B\) and \(B_n \supseteq U\), we have \(\cQ_i \in B(\cQ,e)\). Since \(e \in \mathscr{U}'\) and \(i \geq \overline{n}\) were arbitrary, it follows that \(\cQ_i \to \cQ\), as required.
\end{proof}

Because a uniform space is compact if and only if it is complete and totally bounded, we obtain the following from Lemma \ref{lem:FLC<=>totally bounded} and Proposition \ref{prop:hull complete} (and Lemma \ref{lem:pattern FLC <=> hull FLC} for hulls):

\begin{corollary}\label{cor:cpt<=>FLC}
A pattern space \(\Omega\) is compact if and only if it is FLC. The translational hull \(\Omega_{\cP}\) of a pattern \(\cP\) is compact if and only if \(\cP\) is FLC.
\end{corollary}

\begin{example}\label{exp:pinwheel}
If \(\cT\) is defined (as in Section \ref{sec:patterns from tilings}) for a pinwheel tiling \cite{Rad94}, which is not FLC (with respect to translations), the translational hull \(\Omega_{\cT}\) as defined here, with the local uniformity, is complete but not totally bounded, in particular it is not compact. Of course, in this case the local topology on \(\Omega_{\cT}\) is unnatural, it requires close tilings to have perfect agreement to a large radius up to a small \emph{translation}, rather than a small rigid motion. It would be interesting to investigate which results here may be generalised beyond the local topology, for instance by defining a `rubber' uniformity similarly to how one defines the rubber topology \cite[Remark 5.6]{BG13}.
\end{example}

Although an FLC pattern space may contain infinitely many different languages of individual patterns (e.g., see the one in Example \ref{exp:pattern space closed under LI but not LIin}), there are only finitely many when restricted to each shape \(U\). More precisely:

\begin{lemma}\label{lem:FLC => finitely many configs}
For a pattern space \(\Omega \subseteq A^E\) and \(U \cpt E\), define the set of possible `configurations'
\[
\mathscr{C}_U(\Omega) \coloneqq \{\mathscr{L}_U(\cP) \mid \cP \in \Omega\} 
\]
of \(U\)-patches appearing over all \(\cP \in \Omega\). If \(\Omega\) is FLC then each \(\mathscr{C}_U(\Omega)\) is finite.
\end{lemma}

\begin{proof}
Define \(\cP_1\), \ldots, \(\cP_n \in \Omega\) and \(K = K(U+B_2) \cpt E\) as in Definition \ref{def:FLC pattern space} of FLC. Take a finite set \(S \subseteq K\) that is \(1\)-dense, that is, for each \(x \in K\) there exists some \(s \in S\) with \(s-x \in B_1\). We define maps \(F_i \colon \Omega \to 2^S\) by
\[
F_i(\cP) \coloneqq \{s \in S \mid \cP_i[s,U+B_1] \in \lang(\cP)\} .
\]
Since \(S\) is finite, there are only finitely many possible sets \(F_i(\cP) \subseteq S\), so the result follows if each \((F_1(\cP),\ldots,F_n(\cP))\) fully determines \(\lang_U(\cP)\).

To show this, given \(\cP\), \(\cQ \in \Omega\), suppose that \(F_i(\cP) = F_i(\cQ)\) for each \(i\) and take an arbitrary \(p = \cP[x,U] \in \lang_U(\cP)\). There exists some \(i\) and \(y \in K\) with \(\cP[x,U+B_2] = \cP_i[y,U+B_2]\). Since \(S\) is \(1\)-dense in \(K\), there exists \(s \in S\) and \(b \in B_1\) with \(y+b = s \in S\). Shifting by \(b\), we thus have \(\cP[x+b,U+B_1] = \cP_i[s,U+B_1]\), so \(s \in F_i(\cP)\). Since \(F_i(\cP) = F_i(\cQ)\), we also have \(s \in F_i(\cQ)\), so \(\cP_i[s,U+B_1] = \cQ[z,U+B_1]\) for some \(z \in E\). Shifting by \(-b \in B_1\), we have \(\cP_i[s-b,U] = \cQ[z-b,U]\). Since \(s-b = y\) and \(\cP_i[y,U] = \cP[x,U] = p\), we have thus shown that the arbitrary \(p \in \lang_U(\cP)\) is also in \(\lang_U(\cQ)\). The converse is proved analogously, so \(\lang_U(\cP) = \lang_U(\cQ)\), as required.
\end{proof}

\subsection{Repetitivity and minimality of pattern spaces}
In this subsection, we recover the statement in the generalised pattern space setting that repetitivity corresponds to minimality.

\begin{definition}\label{def:repetitive pattern space}
A pattern space \(\Omega\) is \textbf{repetitive} if, for all \(U \cpt E\), there exists some \(R = R(U) \cpt E\) for which, for all \(p \in \lang_U(\Omega)\) and \(\cP \in \Omega\), we have \(\cP[v,U] = p\) for some \(v \in R\).
\end{definition}

This simply says that any patch from a pattern of \(\Omega\) appears in all patterns of \(\Omega\), and with bounded gaps (by translation invariance of \(\Omega\), we can restrict attention to the origin in the above definition). If \(\Omega\) is repetitive then it is FLC, by taking \(n = 1\) and any \(\cP_1 \in \Omega\) for all \(U \cpt E\) in Definition \ref{def:FLC pattern space}, with \(K(U) = R(U)\).

\begin{lemma}\label{lem:repetitive <=> hull repetitive}
A pattern \(\cP\) is repetitive if and only if \(\Omega_{\cP}\) is repetitive.
\end{lemma}

\begin{proof}
Suppose that \(\cP\) is repetitive and let \(R = R(U)\) be as in Definition \ref{def:repetitive}. We show that it also satisfies Definition \ref{def:repetitive pattern space}. Indeed, let \(p \in \lang_U(\Omega) = \lang_U(\cP)\) and \(\cQ \in \Omega_{\cP}\) be arbitrary. Since \(\cQ \LI \cP\), there exists some \(x \in E\) with \(\cQ[\mathbf{0},U+R] = \cP[x,U+R]\). Shift this equation by some \(v \in R\) for which \(p = \cP[x+v,U]\) (which exists by repetitivity of \(\cP\)) to give \(\cQ[v,U] = p\), as required.

Conversely, suppose instead that \(\Omega\) is repetitive, \(U \cpt E\) and \(R = R(U)\) is as in Definition \ref{def:repetitive pattern space}. Let \(p \in \lang_U(\cP) = \lang_U(\Omega)\) and \(x \in E\) be arbitrary. Applying repetitivity of \(\Omega\) to \(\cP-x \in \Omega\), we have \(p = (\cP-x)[v,U] = \cP[x+v,U]\) for some \(v \in R\), as required for Definition \ref{def:repetitive}.
\end{proof}

The following is an analogue of Proposition \ref{prop:FLC rep}, that repetitivity is equivalent to relative density of repetition of each \emph{individual} patch in the FLC case:

\begin{proposition}\label{prop:FLC rep pattern spaces}
Suppose that \(\Omega\) is FLC. Then repetitivity of \(\Omega\) is equivalent to the following: for all \(p \in \lang(\Omega)\), there exists some \(R = R(p) \cpt E\) for which, for all \(\cP \in \Omega\), there exists some \(v \in R\) with \(p = \cP[v,U]\).
\end{proposition}

\begin{proof}
Again, in the above statement \(U \cpt E\) is as determined by the shape of \(p\). Clearly if \(\Omega\) is repetitive then it is also repetitive for individual patches in the proposed sense, by taking \(R(p) = R(U)\). Conversely, suppose the proposed property holds and \(\Omega\) is FLC, so we may find \(\cP_1\), \ldots, \(\cP_n\) and \(K = K(U) \cpt E\) as in Definition \ref{def:FLC pattern space}. Let \(q_i \coloneqq \cP_i[\mathbf{0},K+U]\). We claim that taking \(R(U) \coloneqq \bigcup_{i=1}^n (K + R(q_i))\) satisfies the requirement of repetitivity in Definition \ref{def:repetitive pattern space}.

Indeed, take arbitrary \(p \in \lang_U(\Omega)\) and \(\cP \in \Omega\). Then \(p = \cP_i[y,U]\) for some \(i\) and \(y \in K\). Moreover, by assumption, \(q_i \coloneqq \cP_i[\mathbf{0},K+U] = \cP[v,K+U]\) for some \(v \in R(q_i)\). Shifting by \(y\), we have \(p = \cP_i[\mathbf{0}+y,U] = \cP[v+y,U]\), where \(v+y \in K+R(q_i) \subseteq R(U)\), as required.
\end{proof}

\begin{proposition}\label{prop:repetitive <=> minimal}
Let \(\Omega\) be a pattern space. Then (1) implies (2) and (2)--(4) are equivalent. If \(\Omega\) is FLC, then (1)--(4) are all equivalent.
\begin{enumerate}
	\item \(\Omega\) is repetitive;
	\item for all \(\cP \in \Omega\), we have \(\lang(\cP) = \lang(\Omega)\);
	\item for all \(\cP\), \(\cQ \in \Omega\), we have that \(\cP \LIs \cQ\);
	\item for all \(\cP\), \(\cQ \in \Omega\), we have that \(\cP \in \overline{(\cQ + E)}\), that is, all orbits are dense, or \(\Omega\) is \textbf{minimal}.
\end{enumerate}
\end{proposition}

\begin{proof}
Let \(p \in \lang_U(\Omega)\) and \(\cP \in \Omega\) be arbitrary. If \(\Omega\) is repetitive then \(p = \cP[v,U]\) for some \(v \in R(U)\), so \(\lang(\cP) \supseteq \lang(\Omega)\). Since \(\lang(\Omega) = \bigcup_{\cP \in \Omega} \lang(\cP)\), by definition, the reverse also always holds, so (1) implies (2). The latter equality also immediately gives equivalence of (2) and (3); recall that \(\cP \LIs \cQ\) when \(\lang(\cP) = \lang(\cQ)\). Equivalence of (3) and (4) follows immediately from Corollary \ref{cor:hulls are orbit closures}.

Finally, suppose that \(\Omega\) is FLC but not repetitive. Thus, by Proposition \ref{prop:FLC rep pattern spaces}, there exists some \(p \in \lang_U(\Omega)\) for which, for all \(n \geq 1\), there exists \(\cP_n \in \Omega\) with \(\cP_n[z,U] \neq p\) for all \(z \in B_n\). Using compactness, \(\cP_n \to \cP \in \Omega\) along some subsequence. We claim that \(p \notin \lang_U(\cP)\), thus \(\lang(\cP) \neq \lang(\Omega)\), establishing that (2) implies (1). Indeed, suppose to the contrary that \(p = \cP[x,U]\), say \(x \in B_k\). By definition of the local topology, for sufficiently large \(n > k\) in the convergent subsequence we have that \(\cP_n[\mathbf{0},B_{k+1}+U] = \cP[y,B_{k+1}+U]\) for some \(y \in B_1\). Shifting by \(z = x-y \in B_{k+1} \subseteq B_n\), we have \(\cP_n[z,U] = \cP[x,U] = p\), contradicting our assumption on \(\cP_n\).
\end{proof}

\begin{remark}
By (1) implies (2) above, repetitive pattern spaces \(\Omega\) are hulls \(\Omega = \Omega_{\cP}\) (for any \(\cP \in \Omega\)). We cannot drop the assumption of FLC in the proof of (2) implying (1). The non-FLC (thus non-repetitive) pattern \(\cP\) where all points are inequivalent (see Example \ref{exp:trivial patterns}) has hull \(\Omega = \cP + E\), so \(\cQ = \cP + x \LIs \cP\) for all \(\cQ \in \Omega\), but \(\Omega \cong E\) is not compact.
\end{remark}

\subsection{Separation properties}\label{sec:separation properties}
So far in this section it has not been prohibited --- and nor has it complicated matters --- that our pattern spaces could be non-Hausdorff. Since we will usually need Hausdorff in our later results, we give conditions for it in less abstract terms of patterns. Of course, for standard examples the Hausdorff property will be essentially automatic. But the following absurd example has not been excluded from any definitions so far in this section:

\begin{example}\label{exp:rationals as point set}
Let \(\cP \colon \R \to \{0,1\}\) be defined by \(\cP[x] = 0\) if \(x \in \Q\) and \(\cP[x] = 1\) if \(x \notin \Q\). We may think of it as the `decoration' of \(E = \R\) by marking the rationals. Clearly \(\cP\) has periods \(\cK = \Q\). In particular, \(\cP\) is FLC: every \(U\)-patch has centre contained in \(K\), for any \(K \cpt \R\) with non-empty interior. All patterns locally indistinguishable from \(\cP\) are translates of \(\cP\) and \(\Omega_{\cP} \cong \R / \Q\), with the non-Hausdorff trivial (or `indiscrete') topology. Amusingly, \(\cP\) happens to be \(L\)-sub (see Definition \ref{def:L-sub}) for any \(L(x) = qx\) with non-zero \(q \in \Q\).
\end{example}

In the following definition we introduce a strong notion of separation that will force the Hausdorff property, and groups of periods of patterns all being discrete. Then we give a weaker notion, which is equivalent to the hull being Hausdorff in the FLC case. Both notions are also extended to general pattern spaces.

\begin{definition}\label{def:return discrete}
A pattern \(\cP\) is \textbf{return discrete} if there exists \(B \nhd E\) and \(V \cpt E\) for which, for all \(x \in E\), if \(\cP[x,V] = \cP[x+b,V]\) for \(b \in B\) then \(b = \mathbf{0}\).

Similarly, a pattern space \(\Omega\) is \textbf{return discrete} if there exists some \(B \nhd E\) and \(V \cpt E\) for which, for all \(\cU \in \Omega\), if \(\cU[\mathbf{0},V] = \cU[b,V]\) for \(b \in B\) then \(b = \mathbf{0}\).
\end{definition}

In words, return vectors between sufficiently large patches cannot be arbitrarily short.

\begin{example}
Suppose that \(\cP\) is a pattern defined by a tiling (say, according to the second approach of defining a pattern from a tiling in Section \ref{sec:patterns from tilings}). Suppose that all tiles contain a translate of some \(B \nhd E\) and that there is some \(V \cpt E\) so that all translates of \(V\) fully contain at least one whole tile. Then \(\cP\) is return discrete. Indeed, suppose that \(\cP[x,V] = \cP[x+b,V]\) for some \(b \in B\). Take any tile \(t\) with support fully contained in \(V+x\), and with interior containing some translate \(B+y\). Since \(\cP[x,V] = \cP[x+b,V]\), we must also have that the tile \(t+b\) appears in the tiling. But both \(t\) and \(t+b\) contain \(y+b\) in their interiors, so these tiles overlap and thus must be exactly equal, hence \(b = \mathbf{0}\), showing return discreteness of \(\cP\).

By a similar argument, if \(\Omega\) is a pattern space whose elements are defined by tilings, with a common \(B\) and \(V\) controlling the in- and out-radii of all tiles, as above, then \(\Omega\) is return discrete.
\end{example}

Similarly to the above example, a pattern defined by a Delone set is return discrete (as is a pattern space with uniform bounds on the uniformly discrete and relatively dense constants), and indeed such patterns (and pattern spaces) agree modulo MLD equivalence. A notable class of examples that are \emph{not} return discrete are uniformly discrete but non-relatively dense point sets. These are nonetheless sufficiently well behaved to be be covered by many of our main results and satisfy the following:

\begin{definition}\label{def:well-separated}
A pattern \(\cP\) is \textbf{well-separated} if there exists \(V \cpt E\) for which, for all \(x\), \(y \in E\), if \(\cP[x] \neq \cP[y]\) then there exists \(B \nhd E\) with \(\cP[x+b,V] \neq \cP[y+b',V]\) for all \(b\), \(b' \in B\).

Similarly, a pattern space \(\Omega\) is \textbf{well-separated} if there exists \(V \cpt E\) for which, for all \(\cU\), \(\cV \in \Omega\), if \(\cU[\mathbf{0}] \neq \cV[\mathbf{0}]\) then there exists \(B \nhd E\) for which \(\cU[b,V] \neq \cV[b',V]\) for all \(b\), \(b' \in B\).
\end{definition}

Note that the neighbourhood \(B\) of \(\mathbf{0}\) in the above definition can depend on the local patches in question (but is fixed in the definition of return discrete).

\begin{example}
Let \(\cP\) be a pattern defined by (see Section \ref{sec:patterns from point sets}) a uniformly discrete point set \(\Lambda\) (possibly coloured, but not necessarily relatively sense). Suppose that \(\cP[x] \neq \cP[y]\). Thus, precisely one of \(x \in \Lambda\) or \(y \in \Lambda\) with some colour \(\ell\) (with the other either being not in \(\Lambda\), or in it but with different colour). Without loss of generality, \(\cP[x] = \ell\) and \(\cP[y] \neq \ell\). By uniform discreteness (in fact, one can even assume less), there exists some \(B \nhd E\) with \(B \subseteq B_1\) for which no point of \(y+(B-B)\) is in \(\Lambda\) with colour \(\ell\). Then \(\cP[x+b,B_1] \neq \cP[y+b',B_1]\) for all \(b\), \(b' \in B_1\). Indeed, for patch \(p \coloneqq \cP[x+b,B_1]\), we have \(p(-b) = \ell\), where \(b \in B\), thus \(-b \in -B \subseteq B_1\). In contrast, for \(q \coloneqq \cP[y+b',B_1]\), we have that \(q(-z) \neq \ell\) for all \(z \in B\). Indeed, otherwise, \(\cP[y+b'-z] = \ell\), meaning that \(y+(b'-z) \in y+(B-B)\) is in \(\Lambda\) with colour \(\ell\), contradicting our assumption.

Thus, taking for instance \(V = B_1\) in Definition \ref{def:well-separated}, uniformly discrete point sets define well-separated patterns, although they will not be return discrete unless they are also relatively dense. Similarly, a pattern space defined by uniformly discrete point sets will be well-separated.
\end{example}

A simple consequence of a pattern space being Hausdorff (c.f.\ Example \ref{exp:rationals as point set}) is the following:

\begin{lemma}\label{lem:HD=>periods closed}
For a Hausdorff pattern space \(\Omega\), for all \(\cP \in \Omega\) we have that \(\cK_{\cP}\) is closed in \(E\).
\end{lemma}

\begin{proof}
Let \(\cP \in \Omega\) and \((g_n)_{n=1}^\infty\) be any sequence in \(\cK \coloneqq \cK_{\cP}\) with \(g_n \to g \in E\) convergent in \(E\). By Proposition \ref{prop:pattern spaces are DS}, translation \(E \times \Omega \to \Omega\) is continuous, so \(\cP + g_n \to \cP + g\). Moreover, \(\cP + g_n = \cP\) for all \(n\) since each \(g_n \in \cK\), so \(\cP + g_n = \cP \to \cP\). If \(\Omega\) is Hausdorff, the limit of \((\cP + g_n)_n\) must be unique, so \(\cP = \cP + g\) and thus \(g \in \cK\) too. Since \((g_n)_{n=1}^\infty\) was an arbitrary sequence in \(\cK\) that converges in \(E\), we have that \(\cK \leqslant E\) is topologically closed, as required.
\end{proof}

\begin{lemma}\label{lem:separation for pattern spaces}
A pattern \(\cP\) is return discrete if and only if \(\Omega_{\cP}\) is return discrete. Similarly, \(\cP\) is well-separated if and only if \(\Omega_{\cP}\) is well-separated.
\end{lemma}

\begin{proof}
Let \(\cP\) be return discrete and take \(B \nhd E\) and \(V \cpt E\), as in Definition \ref{def:return discrete}. Without loss of generality, perhaps taking \(B\) smaller, we may assume that \(B \cptn E\). Let \(\cU \in \Omega = \Omega_{\cP}\) be arbitrary and suppose that \(\cU[\mathbf{0},V] = \cU[b,V]\) for some \(b \in B\). Take \(x \in E\) with \(\cU[\mathbf{0},B+V] = \cP[x,B+V]\), which exists since \(\cU \LI \cP\). Shifting by \(b\), we have \(\cU[b,V] = \cP[x+b,V]\). Thus, \(\cP[x,V] = \cP[x+b,V]\), so \(b = \mathbf{0}\), as required. Conversely, suppose that \(\Omega\) is return discrete and let \(V \cpt E\), \(B \nhd E\) be as in Definition \ref{def:return discrete}. Let \(x \in E\) be arbitrary. If \(\cP[x,V] = \cP[x+b,V]\) then \((\cP-x)[\mathbf{0},V] = (\cP-x)[b,V]\), so \(b = \mathbf{0}\), as required.

Suppose, instead, that \(\cP\) is well-separated and take \(V \cpt E\) as in Definition \ref{def:well-separated}. Making \(V\) larger, if necessary, we may suppose that \(\mathbf{0} \in V\). Take arbitrary \(\cU\), \(\cV \in \Omega\) with \(\cU[\mathbf{0}] \neq \cV[\mathbf{0}]\). Take \(x\), \(y \in E\) with \(\cP[x,V+B_1] = \cU[\mathbf{0},V+B_1]\) and \(\cP[y,V+B_1] = \cV[\mathbf{0},V+B_1]\). In particular, \(\cP[x] = \cU[\mathbf{0}] \neq \cV[\mathbf{0}] = \cP[y]\), so there exists some \(B \nhd E\) for which \(\cP[x+b,V] \neq \cP[y+b',V]\) for all \(b\), \(b' \in B\). Without loss of generality, \(B \subseteq B_1\). Then for all \(b\), \(b' \in B \subseteq B_1\), we have \(\cU[b,V] = \cP[x+b,V] \neq \cP[y+b',V] = \cV[b',V]\) so \(\Omega\) is well-separated. Conversely, suppose that \(\Omega\) is well-separated and take \(V \cpt E\) as in Definition \ref{def:well-separated}. Let \(x\), \(y \in E\) be arbitrary with \(\cP[x] \neq \cP[y]\). Equivalently, \((\cP-x)[\mathbf{0}] \neq (\cP-y)[\mathbf{0}]\), so there exists \(B \nhd E\) with \((\cP-x)[b,V] \neq (\cP-y)[b',V]\), equivalently \(\cP[x+b,V] \neq \cP[y+b',V]\), for all \(b\), \(b' \in B\). Hence, \(\cP\) is well-separated.
\end{proof}

\begin{lemma}\label{lem:return discrete => well-separated}
Let \(\Omega\) be a return discrete pattern space. Then \(\Omega\) is well-separated and, for \(B \nhd E\) as in Definition \ref{def:return discrete}, for all \(\cQ \in \Omega\) we have that \(\cK_{\cQ} \cap B = \{\mathbf{0}\}\). In particular, each \(\cK_{\cQ}\) is discrete.
\end{lemma}

\begin{proof}
Let \(V \cpt E\) and \(B \nhd E\) be as in Definition \ref{def:return discrete}. Without loss of generality, \(\mathbf{0} \in V\) and we may take \(B\) compact (by taking \(V\) larger and \(B\) smaller, if necessary). Take arbitrary \(\cU\), \(\cV \in \Omega\) with \(\cU[\mathbf{0}] \neq \cV[\mathbf{0}]\). Suppose, for a contradiction, that for arbitrarily large \(n \in \N\) we may find \(b_n\), \(b'_n \in B_{1/n}\) with \(\cU[b_n,V+B] = \cV[b_n',V+B]\). Shifting by \(-b_n' \in B_{1/n} \subseteq B\) (which holds for sufficiently large \(n\)), so that \(V-b_n' \subseteq V+B\), and analogously with \(m \in \N\),
\[
\cU[b_n-b_n',V] = \cV[\mathbf{0},V] = \cU[b_m-b_m',V] .
\]
By the return discrete property, this can only happen if \((b_n-b_n') = (b_m-b_m')\) once \(n\) and \(m\) are sufficiently large that \(B_{(2/n + 2/m)} \subseteq B\), since in that case \((b_n-b_n') - (b_m-b_m') \in B_{2/n}+B_{2/m} \subseteq B\). Thus, \(b_n - b_n'\) is eventually constant. But since \(b_n\), \(b_n' \to \mathbf{0}\), we must have \(b_n - b_n' \to \mathbf{0}\) and thus \(b_n = b_n'\) for sufficiently large \(n\). In particular, \(\cU[b_n,V+B] = \cV[b_n,V+B]\) for some \(n \in \N\). Shifting by \(-b_n \in B\), we have that \(\cU[\mathbf{0},V] = \cV[\mathbf{0},V]\) and, since \(\mathbf{0} \in V\), we have \(\cU[\mathbf{0}] = \cU[\mathbf{0}]\), contradicting our assumption, so \(\Omega\) is well-separated.

Now suppose that \(\cQ \in \Omega\) and \(g \in \cK_{\cQ}\). Then, in particular, we have that \(\cQ[\mathbf{0},V] = (\cQ-g)[\mathbf{0},V] = \cQ[g,V]\) so if \(g \in B\) then \(g = \mathbf{0}\), as required.
\end{proof}

\begin{proposition}\label{prop:well-sep <=> HD}
If \(\Omega\) is a well-separated pattern space then \(\Omega\) is Hausdorff. Conversely, if \(\Omega\) is FLC and \(\Omega\) is Hausdorff, then \(\Omega\) is well-separated.

Thus, \(\Omega\) is compact Hausdorff if and only if it is FLC and well-separated.
\end{proposition}

\begin{proof}
Let \(\Omega\) be well-separated and \(V \cpt E\) be as in Definition \ref{def:well-separated}. Take arbitrary \(\cU\), \(\cV \in \Omega\) with \(\cU \neq \cV\). Thus, \(\cU[x] \neq \cV[x]\) for some \(x \in E\). By the definition of being well-separated, there exists \(B \nhd E\) with \((\cU-x)[b,V] \neq (\cV-x)[b',V]\) for all \(b\), \(b' \in B\). Equivalently, \(\cU[b,V+x] \neq \cV[b',V+x]\) for all \(b\), \(b' \in B\). Thus, \(\cU\) and \(\cV\) are not \(e\)-close, for \(e = \Phi(B,V+x) \in \mathscr{U}'\). Since the uniformity separates \(\cU\) and \(\cV\), we have that \(\Omega\) is Hausdorff.

Conversely, suppose that \(\Omega\) is FLC (equivalently, compact) and Hausdorff. For a contradiction, suppose that \(\Omega\) is not well-separated. Equivalently, for arbitrarily large \(n \in \N\), we may find \(\cU_n\), \(\cV_n \in \Omega\) satisfying the following: \(\cU_n[b_1,n] = \cV_n[b_2,n]\) for arbitrarily small \(b_1\), \(b_2 \in E\), yet \(\cU_n[\mathbf{0}] \neq \cV_n[\mathbf{0}]\). Using compactness, take a subsequence for which \(\cU_n \to \cQ \in \Omega\).

Let \(N \in \N\) be such that, for all \(p \in \lang_{B_4}(\cQ)\), there exists some \(x \in B_N\) with \(p = \cQ[x,4]\). Such an \(N\) exists since \(\cQ\) is FLC, as easily follows from \(\Omega\) being FLC, see Proposition \ref{prop:FLC and repetitive inherited} later. There exists an index \(n \geq N+5\) for which
\begin{equation}\label{eq:HD=>ws1}
\cU_n[\mathbf{0},N+5] = \cQ[b,N+5] ,
\end{equation}
for some \(b \in E\), which we can make arbitrarily small by convergence \(\cU_n \to \cQ\), but in particular \(b \in B_1\). Moreover, by assumption, we have that
\begin{equation} \label{eq:HD=>ws2}
\cU_n[b_1,n] = \cV_n[b_2,n]
\end{equation}
for arbitrarily small \(b_1\) and \(b_2\), but also with \(\cU_n[\mathbf{0}] \neq \cV_n[\mathbf{0}]\). Fix some particular \(u = b_1\) and \(v = b_2\) in Equation \ref{eq:HD=>ws2}, with \(u\), \(v \in B_1\). Thus, \(\cV_n[\mathbf{0},n-1] = \cU_n[u-v,n-1]\). Since \(u-v \in B_2\) and \(n \geq N+5\) we have, first from Equation \ref{eq:HD=>ws1} and then the last equality,
\[
\cQ[b + (u-v),N+3] = \cU_n[u-v,N+3] = \cV_n[\mathbf{0},N+3] .
\]
Now, from Equation \ref{eq:HD=>ws2}, we have other \(b_1\), \(b_2 \in B_\epsilon\) for \(\epsilon > 0\) arbitrarily small (and also, say, \(\epsilon \leq 1\)) satisfying \(\cV_n[\mathbf{0},N+4] = \cU_n[b_1 - b_2,N+4]\) and thus, by the above,
\begin{equation}\label{eq:HD=>ws3}
\cQ[b + (u-v),N+3] = \cV_n[\mathbf{0},N+3] = \cU_n[b_1-b_2,N+3] = \cQ[b + (b_1-b_2),N+3] ,
\end{equation}
where the final equality comes from shifting Equation \ref{eq:HD=>ws1} by \(b_1-b_2 \in B_2\). Shifting Equation \ref{eq:HD=>ws3} by \(-(b+(u-v)) \in B_3\), we obtain \(\cQ[\mathbf{0},N] = \cQ[(b_1-b_2)-(u-v),N]\). Since \((b_1-b_2)-(u-v) \in B_4\), taking \(V = B_N\), we deduce that \((b_1-b_2)-(u-v) \in \cK_{\cQ}\) using Lemma \ref{lem:close repeat => period} (in short: we have a short return vector to a full patch). Since we could take \(b_1\), \(b_2\) arbitrarily small here, it also follows that \(u-v \in \cK_{\cQ}\), since the latter is topologically closed, by Lemma \ref{lem:HD=>periods closed}. But then
\[
\cU_n[\mathbf{0},N+3] = \cQ[b,N+3] = \cQ[b+(u-v),N+3] = \cV_n[\mathbf{0},N+3]
\]
where the first equality is from Equation \ref{eq:HD=>ws1}, the second from \(u-v \in \cK_{\cQ}\) and the third from the left-hand equality in Equation \ref{eq:HD=>ws3}. This contradicts \(\cU_n[\mathbf{0}] \neq \cV_n[\mathbf{0}]\), so \(\Omega\) is well-separated.
\end{proof}

Note that the above applies to individual patterns, as well as pattern spaces, by Lemma \ref{lem:separation for pattern spaces}. That is, if \(\cP\) is well-separated, then so is its hull, thus the hull is Hausdorff. Conversely, for \(\cP\) FLC, if \(\Omega_{\cP}\) is Hausdorff then \(\cP\) is well-separated. The assumption that \(\cP\) is FLC is needed in this implication, the reader is invited to try finding a counter-example to the claim if FLC is dropped.

The following shows that, for a compact pattern space, return discreteness is simply a matter of whether any elements have non-discrete groups of periods. In particular, if \(\cP\) is FLC and aperiodic (i.e., \(\Omega_{\cP}\) contains no periodic elements) then \(\cP\) is return discrete and thus \(\Omega_{\cP}\) is Hausdorff.

\begin{proposition}\label{prop:return discrete from discrete periods}
For a pattern space \(\Omega\), (1) implies (2) and, if \(\Omega\) is FLC, (2) implies (1):
\begin{enumerate}
	\item \(\Omega\) is return discrete;
	\item for all \(\cQ \in \Omega\), we have that \(\cK_{\cQ}\) is discrete.
\end{enumerate}
\end{proposition}

\begin{proof}
That (1) implies (2) is shown in Lemma \ref{lem:return discrete => well-separated} (in fact, the \(\cK_{\cQ}\) are `uniformly' discrete across all \(\cQ \in \Omega\)). For the converse, suppose that \(\Omega\) is FLC (equivalently, compact) but not return discrete. Thus, for each \(n \in \N\), we may find non-zero \(b_n \in B_{1/n}\) and \(\cP_n \in \Omega\) with
\begin{equation} \label{eq:rd<=> discrete-per 1}
\cP_n[\mathbf{0},n] = \cP_n[b_n,n] .
\end{equation}
Using compactness, take a convergent subsequence \(\cP_n \to \cQ \in \Omega\). We claim that \(\cK_{\cQ}\) is not discrete.

Indeed, take \(N \in \N\) so that, for all \(p \in \lang_{B_1}(\cQ)\), there exists some \(x \in B_N\) with \(p = \cQ[x,1]\). There exists such an \(N\), since \(\cQ\) is FLC (see Proposition \ref{prop:FLC and repetitive inherited} below). Since the subsequence \(\cP_n \to \cQ\), we have that
\begin{equation} \label{eq:rd<=> discrete-per 2}
\cP_n[b,N+1] = \cQ[\mathbf{0},N+1] 
\end{equation}
for some sufficiently large \(n \geq N+1\) and \(b \in B_1\). Shifting Equation \ref{eq:rd<=> discrete-per 1} by \(b \in B_1\) we have \(\cP_n[b,N] = \cP_n[b+b_n,N]\), since \(n \geq N+1\). Shifting Equation \ref{eq:rd<=> discrete-per 2} by \(b_n \in B_{1/n} \subseteq B_1\), we have \(\cP_n[b+b_n,N] = \cQ[b_n,N]\). Combining these, along with Equation \ref{eq:rd<=> discrete-per 2} restricted to \(B_N\),
\[
\cQ[\mathbf{0},N] = \cP_n[b,N] = \cP_n[b+b_n,N] = \cQ[b_n,N] .
\]
By Lemma \ref{lem:close repeat => period}, \(b_n \in \cK_{\cQ}\) since \(b_n \in B_1\) and \(\cQ\) contains centres of all \(1\)-patches. Since this holds for arbitrarily large \(n\), with \(b_n \neq \mathbf{0}\) and \(\|b_n\| \leq 1/n\), we have that \(\cK_{\cQ}\) is not discrete, as required.
\end{proof}

The following shows that the discussed properties are inherited by inclusion of pattern subspaces and membership of patterns, and similarly for local indistinguishability between patterns.

\begin{proposition}\label{prop:FLC and repetitive inherited}
Given an inclusion \(\Omega' \subseteq \Omega\) of pattern spaces, if \(\Omega\) is FLC, well-separated, return discrete or repetitive, then \(\Omega'\) also has this property. If \(\Omega\) is repetitive, then \(\Omega' = \Omega\).

Given patterns \(\cQ \LI \cP\), if \(\cP\) is FLC, well-separated, return discrete or repetitive then so is \(\cQ\). If \(\cP\) is repetitive, then \(\cP \LIs \cQ\).
\end{proposition}

\begin{proof}
Suppose that \(\Omega\) is FLC, equivalently compact by Corollary \ref{cor:cpt<=>FLC}. Since \(\Omega' \subseteq \Omega\) is closed (Lemma \ref{lem:LIin is closure inclusion}) it is compact, thus FLC again by Corollary \ref{cor:cpt<=>FLC}. That \(\Omega'\) is well-separated or return discrete, if \(\Omega\) has the same, follows trivially from the definitions. If \(\Omega\) is repetitive then, by (2) of Proposition \ref{prop:repetitive <=> minimal}, \(\lang(\Omega) = \lang(\cP)\) for all \(\cP \in \Omega\), from which it follows that \(\lang(\Omega) = \lang(\Omega')\). A pattern space is determined by its language, so \(\Omega = \Omega'\).

Given any \(\cP \in \Omega\), consider its hull \(\Omega_{\cP} \subseteq \Omega\). By the above, \(\Omega_{\cP}\) inherits any of the stated properties from \(\Omega\). Then \(\cP\) also has those properties, respectively, by Lemmas \ref{lem:pattern FLC <=> hull FLC}, \ref{lem:separation for pattern spaces} (twice) and \ref{lem:repetitive <=> hull repetitive}. Finally, if \(\cQ \LI \cP\) then \(\cQ \in \Omega_{\cQ} \subseteq \Omega_{\cP}\), so the corresponding statement for patterns follows from that for pattern spaces (although we note that it can alternatively be quite easily proved directly from definitions). When \(\cP\) is repetitive, \(\cP \LIs \cQ\) follows from Proposition \ref{prop:repetitive <=> minimal}.
\end{proof}

\subsection{Local derivations of pattern spaces}\label{sec:Local derivations of pattern spaces}
In this section we define LD maps between translational pattern spaces, and relate this to LDs of individual patterns, via induced LD maps of hulls. Recall the notion of a local derivation of individual patterns from Definition \ref{def:LD}.

\begin{definition}\label{def:LD map}
Let \(\Omega \subseteq A^E\) and \(\Omega' \subseteq B^E\) be pattern spaces and \(f \colon \Omega \to \Omega'\). We call \(f\) a \textbf{local derivation} (\textbf{LD}) \textbf{map}, and write \(f \colon \Omega \LDmap \Omega'\), if there exists \(c \geq 0\) (called a \textbf{derivation radius}) for which, for all \(\cP\), \(\cQ \in \Omega\) and \(x\), \(y \in E\), if \(\cP[x,c] = \cQ[y,c]\) then \((f\cP)[x] = (f\cQ)[y]\).
\end{definition}

Just as for Definition \ref{def:LD}, the above may be thought of as a sliding block code, but here as a factor map: it must be of the form \((f(\cP))[x] = f'(\cP[x,U])\), where \(f' \colon \lang_U(\Omega) \to B\) is some `coding' map, converting \(U\)-patches into labels for some sufficiently large `window' \(U \cpt E\).

The analogue of Lemma \ref{lem:increasing LD radius} is the following, which again is rather trivial but used enough to be worth stating:

\begin{lemma}\label{lem:increasing derivation}
Suppose that \(f \colon \Omega \LDmap \Omega'\), with derivation radius \(c\). Then, for all \(U \cpt E\), \(\cP\), \(\cQ \in \Omega\) and \(x\), \(y \in E\), if \(\cP[x,B_c+U] = \cQ[y,B_c+U]\) then \((f \cP)[x,U] = (f\cQ)[y,U]\). In particular, if \(\cP[x,r+c] = \cQ[y,r+c]\) then \((f\cP)[x,r] = (f\cQ)[y,r]\).

\end{lemma}

\begin{proof}
Suppose that \(\cP[x,B_c+U] = \cQ[y,B_c+U]\). For \(u \in U\), since \(B_c + u \subseteq B_c + U\), we have \(\cP[x+u,B_c] = \cQ[y+u,B_c]\) (Lemma \ref{lem:pattern properties}) and hence, by the LD property, \((f\cP)[x+u] = (f\cQ)[y+u]\). As this holds for all \(u \in U\), by definition \((f\cP)[x,U] = (f\cQ)[y,U]\), as required. The second claim follows immediately, taking \(U = B_r\).
\end{proof}

\begin{lemma}\label{lem:LD is factor map}
If \(f \colon \Omega \LDmap \Omega'\), then \(f\) is uniformly continuous, and a factor map, that is, \(f(\cP-x) = (f\cP)-x\) for all \(\cP \in \Omega\) and \(x \in E\).

Conversely, if \(f \colon \Omega \to \Omega'\) is a factor map, then \(f \colon \Omega \LDmap \Omega'\) is LD with derivation radius \(c\) if and only if, for all \(\cP\), \(\cQ \in \Omega\), if \(\cP[\mathbf{0},c] = \cQ[\mathbf{0},c]\) then \((f\cP)[\mathbf{0}] = (f\cQ)[\mathbf{0}]\).
\end{lemma}

\begin{proof}
Given \(e = \Phi(B,U) \in \mathscr{U}'\), define \(e' = \Phi(B,U+B_c) \in \mathscr{U}'\). If \(\cP\), \(\cQ \in \Omega\) are \(e'\)-close then \(\cP[b,U+B_c] = \cQ[b',U+B_c]\) for \(b\), \(b' \in B\). By Lemma \ref{lem:increasing derivation}, applying the LD, \((f\cP)[b,U] = (f\cQ)[b',U]\), thus \(f\cP\) and \(f\cQ\) are \(e\)-close so \(f\) is uniformly continuous.

Given \(\cP \in \Omega\), \(x \in E\) and \(r \geq 0\), we have \(\cP[x,r+c] = (\cP-x)[\mathbf{0},r+c]\). By Lemma \ref{lem:increasing derivation}, \((f\cP)[x,r] = (f(\cP-x))[\mathbf{0},r]\), thus \(((f\cP)-x)[\mathbf{0},r] = (f(\cP-x))[\mathbf{0},r]\). Since this holds for all \(r \geq 0\), we have \((f\cP)-x = f(\cP-x)\), as required.

Now suppose that \(f \colon \Omega \to \Omega'\) is a factor map and take arbitrary \(\cP\), \(\cQ \in \Omega\) and \(x\), \(y \in E\) with \(\cP[x,c] = \cQ[y,c]\), equivalently, \((\cP-x)[\mathbf{0},c] = (\cQ-y)[\mathbf{0},c]\). Assuming the final stated property of the lemma then gives the middle equality below:
\[
(f\cP)[x] = (f(\cP)-x)[\mathbf{0}] = (f(\cP-x))[\mathbf{0}] = (f(\cQ-y))[\mathbf{0}] = (f(\cQ)-y)[\mathbf{0}] = (f\cQ)[y] ,
\]
as required. The converse direction, assuming \(f\) is LD, is similarly trivial, as it is just the restriction of the LD property of Definition \ref{def:LD map} to \(x=y=\mathbf{0}\).
\end{proof}

The following shows that taking as objects pattern spaces over a common Euclidean space \(E\), with LD maps between them as morphisms, defines a category (and is a counterpart to Lemma \ref{lem:LD preorder}, that LD defines a pre-order for individual patterns):

\begin{lemma}\label{lem:LD category}
For any pattern space \(\Omega\), the identity map \(\id \colon \Omega \LDmap \Omega\) is LD. Given \(f \colon \Omega \LDmap \Omega'\) and \(g \colon \Omega' \LDmap \Omega''\), we also have that \(g \circ f\) is LD.
\end{lemma}

\begin{proof}
If \(\cP[x,c] = \cQ[y,c]\) then \(\cP[x] = \cQ[y]\), so the identity map is trivially a local derivation (for any derivation radius \(c \geq 0\)). Suppose that \(f\) has derivation radius \(c_1\) and \(g\) has derivation radius \(c_2\). If \(\cP[x,c_1+c_2] = \cQ[y,c_1+c_2]\) then \((f\cP)[x,c_2] = (f\cQ)[y,c_2]\), by Lemma \ref{lem:increasing derivation}. Applying \(g\), we have \((gf\cP)[x] = (gf\cQ)[y]\), so \(g \circ f\) has derivation radius \(c_1+c_2\).
\end{proof}

LD maps of pattern spaces preserve the pre-order of local indistinguishability:

\begin{lemma}\label{lem:LD of LI}
If \(f \colon \Omega \LDmap \Omega'\) and \(\cU \LI \cV\) for \(\cP\), \(\cQ \in \Omega\), then \(f\cP \LI f\cQ\).
\end{lemma}

\begin{proof}
Let \(f\) have derivation radius \(c\) and take arbitrary \(x \in E\) and \(U \cpt E\). Then, since \(\cP \LI \cQ\), we have \(\cP[x,U+B_c] = \cQ[y,U+B_c]\) for some \(y \in E\). By Lemma \ref{lem:increasing derivation}, \((f\cP)[x,U] = (f\cQ)[y,U]\). So the arbitrary patch \((f\cP)[x,U]\) of \(f\cP\) occurs in \(f\cQ\), hence \(f\cP \LI f\cQ\).
\end{proof}

The above shows, in particular, that an LD map \(f \colon \Omega \LDmap \Omega'\) induces a map \(f \colon \mathrm{LI}(\Omega) \to \mathrm{LI}(\Omega')\) of the LI-classes contained in the pattern spaces. By definition, for any LD map \(f \colon \Omega \LDmap \Omega'\), we have that \(\cP \LD f(\cP)\), with the same derivation radius \(c\), for all \(\cP \in \Omega\). The below gives a kind of converse for hulls, defining an \textbf{induced LD map} from an LD of patterns:

\begin{proposition}\label{prop:induced LDs}
Let \(\cP \LD \cQ\) with derivation radius \(c\). Then there is a unique LD map of hulls \(f \colon \Omega_{\cP} \LDmap \Omega_{\cQ}\) (which also has derivation radius \(c\)) with \(f(\cP) = \cQ\). If \(\Omega_{\cQ}\) is Hausdorff, \(f\) is the unique factor map satisfying \(f(\cP) = \cQ\) and, if \(\cP\) is additionally FLC, then \(f \colon \Omega_{\cP} \to \Omega_{\cQ}\) is surjective.
\end{proposition}

\begin{proof}
We define \(f\) explicitly, as follows. Given \(\cU \in \Omega_{\cP}\), we let \(f(\cU) = \cU'\) be the pattern defined by \(\cU'[x] \coloneqq \cQ[z]\) where \(z = z_x \in E\) is any point for which \(\cP[z,c] = \cU[x,c]\). Since \(\cU \in \Omega_{\cP}\), equivalently \(\cU \LI \cP\), such a \(z \in E\) exists. Moreover, \(\cU'[x]\) does not depend on the choice of \(z\) for each \(x \in E\). Indeed, given a different choice \(z' \in E\) with \(\cP[z',c] = \cU[x,c] = \cP[z,c]\), since the derivation radius of \(\cP \LD \cQ\) is \(c\), we have \(\cQ[z'] = \cQ[z]\). Clearly, \(f(\cP) = \cQ\) (taking \(z_x = x\), above). Moreover, by construction, \(f\) satisfies the LD property since, by construction, if \(\cU[x,c] = \cV[y,c]\) then \((f\cU)[x] = (f\cV)[y]\), as both are equal to \(\cQ[z]\) for any \(z \in E\) with \(\cP[z,c] = \cU[x,c] = \cV[y,c]\). Thus, \(f \colon \Omega_{\cP} \LDmap \Omega_{\cQ}\), since it quickly follows from \(f(\cP) = \cQ\) and Lemma \ref{lem:increasing derivation} that \(f(\Omega_{\cP}) \subseteq \Omega_{\cQ}\).

To show that \(f\) is unique, let \(g \colon \Omega_{\cP} \LDmap \Omega_{\cQ}\) be another LD map, with derivation radius \(c'\), and with \(g(\cP) = \cQ\). For any \(\cU \in \Omega_{\cP}\) and \(r \geq 0\), let \(x \in E\) be such that \(\cU[\mathbf{0},r+C] = \cP[x,r+C]\), where \(C = \max\{c,c'\}\). Such an \(x \in E\) exists, since \(\cU \LI \cP\). By Lemma \ref{lem:increasing derivation} and applying both LD maps, \((f\cU)[\mathbf{0},r] = (f\cP)[x,r] = \cQ[x,r] = (g\cP)[x,r] = (g\cU)[\mathbf{0},r]\). Since \((f\cU)[\mathbf{0},r] = (g\cU)[\mathbf{0},r]\) for all \(r \geq 0\), we have \(f\cU = g\cU\), as required.

For any \(\cU \in \Omega_{\cP}\), we have \(\cU \in \overline{\cP+E}\) by Corollary \ref{cor:hulls are orbit closures}, so we may take a net (even a sequence, see comments below Lemma \ref{lem:uniformity well-defined}) \(\cP - x_n \to \cU\). By continuity and being a factor map, \(f(\cP - x_n) = \cQ - x_n \to f(\cU)\). Assuming \(\Omega_{\cQ}\) is Hausdorff, this limit point is unique, so \(f(\cU)\) is determined by the images \(f(\cP - x_n) = \cQ - x_n\) and thus \(f\) is uniquely defined over all of \(\Omega_{\cP}\) from just the factor map property and \(f(\cP) = \cQ\). If \(\cP\) is also FLC then \(\Omega_{\cP}\) is compact, by Corollary \ref{cor:cpt<=>FLC}, and thus so is \(f(\Omega_{\cP})\), by continuity. We have \(f(\Omega_{\cP}) \supseteq f(\cP+E) = \cQ+E\) is dense in \(\Omega_{\cQ} = \overline{\cQ+E}\) (Corollary \ref{cor:hulls are orbit closures}). Since compact subspaces of Hausdorff spaces are closed, we also have \(f(\Omega_{\cP})\) is closed, so \(f(\Omega_{\cP}) = \Omega_{\cQ}\), as required.
\end{proof}

\begin{remark}
We cannot drop \(\cP\) being FLC in general in deducing that the induced map is surjective. Indeed, consider the pattern \(\cP\) of the tiling of unit interval tiles, where each tile \([n,n+1]\), for \(n \in \Z\), is given unique label \(n\). Clearly, \(\Omega_{\cP} = \cP + E \cong E\) (indeed, this example is MLD to the `trivial' pattern of Example \ref{exp:trivial patterns} where all points are inequivalent). Let \(\cQ\) have the same geometric tiles but some arbitrary chosen labelling; for example, perhaps \(\cQ\) is defined by a Thue--Morse tiling. Then \(\cP \LD \cQ\) (the index \(n\) tile in \(\cQ\) may deduce its label from the label \(n\) in \(\cP\)), and \(\mathrm{im}(f) = \cQ + E\). So \(f \colon \Omega_{\cP} \LDmap \Omega_{\cQ}\) will only be surjective when \(\Omega_{\cQ}\) has a single orbit.
\end{remark}

Given an LD map between pattern spaces and an affine automorphism, we get an induced LD map of the distorted pattern spaces:

\begin{lemma}\label{lem:distorted LD}
Let \(f \colon \Omega \LDmap \Omega'\) and \(T \colon E \to E\) be an affine automorphism. Then the `distortion' \(f_T \coloneqq T \circ f \circ T^{-1} \colon T\Omega \LDmap T\Omega'\) of \(f\) is also an LD map. Here, \(T\) and \(T^{-1}\) also denote the maps between pattern spaces, see Proposition \ref{prop:pattern spaces are DS}.
\end{lemma}

\begin{proof}
Let \(f\) have derivation radius \(c\) and write \(T(x) = L(x)+z\) for \(L\) linear. Let \(c'\) be such that \(B_{c'} \supseteq L(B_c)\). We claim that \(f_T\) has derivation radius \(c'\).

Indeed, suppose that \((T\cP)[x,c'] = (T\cQ)[y,c']\), where \(T\cP\), \(T\cQ \in T\Omega\). Thus, \((T\cP)[x,LB_c] = (T\cQ)[y,LB_c]\). Equivalently, by Lemma \ref{lem:pattern properties}, \(\cP[T^{-1}x,B_c] = \cQ[T^{-1}y,B_c]\). Applying the local derivation, \((f\cP)[T^{-1}x] = (f\cQ)[T^{-1}y]\). Applying \(T\), we have \((Tf\cP)[x] = (Tf\cQ)[y]\), that is,
\[
(T \circ f \circ T^{-1}(T\cP))[x] = (T \circ f \circ T^{-1}(T\cQ))[y] .
\]
Since \(T\cP\), \(T\cQ \in T\Omega\) and \(x\), \(y \in E\) were arbitrary, we have that \(f_T \coloneqq TfT^{-1}\) is LD.
\end{proof}

The following shows that the operation of taking a pattern to its translational hull is functorial and commutes with applying affine transformations:

\begin{proposition}\label{prop:LD functor}
Taking induced LD maps of local derivations is functorial, in the following sense: the induced map of \(\cP \LD \cP\) is the identity map \(\mathrm{id} \colon \Omega_{\cP} \LDmap \Omega_{\cP}\), and if \(\cP \LD \cQ\) and \(\cQ \LD \cR\), with induced maps \(f\) and \(g\), respectively, then the induced map of \(\cP \LD \cR\) is \(g \circ f\). Thus, if \(\cP \LD \cQ\) and \(\cQ \LD \cP\), then their induced maps are homeomorphisms and inverses of each other.

Given \(\cP \LD \cQ\), with induced map \(f\), and an affine automorphism \(T \colon E \to E\), we have \(T\cP \LD T\cQ\) has induced map \(f_T\) making the diagram below commute
\[
\begin{codi}
\obj {
|(a)| \Omega_{\cP}  & |(b)| \Omega_{\cQ} \\[-1.5em]
|(c)| \Omega_{T\cP} & |(d)| \Omega_{T\cQ} \\
};
\mor a ["\mathrm{LD}"]["f",swap]  :-> b;
\mor c ["\mathrm{LD}"]["f_T",swap]:-> d;
\mor a T:-> c;
\mor b T:-> d;
\end{codi}
\]
\end{proposition}

\begin{proof}
We have that \(\mathrm{id} \colon \Omega_{\cP} \LDmap \Omega_{\cP}\) is LD with \(\mathrm{id}(\cP) = \cP\). By uniqueness (Proposition \ref{prop:induced LDs}), this is also the induced LD map of \(\cP \LD \cP\). Similarly, given \(\cP \LD \cQ\) and \(\cQ \LD \cR\) inducing \(f\) and \(g\), respectively, we have \((g \circ f)(\cP) = \cR\). Since \(g \circ f\) is an LD map (Lemma \ref{lem:LD category}), again by uniqueness \(g \circ f\) is the LD map induced by \(\cP \LD \cR\).

Now let \(\cP \LD \cQ\), with induced map \(f\), and \(\cQ \LD \cR\), with induced map \(g\). Then, by the above, \(g \circ f \colon \Omega_{\cP} \to \Omega_{\cP}\) is the induced map of \(\cP \LD \cP\), which also by the above is the identity map. Similarly, \(f \circ g = \mathrm{id}\) so \(f\) and \(g = f^{-1}\) are homeomorphisms.

Given \(\cP \LD \cQ\) and affine automorphism \(T\), we have (from Lemma \ref{lem:distorted LD}) that \(T\cP \LD T \cQ\) and \(f_T \coloneqq T \circ f \circ T^{-1}\) is LD and, by definition, the diagram commutes. We have \(f_T(T\cP) = (T \circ f \circ T^{-1})(T\cP) = T \circ f(\cP) = T\cQ\). Since there is a unique LD map \(g \colon \Omega_{T\cP} \LDmap \Omega_{T\cQ}\) with \(g(T\cP) = T\cQ\), we must have that \(f_T\) is the map induced by \(T\cP \LD T\cQ\), by Proposition \ref{prop:induced LDs}.
\end{proof}

If \(f\) is bijective, with \(f^{-1}\) also an LD map, then we say that \(\Omega\) and \(\Omega'\) are \textbf{MLD}. The result below states that, in the compact Hausdorff case, the LD property of the inverse is always automatic. Its proof is very simple in the return discrete setting but needs a bit more work in the full generality here. It is of particular relevance later, when we consider a notion of being substitutional using LD maps; when subdivision is bijective, this shows that its inverse (`composition') is also a local derivation map:

\begin{proposition} \label{prop:inverse of LD is LD}
If \(f \colon \Omega \LDmap \Omega'\) is bijective, \(\Omega\) is FLC and \(\Omega'\) is Hausdorff (equivalently, \(\Omega'\) is well-separated) then \(g = f^{-1}\) is also an LD map, so \(\Omega\) and \(\Omega'\) are MLD.
\end{proposition}

\begin{proof}
Since \(\Omega\) is FLC, it is compact. Thus, since \(f\) is bijective and continuous (Lemma \ref{lem:LD is factor map}), \(\Omega' = f(\Omega)\) is also compact, thus FLC. This verifies that \(\Omega'\) is Hausdorff if and only if it is well-separated, by Proposition \ref{prop:well-sep <=> HD}. In this case, \(f\) is a continuous map from a compact space to a Hausdorff space, so \(\Omega'\) is also compact Hausdorff and \(f\) has (uniformly) continuous inverse \(g\). We need to show that \(g\) is an LD map.

We claim that, for sufficiently large \(n\), we have the following: for all \(\cP \in \Omega\), if \((f\cP)[\mathbf{0},n] = (f\cP)[b,n]\) for some \(b \in B_1\) then \(\cP[\mathbf{0}] = \cP[b]\). Indeed, suppose otherwise, so that for arbitrarily large \(n \in \N\) we may find \(\cP_n \in \Omega\) and \(b_n \in B_1\) with \(\cQ_n[\mathbf{0},n] = \cQ_n[b_n,n]\) but \(\cP_n[\mathbf{0}] \neq \cP_n[b_n]\), where we denote \(\cQ_n \coloneqq f(\cP_n)\). Since \(\Omega\) and \(\Omega'\) are compact, we may find a subsequence of the above \(n \in \N\) for which \((\cP_n,\cQ_n) \to (\cP,\cQ) \in \Omega \times \Omega'\).

Let \(N \in \N\) be such that, for all \(p \in \lang_{B_1}(\cQ)\), there exists some \(x \in B_N\) with \(\cQ[x,1] = p\); such \(N\) exists since \(\cQ\) is FLC, by Proposition \ref{prop:FLC and repetitive inherited}. Take \(n \geq N+1\) in the above subsequence for which \(\cP_n[z,2] = \cP[\mathbf{0},2]\) and \(\cQ_n[z',N+1] = \cQ[\mathbf{0},N+1]\) for \(z\), \(z' \in B_1\). Such an \(n \in \N\) and \(z\), \(z' \in B_1\) exist by convergence of the subsequences to \(\cP\) and \(\cQ\). From \(\cQ_n[\mathbf{0},n] = \cQ_n[b,n]\), where we denote \(b \coloneqq b_n \in B_1\), and shifting by \(z' \in B_1\), we have \(\cQ_n[z',n-1] = \cQ_n[z'+b,n-1]\). Restricting this to radius \(N \leq n-1\) and combining with \(\cQ[b,N] = \cQ_n[z'+b,N]\) (obtained from shifting \(\cQ_n[z',N+1] = \cQ[\mathbf{0},N+1]\) by \(b \in B_1\)) and our definition of \(z'\),
\[
\cQ[b,N] = \cQ_n[z'+b,N] = \cQ_n[z',N] = \cQ[\mathbf{0},N] .
\]
It follows from Lemma \ref{lem:close repeat => period} that \(b \in \cK_{\cQ}\). We claim that \(b \in \cK_{\cP}\) too. Indeed, from continuity of \(f\) (and \(\Omega'\) being Hausdorff), \(f(\cP) = f \lim \cP_n = \lim f \cP_n = \lim \cQ_n = \cQ\), where the limit is over the subsequence. If \(b \notin \cK_{\cP}\) then \(\cP \neq \cP + b\) yet \(f(\cP + b) = f(\cP) + b = \cQ + b = \cQ\), contradicting injectivity of \(f\).

Since \(b \in \cK_{\cP}\) and \(b \in B_1\), from \(\cP_n[z,2] = \cP[\mathbf{0},2]\) we obtain:
\[
\cP_n[z + b,1] = \cP[b,1] = \cP[\mathbf{0},1] = \cP_n[z,1] .
\]
Shifting the above by \(-z \in B_1\), we find \(\cP_n[b,0] = \cP_n[\mathbf{0},0]\), contradicting our assumption that \(\cP_n[\mathbf{0}] \neq \cP_n[b]\). We have thus proved that, for some \(n \in \N\) (which we now fix), for all \(\cP \in \Omega\) and \(b \in B_1\), if \((f\cP)[\mathbf{0},n] = (f\cP)[b,n]\) then \(\cP[\mathbf{0}] = \cP[b]\).

Since \(f\) is a factor map then obviously \(g = f^{-1}\) is too. Since \(g\) is uniformly continuous, we may choose a sufficiently large \(M \geq n\) so that, for all \(\cQ_1\), \(\cQ_2 \in \Omega'\) and \(\cP_i \coloneqq g(\cQ_i)\), if \(\cQ_1[\mathbf{0},M] = \cQ_2[\mathbf{0},M]\) then \(\cP_1[b,n+c] = \cP_2[\mathbf{0},n+c]\) for some \(b \in B_1\), where \(n \in \N\) is as above. We claim that \(g\) has derivation radius \(M\). Indeed, assume that \(\cQ_1[\mathbf{0},M] = \cQ_2[\mathbf{0},M]\). Applying \(f\) (and using Lemma \ref{lem:increasing derivation}) to \(\cP_1[b,n+c] = \cP_2[\mathbf{0},n+c]\), we have \(\cQ_1[b,n] = \cQ_2[\mathbf{0},n] = \cQ_1[\mathbf{0},n]\), where the latter is given by restricting \(\cQ_1[\mathbf{0},M] = \cQ_2[\mathbf{0},M]\) to radius \(n \leq M\). By our choice of \(n \in \N\) justified above, \(\cP_1[\mathbf{0}] = \cP_1[b]\). Combining this with \(\cP_1[b,n+c] = \cP_2[\mathbf{0},n+c]\), it follows that \(\cP_1[\mathbf{0},0] = \cP_1[b,0] = \cP_2[\mathbf{0},0]\), so \(\cP_1[\mathbf{0}] = \cP_2[\mathbf{0}]\). Thus the LD property of Definition \ref{def:LD map} holds for \(x = y = \mathbf{0}\), so the full LD property holds, by Lemma \ref{lem:LD is factor map}.
\end{proof}

The above specialises to the below correspondence result, relating LDs (of individual patterns) being MLDs, to the induced map of hulls being injective:

\begin{corollary}\label{cor:MLD<=>injective}
Suppose that \(\cP \LD \cQ\). Denote its induced LD map by \(f \colon \Omega_{\cP} \LDmap \Omega_{\cQ}\). Then the following hold:
\begin{itemize}
	\item[(A)] if \(\cQ \LD \cP\) then its induced map is \(f^{-1} \colon \Omega_{\cQ} \LDmap \Omega_{\cP}\) with \(f\) necessarily bijective;
	\item[(B)] conversely, provided \(\cP\) is FLC, if \(\Omega_{\cQ}\) is Hausdorff (equivalently, \(\cQ\) is well-separated) and \(f\) is injective then \(\cQ \LD \cP\).
\end{itemize}
Hence, if \(\cP\) is FLC, \(\cQ\) is well-separated and \(\cP \LD \cQ\), the following are equivalent:
\begin{enumerate}
	\item the induced LD map \(f\) is injective;
	\item the induced LD map \(f\) is a homeomorphism;
	\item \(\cQ \LD \cP\), that is, \(\cP \MLD \cQ\).
\end{enumerate}
\end{corollary}

\begin{proof}
If \(\cQ \LD \cP\) then Proposition \ref{prop:LD functor} already shows that its induced LD map is the inverse of \(f\), establishing (A) and that (3) implies (1). For (B), suppose that \(\cP\) is FLC (equivalently, \(\Omega_{\cP}\) is FLC). As in the previous proof, \(\Omega_{\cQ}\) is Hausdorff if and only if it is well-separated (which in turn is equivalent to \(\cQ\) being well-separated by Lemma \ref{lem:separation for pattern spaces}). By Proposition \ref{prop:induced LDs}, the map \(f \colon \Omega_{\cP} \to \Omega_{\cQ}\) induced by \(\cP \LD \cQ\) is a surjective map from a compact space to a Hausdorff space, so is a homeomorphism if and only if it is injective, which also shows (1) implies (2). If it is injective, Proposition \ref{prop:inverse of LD is LD} applies and \(f^{-1}\) is an LD map. In particular, since \(f(\cP) = \cQ\), we have that \(\cQ \LD f^{-1}(\cQ) = \cP\), as required. This also establishes that (2) implies (3), as required.
\end{proof}

Note that \(\cP\) being FLC cannot be dropped in proving the second claim above:

\begin{example}
Consider the tiling of unit interval tiles in \(E = \R\), with origin over a boundary. We define a tiling, with associated pattern \(\cP\), by adding labels \(n \in \N\) to the tiles with support \([2^{n-1}-1,2^{n-1}]\), label \(1\) to all tiles with support \([-2^{n-1},-2^{n-1}+1]\) for \(n \in \N\) and otherwise tiles are labelled with a \(0\). Define another pattern \(\cQ\) similarly, except all tiles marked with \(n \neq 0\) are relabelled to \(1\). Then clearly \(\cP \LD \cQ\), with \(\cP\) and \(\cQ\) well-separated and \(\cQ\) FLC. Moreover, it is not hard to see that \(\Omega_{\cP} = (\cP + E) \cup (\mathcal{Z} + E) \cup (\mathcal{O} + E)\), where \(\mathcal{Z}\) is the pattern associated to the all \(0\)s tiling, and \(\mathcal{O}\) is a tiling with all \(0\) tiles except for a single \(1\) tile. We have an analogous description for \(\Omega_{\cQ} = (\cQ + E) \cup (\mathcal{Z} + E) \cup (\mathcal{O} + E)\) and \(f \colon \Omega_{\cP} \to \Omega_{\cQ}\) is bijective (and continuous). However, clearly we do not have \(\cQ \LD \cP\), since the tiles marked \(n \neq m\) are distinct in \(\cP\) but have the same patches of tiles to arbitrarily large distances in \(\cQ\), as \(n\), \(m \to \infty\). Alternatively, if \(\cQ \LD \cP\), then by Lemma \ref{prop:induced LDs} we would have a surjective and continuous map \(\Omega_{\cQ} \to \Omega_{\cP}\), which cannot exist since \(\Omega_{\cQ}\) is compact but \(\Omega_{\cP}\) is not.
\end{example}

\section{Substitutional pattern spaces}
\label{sec:Substitutional pattern spaces}

In this section we define a notion of a pattern being hierarchical with respect to a linear expansion and derive some important resulting properties. The definition aims to simultaneously cover standard definitions including stone and non-stone inflation tilings, substitution Delone (multi)sets \cite{LW03} and more. As noted in the introduction, the definition here stems from the collaboration \cite{HKW25} and is a slight relaxation of being pseudo self-affine \cite{Sol07} (or LIDS \cite{BG13}). We will also define a notion of a pattern \emph{space} being substitutional. Whilst this latter concept was initially motivated by the one for individual patterns, it is more economical to start with the pattern space version:

\subsection{Substitutional pattern spaces}

Throughout, \(L\) will denote a linear automorphism \(L \colon E \to E\). For our main results it will need to be expansive (see Definition \ref{def:expansive}) and our pattern spaces will need to be compact Hausdorff, although in this section we do not make these running assumptions.

\begin{definition}\label{def:L-sub pattern space}
Let \(\Omega\) be a pattern space and \(L \colon E \to E\) be a linear automorphism. We say that \(\Omega\) is \textbf{\(L\)-substitutional}, or \textbf{\(L\)-sub}, for short, when it is equipped with a surjective LD map \(S \colon L\Omega \LDmap \Omega\), called \textbf{subdivision}. We then define \textbf{substitution} to be the map \(\sub \coloneq S \circ L \colon \Omega \to \Omega\). If \(L\) is also expansive, we call \(\Omega\) \textbf{expansive \(L\)-sub}.
\end{definition}

Technically, being \(L\)-sub is not a property of \(\Omega\), it is the declaration that some subdivision map has been chosen for it, but we will only ever consider one such map so this should not cause confusion. The idea of this definition is simple: the map \(S\) should be thought of as like the tile replacement (or `subdivision') rule that replaces inflated (by \(L\)) tiles with tiles of the original size, see Figure \ref{fig:pen_sub}. There are several benefits to this broader concept, though; in particular it applies to general pattern spaces and not just ones defined by tilings, and also allows one to consider different substitution-invariant pattern spaces to just those defined by generating a language by iteration of substitution on an initial set of prototiles.

The substitution map \(\sigma = S \circ L \colon \Omega \to \Omega\), as usual in the geometric setting, is the composition `inflate, subdivide', where the homeomorphism \(L \colon \Omega \to L\Omega\) is given by \(\cP \mapsto L\cP\). In particular, surjectivity of \(S\) is equivalent to surjectivity of \(\sub\). This requirement is quite natural: it asks that all elements \(\cP \in \Omega\) are properly hierarchical, in that they may always be composed into `superpatterns' \(L\cP'\), where \(S(L\cP') = \cP\) and \(\cP' \in \Omega\) too. If we drop the surjectivity requirement, then later recognisability results will still at least apply to the subset of `hierarchical' elements of \(\Omega\), as we will see shortly.

\begin{lemma}\label{lem:sub preserves LI}
Let \(S \colon L\Omega \LDmap \Omega\), \(\sub = S \circ L\) and \(\cP\), \(\cQ \in \Omega\) with \(\cP \LI \cQ\). Then \(\sub(\cP) \LI \sub(\cQ)\). In particular, \(\sub\) induces a well-defined map on LI-classes.
\end{lemma}

\begin{proof}
We have \(\cP \LI \cQ\) if and only if \(L \cP \LI L \cQ\) (see Lemma \ref{lem:distortions of languages}). Thus, \(S(L\cP) \LI S(L\cQ)\), by Lemma \ref{lem:LD of LI}, as required.
\end{proof}

Note that, in the below, \(f^{-1}\) is automatically LD when the pattern spaces are compact Hausdorff, by Proposition \ref{prop:inverse of LD is LD}.

\begin{proposition}\label{prop:MLD invariance of L-sub}
Suppose that \(\Omega\) is an \(L\)-sub pattern space, with substitution \(\sub\), and that \(\Omega'\) is MLD to \(\Omega\), say with \(f \colon \Omega \LDmap \Omega'\) a bijective LD map, with LD inverse. Then \(\Omega'\) is also \(L\)-sub, and naturally so, in the sense that it has substitution \(\sub'\) satisfying \(f \circ \sub = \sub' \circ f\).
\end{proposition}

\begin{proof}
The LD map \(f\) defines also the bijective LD map \(f_L = L \circ f \circ L^{-1} \colon L\Omega \to L\Omega'\) and similarly \(f_L^{-1} = L \circ f^{-1} \circ L^{-1}\) is also an LD map. Since it is also surjective, and compositions of LD maps are LD (Lemma \ref{lem:LD category}), \(S' \coloneqq f \circ S \circ f_L^{-1} \colon L\Omega' \LDmap \Omega'\) is a surjective LD map, with substitution \(\sub' \coloneqq S' \circ L\), for which it is trivial to check that \(f \circ \sub = \sub' \circ f\).
\end{proof}

\begin{remark}
The above shows that being \(L\)-sub is `naturally' MLD invariant. One may wonder if this can be weakened to a one-directional LD version, that is, if given a substitution \(\sub\) on \(\Omega\) and surjective \(f \colon \Omega \LDmap \Omega'\), we may define a substitution \(\sub'\) on \(\Omega'\) for which \(f \circ \sub = \sub' \circ f\). The answer is no, in general, even for hulls of 1d tilings. Consider the constant length symbolic substitution on \(\{o,a,b,c\}\) given by \(o \mapsto aoa\), \(a \mapsto bbb\), \(b \mapsto ccc\) and \(c \mapsto aaa\). This has a fixed and palindromic word \(\cdots a^{27} c^9 b^3 aoa b^3 c^9 a^{27} \cdots\), with the single \(o\) at index \(0\). Also consider the word given by applying the recoding \(o\), \(a\), \(b \mapsto X\), \(c \mapsto Y\). This defines a surjective LD map \(f \colon \Omega \LDmap \Omega'\) between the suspensions of the orbit closures (which may be thought of as tiling spaces, of 1d tilings with unit interval labelled tiles). Note that \(\Omega\) contains periodic elements, corresponding to all \(a\)s, all \(b\)s and all \(c\)s tilings, which are mapped under \(f\) to the periodic all \(X\)s, all \(X\)s and all \(Y\)s tilings, respectively. Since substitution on \(\Omega\) permutes the three periodic orbits, there is no compatible substitution on \(\Omega'\) that intertwines \(f\).

We are not aware of any minimal examples exhibiting this sort of behaviour and it is possible there are none; in this direction, in the symbolic setting, it is at least known that Cantor factors of substitution subshifts are substitution subshifts or odometers with constant base, see \cite{Dur00}.
\end{remark}

\begin{definition}\label{def:hierarchy}
Let \(\Omega\) be a pattern space and \(S \colon L\Omega \LDmap \Omega\), possibly not surjective, with associated substitution \(\sub = S \circ L\). For \(\cP \in \Omega\), a \textbf{hierarchy} for \(\cP\) is a sequence \((\cP_n)_{n \in \N_0}\) in \(\Omega\) with each \(\sub(\cP_{n+1}) = \cP_n\) and \(\cP_0 = \cP\). When such a hierarchy exists, we call \(\cP\) \textbf{hierarchical}. We let \(\sH = \sH(\Omega)\) denote the subset of hierarchical elements of \(\Omega\).
\end{definition}

Of course, when \(\Omega\) is \(L\)-sub, every element has a hierarchy (by iteratively choosing pre-images), so \(\sH = \Omega\). More generally, in the compact Hausdorff case, we can at least always restrict a non-surjective substitution map to a surjective one on the hierarchical elements:

\begin{proposition}\label{prop:restriction to hierarchical elements}
Let \(\Omega\), \(S\) and \(\sub\) be as in Definition \ref{def:hierarchy}. Then \(\sub\) restricts to a surjective map \(\sub \colon \sH \to \sH\). Define the eventual range
\[
\mathscr{E} = \mathscr{E}(\sub) \coloneqq \bigcap_{n=0}^\infty \sub^n(\Omega) .
\]
Then \(\sH \subseteq \mathscr{E}\). If \(\Omega\) is compact Hausdorff, then \(\sH = \mathscr{E} \neq \emptyset\). Thus, in this case, subdivision restricts to a surjective map \(S \colon L \mathscr{H} \to \mathscr{H}\), making \(\mathscr{H} \subseteq \Omega\) a sub-pattern space that is compact Hausdorff and \(L\)-sub.
\end{proposition}

\begin{proof}
Let \(\cP \in \sH\) and choose some hierarchy \((\cP_i)_i\) for it. Then \(\sub \cP \in \sH\) too, since it has hierarchy \((\cP_{i-1})_i\), where we let \(\cP_{-1} = \sub(\cP)\), so \(\sub \colon \sH \to \sH\). This map is surjective, since \(\cP = \sub(\cP_1)\), where \(\cP_1 \in \sH\) as it has hierarchy \((\cP_{i+1})_i\). Moreover, since \(\cP = \sub^n(\cP_n) \in \sub^n(\Omega)\) for all \(n \in \N\), we have \(\sH \subseteq \mathscr{E}\).

So now suppose that \(\Omega\) is compact Hausdorff. We wish to show that \(\mathscr{H} \supseteq \mathscr{E}\), which follows from a standard inverse limit type argument. Note that each \(\sub^n(\Omega) \neq \emptyset\) is compact in a Hausdorff space, and is thus closed, so \(\mathscr{E} \subseteq \Omega\) is closed and hence also compact Hausdorff (and also non-empty, by the Cantor Intersection Theorem). So the same is true of \(X \coloneqq \mathscr{E}^{\N_0}\). Let \(\pi_n \colon X \to \mathscr{E}\) be the projection \(\pi_n((\cP_i)) \coloneqq \cP_n\) and define the continuous self-map \(\psi \colon X \to X\) by \(\psi(\cP_i)_i \coloneqq (\sub \cP_{i+1})_i\) (trivially, \(\sub(\mathscr{E}) \subseteq \mathscr{E}\), so this is well-defined). Let
\[
X_N \coloneqq \{(\cP_i)_i \in X \mid \sub(\cP_{i+1}) = \cP_i \text{ for all } 0 \leq i \leq N \} .
\]
Then each \(X_N \subseteq X\) is closed, since it is the intersection of pre-images of the (closed) diagonal \(\Delta \subseteq \mathscr{E} \times \mathscr{E}\) under the continuous maps \(x \mapsto (\pi_i (x), \pi_i(\psi x))\), for \(i = 0\), \ldots, \(N\). By definition, they are also nested, that is, \(X_N \supseteq X_{N+1}\). Given \(\cP \in \mathscr{E}\), if we further restrict each to \(X_N' = \{(\cP_i)_i \in X_N \mid \cP_0 = \cP\} = \pi_0^{-1}(\{\cP\}) \cap X_N \neq \emptyset\), then we get nested closed subsets \(X_N'\) where each \((\cP_i)_i\) has \(\cP_0 = \cP\). By the Cantor Intersection Theorem, there exists some \((\cP_i)_i \in \bigcap_{N=0}^\infty X_N'\). Thus, for each \(i \in \N_0\) we have \(\sub(\cP_{i+1}) = \cP_i\), so this defines a hierarchy for \(\cP\) and \(\cP \in \sH\), as required. Thus, \(\sH = \mathscr{E}\) and since we have already shown that \(\mathscr{E}\) is compact Hausdorff, and \(\sub \colon \sH \to \sH\) is surjective, we have that \(\sH\) is compact Hausdorff and \(L\)-sub, with subdivision \(S\) (restricted to \(\sH\)).
\end{proof}

We now restrict attention to \(L\)-sub pattern spaces i.e., with \(\sub\) surjective, where certain results for non-surjective substitutions will at least hold upon restriction to the hierarchical elements. Naturally, powers of substitution maps are still substitutions (in the sense defined in Definition \ref{def:L-sub pattern space}):

\begin{lemma}
If \(\Omega\) is \(L\)-sub then, for all \(n \in \N\), it is also \(L^n\)-sub with subdivision \(S' \coloneqq S_0 \circ S_1 \circ \cdots \circ S_n\), where \(S_i = L^i \circ S \circ L^{-i}\) is the distortion of subdivision \(S \colon L\Omega \LDmap \Omega\) to \(S_i \colon L^{i+1}\Omega \LDmap L^i\Omega\), as in Lemma \ref{lem:distorted LD}. Moreover, \(S' \circ L^n = \sub^n\).
\end{lemma}

\begin{proof}
The distorted subdivision maps \(S_i\) are LD maps, by Lemma \ref{lem:distorted LD}. Their composition is also LD (Lemma \ref{lem:LD category}). From \(S_i = L^i \circ S \circ L^{-i}\), each \(S_i\) is surjective because \(S\) is (and recall that \(L \colon \Omega \to L\Omega\) is a homeomorphism). Thus, \(S'\) must also be surjective. Moreover, iteratively using the identity \(S_i \circ L = L \circ S_{i-1}\), we have \(\sub^n = (S \circ L)^n = (S_0 \circ L)^n = S' \circ L^n\). Thus, \(\Omega\) is \(L^n\)-sub, with subdivision \(S'\) and substitution \(\sub^n\).
\end{proof}

\subsection{Languages for expansive \(L\)-sub pattern spaces}
\label{Languages for expansive L-sub pattern spaces}

When \(L\) is expansive, we fix a norm of \(E\) and \(\lambda > 1\) with \(\|L(x)\| \geq \lambda \|x\|\); the choice of norm, otherwise, does not affect results. It will be important to control agreement of patches fairly tightly, which may be pushed to larger and larger regions by substitution in the expansive case. A minor annoyance here is that, due to the derivation radius \(c\), some information near the boundary of a region is lost under each application. This can be addressed quite easily, using expansivity. Here, we choose to slightly modify our `level-\(n\) ball', using the following simple lemma. Throughout, given \(X \subseteq E\) and \(t \geq 0\), we let \(X_{+t}\) denote the set of points within \(t\) of \(X\), and \(X_{-t}\) the set of points with \(t\)-neighbourhoods wholly in \(X\). That is, \(X_{+t} = \{y \in E \mid \|y-x\| \leq t \text{ for some } x \in X\} = X + B_t\) and \(X_{-t} = \{x \in E \mid B(x,t) \subseteq X\}\). Trivially, \(X_{-t} \supseteq Y\) if and only if \(X \supseteq Y_{+t}\).

\begin{lemma}\label{lem:modified balls}
There exists \(\kappa \geq 0\) for which, for all \(X \subseteq E\), we have \((L(X_{+\kappa}))_{-c} \supseteq (L X)_{+\kappa}\).
\end{lemma}

\begin{proof}
Let \(B \coloneqq B_1\). The lemma holds if \(\kappa LB \supseteq (\kappa + c)B\). Indeed, in this case,
\[
L(X_{+\kappa}) = L(X+\kappa B) = LX + \kappa LB \supseteq LX + (\kappa + c)B = (LX + \kappa B) + cB = ((LX)_{+\kappa})_{+c} ,
\]
as required. To show \(\kappa LB \supseteq (\kappa + c)B\) for some \(\kappa \geq 0\), recall that \(LB \supseteq \lambda B\) where \(\lambda > 1\) (by expansivity), so \(\kappa LB \supseteq \kappa \lambda B \supseteq (\kappa + c)B\) provided \(\kappa \lambda \geq \kappa + c\), which holds for \(\kappa \geq \frac{c}{\lambda - 1} \geq 0\).
\end{proof}

\begin{notation}\label{not:balls}
When an expansive \(L\) and subdivision map \(S\) are understood, we denote the derivation radius of \(S\) by \(c\) and let \(\kappa \geq 0\) be as in Lemma \ref{lem:modified balls}. We then define \(V^k \coloneqq (L^k B)_{+\kappa}\), where \(B = B_1\). Thus, \((L(V^k))_{-c} \supseteq V^{k+1}\) since, applying the lemma to \(X = L^kB\),
\[
(L(V^k))_{-c} = (L(X_{+\kappa}))_{-c} \supseteq (LX)_{+\kappa} = (L^{k+1} B)_{+\kappa} = V^{k+1} .
\]
If \(c = 0\) (for instance, when \(S\) corresponds to the subdivision map of a tile stone inflation) we may take \(\kappa = 0\) and then \(V^k = L^k B\). Generally, we just add a small `fringe' to these inflated balls.
\end{notation}

The following, using expansivity of \(L\), explicitly details the obvious fact that substitution increases the size of agreement between (sufficiently large) patches. In fact, it still makes sense if `substitution' is from a different pattern space:

\begin{lemma}\label{lem:bigger agreement after substitution}
Let \(\Omega\), \(\Omega'\) be pattern spaces and \(S \colon L\Omega' \LDmap \Omega\). As usual, define \(\sub \coloneqq S \circ L \colon \Omega' \to \Omega\). Then, for all \(\cP\), \(\cQ \in \Omega'\), we have that
\[
\mathrm{if} \ \cP[x,U] = \cQ[y,U] \ \mathrm{ then } \ (\sub \cP)[Lx,(LU)_{-c}] = \newline (\sub \cQ)[Ly,(LU)_{-c}] .
\]
In particular, when \(L\) is expansive, for all \(r \geq c/\lambda\) we have
\[
\mathrm{if} \ \cP[x,r] = \cQ[y,r] \ \mathrm{ then } \ (\sub \cP)[Lx,\lambda r - c] = (\sub \cQ)[Ly,\lambda r - c] 
\]
and, for all \(k \in \N_0\),
\[
\mathrm{if} \  \cP[x,V^k] = \cQ[y,V^k] \ \mathrm{ then } \ (\sub \cP)[Lx,V^{k+1}] = (\sub \cQ)[Ly,V^{k+1}] .
\]
\end{lemma}

\begin{proof}
By Lemma \ref{lem:pattern properties} (2), given \(\cP[x,U] = \cQ[y,U]\) we have that \((L\cP)[Lx,LU] = (L\cQ)[Ly,LU]\). Given \(z \in (LU)_{-c}\), we have \(B(z,c) \subseteq LU\) and thus \((L\cP)[Lx,B(z,c)] = (L\cQ)[Ly,B(z,c)]\), by Lemma \ref{lem:pattern properties} (1), and by (2) we have \((L \cP)[Lx + z,B_c] = (L \cQ)[Ly + z,B_c]\). Applying \(S\), we thus have \((S(L \cP))[Lx+z] = (S(L \cQ))[Ly+z]\), that is, \((\sub \cP)[Lx+z] = (\sub \cQ)[Ly+z]\). Since this holds for all \(z \in (LU)_{-c}\), we have \((\sub \cP)[Lx,(LU)_{-c}] = (\sub \cQ)[Ly,(LU)_{-c}]\), as required.

Applying the above to \(U = B_r\) derives the second claim, using that \((LB_r)_{-c} \supseteq (\lambda B_r)_{-c} = (B_{\lambda r})_{-c} = B_{\lambda r - c}\) for \(\lambda r - c \geq 0\) i.e., \(r \geq c/\lambda\). The third claim follows from applying the first to \(U = V^k\) since \(L(V^k)_{-c} \supseteq V^{k+1}\), see Notation \ref{not:balls}.
\end{proof}

We already saw, in Lemma \ref{lem:sub preserves LI}, that substitution preserves LI-classes. To get finer control on the effect of substitution on languages, we introduce the following notation:

\begin{notation}\label{not:language levels}
When \(L\) is expansive, for \(n \in \N_0\) we denote \(\lang_n(\Omega) = \lang_{V^n}(\Omega)\) and similarly \(\lang_n(\cP) = \lang_{V^n}(\cP)\) (see Notation \ref{not:balls}). Recalling \(\mathscr{C}_U(\Omega) = \{\lang_U(\cP) \mid \cP \in \Omega\}\) from Lemma \ref{lem:FLC => finitely many configs}, we similarly denote \(\mathscr{C}_n(\Omega) \coloneqq \mathscr{C}_{V^n}(\Omega)\).
\end{notation}

The next result follows quickly from Lemma \ref{lem:bigger agreement after substitution}.

\begin{lemma}\label{lem:induced substitution on patches}
Let \(\Omega\), \(\Omega'\), \(S\), \(L\) and \(\sub\) be as in Lemma \ref{lem:bigger agreement after substitution}, with \(L\) expansive. Given \(p \in \lang_n(\Omega')\), define \(\sub p \in \lang_{n+1}(\Omega)\) by \(\sub p \coloneqq (\sub \cP)[Lx,V^{n+1}]\), where we take any \(\cP \in \Omega'\) and \(x \in E\) with \(p = \cP[x,V^n]\). Then this is a well-defined function \(\sub \colon \lang_n(\Omega') \to \lang_{n+1}(\Omega)\).

Similarly, defining \(\sub(\lang_n(\cP)) \coloneqq \lang_{n+1}(\sub \cP)\) induces a well-defined function \(\sub \colon \mathscr{C}_n(\Omega') \to \mathscr{C}_{n+1}(\Omega)\). If \(\sub\) is surjective, then both \(\sub \colon \lang_n(\Omega') \to \lang_{n+1}(\Omega)\) and \(\sub \colon \mathscr{C}_n(\Omega') \to \mathscr{C}_{n+1}(\Omega)\) are surjective.
\end{lemma}

\begin{proof}
Suppose that \(p \in \lang_n(\Omega')\) and \(p = \cP[x,V^n] = \cQ[y,V^n]\). By Lemma \ref{lem:bigger agreement after substitution}, \((\sub \cP)[Lx,V^{n+1}] = (\sub \cQ)[Ly,V^{n+1}]\) so \(\sub \colon \lang_n(\Omega') \to \lang_{n+1}(\Omega)\) is well-defined. Moreover, given any \(C = \lang_n(\cP) \in \mathscr{C}_n(\Omega)\), this shows that \(\sub(C) = \{\sub p \mid p \in \lang_n(\cP)\}\) does not depend on the representative \(\cP\) we take for \(C\). We also have \(\sub(\lang_n(\cP)) = \lang_{n+1}(\sub \cP) \in \mathscr{C}_{n+1}(\Omega)\), since given any \(p' \in \lang_{n+1}(\sub \cP)\), say \(p' = (\sub \cP)[y,V_{n+1}]\), we have \(p' = \sub p\) with \(p = \cP[L^{-1} y, V^n]\). So \(\sub \colon \mathscr{C}_n(\Omega') \to \mathscr{C}_{n+1}(\Omega)\) is well-defined.

Supposing that \(\sub\) is surjective, given any \(p = \cP[x,V^{n+1}] \in \lang_{n+1}(\Omega)\), we have that \(p = \sub p'\) for any \(p' = \cP'[L^{-1}x,V^n]\), where \(\cP' \in \sub^{-1}(\Omega')\), showing that each \(\sub \colon \mathscr{C}_n(\Omega') \to \mathscr{C}_{n+1}(\Omega)\) is surjective. Similarly, given \(C = \lang_{n+1}(\cP) \in \mathscr{C}_{n+1}(\Omega)\), by definition \(C = \sub C'\) for \(C' = \lang_n(\cP')\), showing surjectivity from \(\mathscr{C}_n(\Omega')\).
\end{proof}

Crucially, recall that the sets \(\mathscr{C}_n(\Omega)\) are finite when \(\Omega\) is FLC (Lemma \ref{lem:FLC => finitely many configs}), which we will soon exploit. By restriction of patches, there are natural, surjective maps \(\pi \colon \lang_{n+1}(\Omega) \to \lang_n(\Omega)\). That is, given \(p = \cP[x,V^{n+1}] \in \lang_{n+1}(\Omega)\), we let \(\pi p \coloneqq \cP[x,V^n]\), which is clearly well-defined. Similarly, by restriction of all patches at some level of a language, we obtain surjections \(\pi \colon \mathscr{C}_{n+1}(\Omega) \to \mathscr{C}_n(\Omega)\).

Thus, when \(\Omega\) is expansive \(L\)-sub (so we take \(\Omega' = \Omega\) in the above), we obtain surjective self-maps \(\pi \circ \sub \colon \lang_n(\Omega) \to \lang_n(\Omega)\) and \(\pi \circ \sub \colon \mathscr{C}_n(\Omega) \to \mathscr{C}_n(\Omega)\). Continuing a slight abuse of notation, we also denote these by \(\sub\). Unpacking some notation here, given \(C = \lang_n(\cP) \in \mathscr{C}_n(\Omega)\), its substitution is simply \(\sub(C) = \lang_n(\sub \cP)\).

\begin{theorem}\label{thm:finite languages of hierarchical}
Let \(\Omega\) be compact and expansive \(L\)-sub. Then the set \(\{\lang(\cP) \mid \cP \in \Omega\}\) of LI-classes of \(\Omega\) is finite.
\end{theorem}

\begin{proof}
Let \(\mathscr{C} \coloneqq \mathscr{C}_0(\Omega)\). As explained above, substitution induces a surjection \(\sub \colon \mathscr{C} \to \mathscr{C}\). It is therefore bijective, since \(\mathscr{C}\) is finite by compactness (equivalently, FLC) of \(\Omega\) and Lemma \ref{lem:FLC => finitely many configs}, so \(\sub^N = \id\) on \(\mathscr{C}\) for some \(N \in \N\).

Let \(\cP \in \Omega\), say with hierarchy \((\cP_n)_n\). We claim that \(\lang(\cP)\) is determined by \(\lang_0(\cP) \in \mathscr{C}\), in which case the result follows from \(\# \mathscr{C} < \infty\). Indeed, suppose that \(\lang_0(\cP) = \lang_0(\cQ)\) for some other \(\cQ \in \Omega\), say with hierarchy \((\cQ_n)_n\). Take an arbitrary \(p \in \lang(\cP)\), say \(p = \cP[x,U]\). Let \(\ell \in \N\) be such that \(V^{\ell N} \supseteq U\). Since \(\sub^{\ell N}(\cP_{\ell N}) = \cP\) (and similarly for \(\cQ\)) and \(\sub^N = \id\) on \(\mathscr{C}\), we have that \(\lang_0(\cP_{\ell N}) = \lang_0(\sub^{\ell N} \cP) = \sub^{\ell N} \lang_0(\cP) = \lang_0(\cP) = \lang_0(\cQ) = \sub^{\ell N}\lang_0(\cQ) = \lang_0(\sub^{\ell N} \cQ) = \lang_0(\cQ_{\ell N})\). Thus, there exists some \(y \in E\) with \(\cQ_{\ell N}[y,V^0] = \cP_{\ell N}[L^{-\ell N} x,V^0]\). Substituting \(\ell N\) times (for instance, with Lemma \ref{lem:bigger agreement after substitution}) gives \(\cQ[L^{\ell N} y,V^{\ell N}] = \cP[x,V^{\ell N}]\). Restricting these patches to \(U\), we have \(\cQ[L^{\ell N} y,U] = p\), so that \(p \in \lang(\cQ)\). Thus, \(\lang(\cP) \subseteq \lang(\cQ)\), with an identical argument for the reserve.
\end{proof}

The following shows the importance of surjectivity of subdivision in the above result that there are always finitely many LI-classes, and the one after shows why surjectivity of substitution is generally needed for recognisability:

\begin{example}
Consider the pattern space \(\Omega\) of tilings of \(E = \R^1\) by unit interval tiles, labelled \(a\) and \(b\), where the number of \(b\) tiles is either \(0\), \(2^k\) for \(k \in \N_0\), or \(\infty\), with all such appearing in a connected block. So, symbolically, the tilings are of one of the types \(b^\infty a^\infty\), \(a^\infty b^n a^\infty\) (for \(n = 2^k\), with \(k \in \N_0\)), \(a^\infty b^\infty\) or \(b^\infty\). This defines a pattern space according to Definition \ref{def:hull}. These types may also be identified with the LI-classes. In particular, there are infinitely many LI-classes; tilings with a different number of \(b\) tiles are not LI.

Take the inflation \(L(x) = 2x\) and define \(S \colon L\Omega \LDmap \Omega\), in the obvious way, from the symbolic substitution \(a \mapsto aa\), \(b \mapsto bb\). All tilings of \(\Omega\) have pre-images, except those with a single \(b\) tile, so \(S\) (only just) fails to be surjective. Indeed, substitution maps LI-classes according to the map \(0 \mapsto 0\), \(1 \mapsto 2 \mapsto 4 \mapsto 8 \mapsto \cdots\), \(\text{left} \mapsto \text{left}\), \(\text{right} \mapsto \text{right}\) and \(\text{all} \mapsto \text{all}\). The hierarchical elements \(\sH \subset \Omega\) are given by all tilings of types \(a^\infty\), \(b^\infty\), \(a^\infty b^\infty\) or \(b^\infty a^\infty\), on which substitution is now surjective.
\end{example}

\begin{example}
Let \(\Omega\), similarly to above, be defined from symbolic sequences, this time over alphabet \(\{a,b,c\}\). Consider all such tilings associated to symbolic sequences of the form \(a^\infty\), \(a^\infty b a^\infty\) and \(a^\infty c a^\infty\) (i.e., with all \(a\) tiles except perhaps one \(b\) or one \(c\) tile). Let \(L(x) = 2x\) and consider the LD map \(S \colon L\Omega \LDmap \Omega\) defined from the symbolic substitution \(a \mapsto aa\), \(b \mapsto ba\) and \(c \mapsto ba\). Tilings of the form \(a^\infty b a^\infty\), whilst non-periodic, have two pre-images so recognisability fails. Of course, tilings of the form \(a^\infty c a^\infty\) have no pre-images and \(S\) is not surjective. If one restricts to the hierarchical elements \(\sH \subset \Omega\), which are precisely those of the form \(a^\infty\) and \(a^\infty b a^\infty\), then we obtain an \(L\)-sub pattern space (that is, substitution becomes surjective) and the non-periodic points now have unique pre-images.
\end{example}

\begin{corollary}\label{cor:power fixing LI-classes}
If \(\Omega\) is compact and expansive \(L\)-sub, \(\sub\) is bijective on the set of LI-classes, which is finite. In particular, there exists some \(N \in \N\) so that \(\sub^N\) fixes LI-classes.
\end{corollary}

\begin{proof}
Since \(\sub \colon \Omega \to \Omega\) is surjective, so must be the associated map on LI-classes. Since the set of LI classes is finite by the previous theorem, \(\sub\) is must also be injective on this finite set. Thus, there exists \(N \in \N\) so that \(\sub^N\) acts as the identity on LI-classes.
\end{proof}

The following will be useful when considering hulls in the next section.

\begin{lemma}\label{lem:expansive on sub-PS surjective=>equal PS}
Let \(\Omega\), \(\Omega'\), \(S\), \(L\) and \(\sub\) be as in Lemma \ref{lem:bigger agreement after substitution}. Suppose, further, that \(S \colon L\Omega' \to \Omega\) is surjective, \(L\) is expansive, \(\Omega\) is compact and \(\Omega' \subseteq \Omega\). Then \(\Omega = \Omega'\).
\end{lemma}

\begin{proof}
Since \(\Omega\) is compact, each \(\mathscr{C}_n(\Omega)\) is finite by Lemma \ref{lem:FLC => finitely many configs}. Clearly \(\mathscr{C}_n(\Omega') \subseteq \mathscr{C}_n(\Omega)\) since \(\Omega' \subseteq \Omega\). In fact, these sets must be equal, since \(\sub \colon \Omega' \to \Omega\) is surjective and thus so is \(\sub \colon \mathscr{C}_n(\Omega') \to \mathscr{C}_n(\Omega)\), by Lemma \ref{lem:induced substitution on patches}. Thus, for each \(n \in \N_0\), we have
\[
\lang_n(\Omega') = \bigcup_{\cP \in \Omega'} \lang_n(\cP) = \bigcup_{C \in \mathscr{C}_n(\Omega')} \left( \bigcup_{p \in C} p \right) = \bigcup_{C \in \mathscr{C}_n(\Omega)} \left( \bigcup_{p \in C} p \right) = \bigcup_{\cP \in \Omega} \lang_n(\cP) = \lang_n(\Omega) .
\]
It is then clear that \(\lang(\Omega') = \lang(\Omega)\) (every patch, of some shape \(U\), is the restriction of some patch of shape \(V^n \supseteq U\)). Since a pattern space is determined by its language, \(\Omega = \Omega'\), as required.
\end{proof}

Another consequence of Theorem \ref{thm:finite languages of hierarchical}, useful for recognisability, is the following:

\begin{corollary}\label{lem:groups of periods FI}
Let \(\Omega\) be compact Hausdorff and expansive \(L\)-sub. Then, for all \(\cP \in \Omega\) and all \(\cP' \in \sub^{-1}(\cP)\), we have that \(L \cK_{\cP'} \leqslant \cK_{\cP}\) with finite index.
\end{corollary}

\begin{proof}
By Corollary \ref{cor:power fixing LI-classes}, some power \(\sub^N\) of substitution fixes LI-classes. Let \(\cQ \in \sub^{-(N-1)}(\cP')\). Then \(L^N \cQ \LD L\cP' \LD \cP\) (as we see by applying subdivision \(N-1\) times and then one extra time), so it follows that \(L^N \cK_{\cQ} \leqslant L\cK_{\cP'} \leqslant \cK_{\cP}\) (as LDs induce inclusions of groups of periods, see Lemma \ref{lem:periods under LI and LD}). Moreover, since \(\cP = \sub^N(\cQ) \LIs \cQ\), we have that \(\cK_{\cQ} = \cK_{\cP}\). Since \(\Omega\) is Hausdorff, \(\cK_{\cP}\) is topologically closed, so that \(L^N \cK_{\cQ} = L^N \cK_{\cP} \cong \cK_{\cP}\) must be finite index in \(\cK_{\cP}\), forcing the same for the intermediate subgroup \(L\cK_{\cP'}\).
\end{proof}

\subsection{Substitutional hulls and \(L\)-sub for patterns}
\label{sec:Substitutional hulls and L-sub for individual patterns}

In this section we fix some pattern \(\cP\) and consider the implications of its hull having a subdivision map.

\begin{proposition}\label{prop:L-sub hull properties}
Suppose that \(S \colon L\Omega_{\cP} \LDmap \Omega_{\cP}\), not necessarily surjective, with associated substitution \(\sub = S \circ L\). We have the following:
\begin{enumerate}
	\item If \(S\) is surjective, then \(\cP \LIs \sub(\cP)\), so \(\sub\) fixes the LI-class of \(\cP\).
	\item If \(\Omega_{\cP}\) is compact Hausdorff then \(S\) is surjective if and only if \(\sub^{-1}(\cP) \neq \emptyset\).
	\item If \(\Omega_{\cP}\) is compact and expansive \(L\)-sub, so that \(S\) is surjective, then \(\cP' \LIs \cP\) for all \(\cP' \in \sub^{-1}(\cP)\).
\end{enumerate}
\end{proposition}

\begin{proof}
Throughout, let \(\Omega \coloneqq \Omega_{\cP}\). Suppose that \(S\), equivalently \(\sub\), is surjective. Then there exists \(\cP' \in \Omega\) with \(\sub(\cP') = \cP\). Since \(\cP' \in \Omega\), equivalently \(\cP' \LI \cP\), by applying substitution we have \(\sub(\cP') = \cP \LI \sub(\cP)\) by Lemma \ref{lem:sub preserves LI}, so \(\cP \LIs \sub (\cP)\), proving (1).

Now suppose that \(\Omega\) is compact Hausdorff. Obviously it is necessary for there to exist some \(\cP' \in \sub^{-1}(\cP)\) for surjectivity. Given such a \(\cP'\), we have \(S(L\cP') = \cP\), thus \(L\cP' \LD \cP\). This induces a subdivision of hulls \(S' \colon L\Omega_{\cP'} \LDmap \Omega_{\cP}\) (see Proposition \ref{prop:induced LDs}) which must agree with \(S\) when restricted to \(L\Omega_{\cP'} \subseteq L\Omega_{\cP}\) (since \(S'(L\cP') = S(L\cP)\)), and be surjective (by the same proposition), so \(S\) is surjective, showing (2).

Finally, suppose that \(\Omega\) is compact (equivalently, \(\cP\) is FLC) and expansive \(L\)-sub. Then \(\Omega\) contains only finitely many LI-classes by Theorem \ref{thm:finite languages of hierarchical}. So there exists some \(N \in \N\) with \(\cP' \LIs \sub^N(\cP') = \sub^{N-1}(\cP) \LIs \cP\), where the latter follows from (1), proving (3).
\end{proof}

For a linear automorphism \(L\), if \(\Omega = \Omega_{\cP}\) is an \(L\)-sub hull then, by definition, \(S(L\cP') = \cP\) for some \(\cP' \in \Omega\). This means that \(L\cP' \LD \cP\) and \(\cP' \LI \cP\), which provides an elegant notion of an individual pattern being substitutional, purely in terms of the fundamental relations of local derivation and local indistinguishability:

\begin{definition}\label{def:L-sub}
A pattern \(\cP\) is called \textbf{\(L\)-substitutional}, or \textbf{\(L\)-sub}, for short, if there exists another pattern \(\cP' \LI \cP\) for which \(L\cP' \LD \cP\). The pattern \(\cP'\) is called a \textbf{predecessor} of \(\cP\). We say that \(\cP\) is \textbf{expansive \(L\)-sub} when \(L\) is expansive.
\end{definition}

\begin{remark}
It is easily checked that the property of being \(L\)-sub is invariant (and with a canonically defined predecessor) under local isomorphism and MLD equivalence. On the latter, see a similar statement for \(L\)-sub pattern spaces in Proposition \ref{prop:MLD invariance of L-sub}.
\end{remark}

\begin{remark}\label{rem:all elements L-sub for some power}
For an \(L\)-sub pattern space (even if it is a hull), we need not have \(\sub(\cQ) \LIs \cQ\) for all \(\cQ \in \Omega\), so it is not necessarily true that all elements of \(\Omega\) are \(L\)-sub. However, when \(\Omega\) is compact and expansive \(L\)-sub, we at least have that some power \(\sub^N\) of substitution fixes LI-classes, by Corollary \ref{cor:power fixing LI-classes}. Thus, every element of \(\cP \in \Omega\) is naturally \(L^N\)-sub with predecessor any \(\cP' \in \sub^{-N}(\cP)\), since then \(L^N \cP' \LD \cP\) and \(\cP' \LIs \cP\).
\end{remark}

As explained in the introduction, the \(L\)-sub definition is geometrically intuitive: a pattern \(\cP\) is considered \(L\)-sub when a pattern \(\cP'\) locally indistinguishable from it, but inflated by \(L\), may derive \(\cP\) by a local rule (see Figure \ref{fig:pen_sub}). And, indeed, the utility of this property was noticed before the notion of \(L\)-sub for a pattern space; we also note that it is an appropriate weakening of the pseudo-self affine property \cite{Sol07}. The results in the previous section allow us to show that the notions for hulls and for individual patterns are essentially equivalent. This is precisely formulated below, with the main result being that reasonable expansive \(L\)-patterns (FLC and well-separated ones) have hulls that are naturally \(L\)-sub. Note that, to keep our definition of \(L\)-sub as basic and applicable as possible, we did not even demand that the predecessor \(\cP' \LIs \cP\), merely that \(\cP' \LI \cP\). However, the below shows that the two are necessarily locally isomorphic given some standard conditions.

\begin{theorem}\label{thm:predecessor LI}
Suppose that \(S \colon L\Omega_{\cP} \LDmap \Omega_{\cP}\). Then \(\cP\) is \(L\)-sub with respect to any predecessor \(\cP' \in \sub^{-1}(\cP)\). When \(\Omega\) is compact Hausdorff, the existence of such an element is equivalent to surjectivity of \(S\). If, additionally, \(L\) is expansive, then necessarily \(\cP' \LIs \cP\).

Conversely, suppose that \(\cP\) is \(L\)-sub, with predecessor \(\cP'\). If \(\cP\) is FLC and well-separated, and \(L\) is expansive, then \(\Omega_{\cP}\) is \(L\)-sub, with respect to a uniquely defined and surjective LD map \(S \colon L\Omega_{\cP} \LDmap \Omega_{\cP}\) satisfying \(S(L\cP') = \cP\). Necessarily, \(\cP' \LIs \cP\).
\end{theorem}

\begin{proof}
If \(S \colon L\Omega_{\cP} \LDmap \Omega_{\cP}\) and \(\cP' \in \sub^{-1}(\cP)\) then \(L\cP' \LD \cP\) and \(\cP' \in \Omega_{\cP}\) i.e., \(\cP' \LI \cP\), so \(\cP\) is \(L\)-sub. The next pair of statements follow directly from Proposition \ref{prop:L-sub hull properties}.

So suppose we now start with some \(L\)-sub pattern \(\cP\), with predecessor \(\cP'\), so \(L\cP' \LD \cP\) with \(\cP' \LI \cP\), that is, \(\cP' \in \Omega_{\cP}\). The LD induces a unique LD map \(S \colon \Omega_{L\cP'} \LDmap \Omega_{\cP}\) with \(S(L\cP') = \cP\), by Proposition \ref{prop:induced LDs}. Assuming \(\cP\) is FLC and well-separated is equivalent (by Proposition \ref{prop:well-sep <=> HD}) to \(\Omega_{\cP}\) being compact Hausdorff. Then the closed subset \(\Omega_{\cP'} \subseteq \Omega_{\cP}\) must also be compact, thus so is \(\Omega_{L\cP'} = L \Omega_{\cP'} \cong \Omega_{\cP'}\), so the LD map is necessarily surjective, again by Proposition \ref{prop:well-sep <=> HD}. Finally, Lemma \ref{lem:expansive on sub-PS surjective=>equal PS} implies that \(\Omega_{\cP'} = \Omega_{\cP}\), equivalently \(\cP' \LIs \cP\), as required.
\end{proof}

The above result has some important consequences. Firstly, suppose we start with a compact Hausdorff pattern space \(\Omega\), expansive \(L\), a (not necessarily surjective) subdivision map \(S \colon L\Omega \to \Omega\) and induced substitution \(\sub = S \circ L\). By Proposition \ref{prop:restriction to hierarchical elements}, we can at least restrict this to a surjective subdivision/substitution map on the hierarchical elements \(\sH\). By Theorem \ref{thm:finite languages of hierarchical}, the set of LI-classes of \(\sH\) is finite, so by passing to a power \(\sub^N\) of substitution, we may assume that substitution preserves LI-classes. Then, by the above result, \(\sH\) is a finite union of \(L^N\)-sub hulls \(\Omega_{\cP} \subseteq \sH\) (where we may take the finite union over patterns \(\cP \in \sH\) with languages that are maximal).

If, instead, we start with an FLC, well-separated and expansive \(L\)-sub pattern \(\cP\), then necessarily any predecessor \(\cP' \LIs \cP\). Moreover, from this choice of \(\cP'\) we obtain an associated surjective subdivision map \(S \colon L\Omega_{\cP} \LDmap \Omega_{\cP}\). In particular, all elements \(\cP_i\) in a hierarchy over \(\cP\) are locally isomorphic. This need \emph{not} hold for other elements of \(\Omega_{\cP}\), we still only have that some power \(\sub^N\) of substitution fixes LI-classes, possibly with \(N \neq 1\), see Example \ref{exp:subs not LI}.

One would expect that we cannot drop the requirement that \(L\) is expansive in showing that \(\cP' \LIs \cP\) in Theorem \ref{thm:predecessor LI}. So it is natural to try to find a counter example in the simplest case of non-expansion \(L = \id\). That is, to find FLC and well-separated patterns \(\cP\) and \(\cP'\) with \(L\cP' = \cP' \LD \cP\) and \(\cP' \LI \cP\) but \emph{not} \(\cP \LI \cP'\). Specialising further, one is led to the following simply-stated question in the symbolic setting:

\begin{question}
Is there a transitive subshift \(X \subseteq A^{\Z}\), for a finite set \(A\), that contains a proper transitive subshift \(Y \subset X\) admitting a surjective factor map \(f \colon Y \to X\)?
\end{question}

We are unaware of any such examples, although consensus appears to be that they should exist. In this case, expansivity of \(L\) cannot be dropped in showing that \(\cP' \LIs \cP\).

\subsection{Minimal \(L\)-sub pattern spaces}
Although we mostly aim to address recognisability without minimality, in this section we will briefly recover some basic recurrence results in the repetitive case. Since minimal pattern spaces are hulls (see Proposition \ref{prop:repetitive <=> minimal}), we start with an individual pattern \(\cP\). We assume it is expansive \(L\)-sub, FLC and well separated, with predecessor \(\cP'\). This defines a surjective subdivision map \(S \colon L\Omega_{\cP} \to \Omega_{\cP}\) according to Theorem \ref{thm:predecessor LI}.

Even without an initial substitution rule on a prototile set, we may still define an essentially equivalent notion of `primitivity' for the individual \(L\)-sub pattern:

\begin{definition}\label{def:primitive}
We call an \(L\)-sub pattern \(\cP\) \textbf{primitive} if its \(V^0\)-patches appear relatively densely, that is, there exists some \(X \cpt E\) for which, for all \(x\), \(y \in E\), there exists some \(z \in X\) with \(\cP[x,V^0] = \cP[y+z,V^0]\).
\end{definition}

\begin{example}
The name `primitive' is motivated by FLC tile stone inflations. Such a substitution is primitive if there exists some \(k \in \N\) so that the \(k\)-fold substitution of any prototile contains at least one translated copy of each prototile \cite{BG13}. It is easy to show that tilings admitted by a primitive substitution are also primitive in the above sense. Conversely, if all \(V^0\)-patches appear relatively densely, then by taking a high enough supertile level so that all supertiles contain a translated copy of \(X\), we see that the substitution is primitive in the usual sense (after, perhaps, discarding any prototiles that do not appear in the tiling).
\end{example}

\begin{lemma}\label{lem:seeing all patches from just 0-patches}
Suppose that \(X \subset E\) meets the centres of all \(V^0\)-patches in \(\cP = \cP_0\), that is, for all \(p \in \lang_0(\cP)\), there exists some \(x \in X\) for which \(p = \cP[x,V^0]\). Then, for each \(k \in \N_0\), we have that \(L^k X\) meets the centres of all \(V^k\)-patches in \(\sub^k(\cP)\), that is, for all \(p \in \lang_k(\sub^k \cP)\) there exists some \(x \in L^k X\) so that \(p = (\sub^k \cP)[x,V^k]\).

In particular, suppose that \(\cP\) is FLC, well-separated and primitive, with \(X \cpt E\) as in Definition \ref{def:primitive}. Then there exists some \(n \in \N_0\) for which \(\cP\) is \(\rho\)-repetitive for \(\mathrm{dom}(\rho) = \{V^k \mid k \in \N_0\}\) and \(\rho(V^k) =  V^{k+n}\). In particular, \(\cP\) is repetitive.
\end{lemma}

\begin{proof}
This is essentially a corollary of Lemma \ref{lem:bigger agreement after substitution}. Indeed, take any \(p \in \lang_k(\sub^k \cP)\), say \(p = (\sub^k \cP)[y,V^k]\) for \(y \in E\). Then, by assumption, we have that \(\cP[x,V^0] = \cP[L^{-k} y,V^0]\) for some \(x \in X\). Substituting \(k\) times using this lemma, \((\sub^k \cP)[L^k x ,V^k] = (\sub^k \cP)[y,V^k]\). Since \(L^k x \in L^k X\), the first claim follows.

So now suppose that \(\cP\) is FLC and well-separated and there exists \(X \cpt E\) for which each \(X+z\) meets centres of all \(V^0\)-patches. We have that \(X \subseteq L^n B\) for some \(n \in \N_0\). By the claim just established, all translates \(L^k X + z \subseteq L^k(L^n B) + z\) contain centres of all \(V^k\)-patches in \(\sub^k(\cP)\). Moreover, \(\sub^k(\cP) \LIs \cP\) (see Proposition \ref{prop:L-sub hull properties}), so clearly the same is true in \(\cP\). Since \(L^k(L^n B) = L^{k+n}B \subseteq V^{k+n}\), the result follows.
\end{proof}

\begin{example}\label{exp:self-similar repetitive}
Suppose that \(\cP\) is primitive \(L\)-sub with \(Lx = M(\lambda x)\) an expansive similarity (an expansion \(x \mapsto \lambda x\) for \(\lambda > 1\) followed by a norm-preserving linear \(M \colon E \to E\)). Then the above shows that every \((\lambda^k+\kappa)\)-patch, for \(k \in \N_0\), has centre meeting every \((\lambda^{k+n}+\kappa)\)-ball for some fixed \(n \in \N\), since \(V^k \coloneqq (L^k(B_1))_{+\kappa} = B(\mathbf{0},\lambda^k+\kappa)\) in this case. Take any \(r \geq \kappa + 1\). Then there is some \(k \in \N\) with \(\lambda^{k-1} + \kappa \leq r < \lambda^k + \kappa\), so every \(r\)-patch has a centre in every \((\lambda^{k+n}+\kappa)\)-ball. For \(r \geq \kappa + 1\), since
\[
\frac{\lambda^{k+n}+\kappa}{r} \leq \frac{\lambda^{k+n}+\kappa}{\lambda^{k-1}+\kappa} \leq \frac{\lambda^{k+n}+\kappa}{\lambda^{k-1}} \leq \lambda^{n+1} + \kappa ,
\]
all \(r\)-patches occur with centres in all \(Cr\)-balls (\(C = \lambda^{n+1} + \kappa\)), so \(\cP\) is \emph{linearly} repetitive.
\end{example}

\begin{example}\label{exp: short repeats for self-affine}
By Lemma \ref{lem:seeing all patches from just 0-patches}, if \(\cP\) is primitive then there exists some \(\ell \in \N\) (which, by the proof, depends only on \(L\), and density of recurrence of \(V^0\)-patches) satisfying the following: if \(\cP[x,V^{k+\ell}] = \cP[x+u,V^{k+\ell}]\) for some \(u \in V^k\), then \(u \in \cK\). In the self-similar case, by Example \ref{exp:self-similar repetitive}, there exists some \(C > 0\) so that if \(\cP[x,Cr] = \cP[x+u,Cr]\) for \(\|u\| \leq r\) and \(r \geq \kappa + 1\), then \(u \in \cK\), by Lemma \ref{lem:close repeat => period}. In particular, if \(\cP\) is non-periodic then it is `linearly repulsive' \cite{BBL13} in this case, that is, there exists some \(C > 0\) so that, for sufficiently large \(r\), all return vectors to \(r\)-patches have length at least \(Cr\).
\end{example}

\subsection{Full patches near the origin and elements fixed by substitution}

For our main recognisability proof, a key step will be using that all \(U\)-patches of a pattern \(\cP \in \Omega\) can be found reasonably close to the origin, relative to the size of \(U\). This allows for an application of Lemma \ref{lem:close repeat => period}, in constructing a period of \(\cP\) from a return vector to such a patch.

\begin{proposition}\label{prop:full patch at origin}
Let \(\Omega\) be compact Hausdorff and expansive \(L\)-sub. There exists some \(N \in \N\) so that \(\sub^N\) fixes LI-classes and \(Y \cpt E\) satisfying the following. Let \(\cP \in \Omega\) be arbitrary and choose any hierarchy \((\cP_i)_{i=0}^\infty\) of \(\cP\). Then, for all \(i \in \N_0\), every \(V^0\)-patch of \(\cP\) appears in \(\cP_{Ni}\) with centre in \(Y\). That is, for all \(i \in \N_0\) and \(p \in \lang_0(\cP)\), there exists some \(y \in Y\) with \(\cP_{Ni}[y,V^0] = p\).

Thus, there exists \(\ell \in \N\) satisfying the following. For all \(\cP \in \Omega\) and \(i \in \N_0\), every \(V^i\)-patch of \(\cP\) appears with centre in \(V^{i+\ell}\) in \(\cP\). That is, for all \(p \in \lang_i(\cP)\), there exists some \(y \in V^{i+\ell}\) with \(p = \cP[y,V^i]\).
\end{proposition}

\begin{remark}
Before embarking on the proof, it may be helpful to the reader to consider the case of FLC tile substitutions. For each tile in \(\cP\), repeated substitution of it either eventually generates all tile types in \(\cP\), a so-called `generator' tile, or it does not, a `non-generator'. The origin must lie near generators arbitrarily high up the hierarchy, or else we would not see all tile types in \(\cP\) to larger and larger radii (in larger and larger non-generating supertiles) under iterated substitution. Since there are only finitely many tiles (FLC), the generating tiles generate all tiles in a bounded number \(\nu \in \N_0\) of applications of substitution, so we should be able to find all tiles in \(\cP_i\) within the (bounded) \(q\)-supertile defined by the tile of \(\cP_{i+\nu}\) over the origin, which needs to be generating; technically, to avoid issues near (or on) tile boundaries, we need to talk about (non-) generating small \(V^0\)-patches, but the argument is similar. The proof below makes this argument precise for general \(L\)-sub pattern spaces.
\end{remark}

\begin{proof}
As in the statement, take arbitrary \(\cP \in \Omega\) and choose any hierarchy \((\cP_i)_i\) for it. Let \(N \in \N\) be such that \(\sub^N\) fixes LI-classes, using Corollary \ref{cor:power fixing LI-classes}. To simplify notation for proving the first statement, we will pass to this power by letting \(\sub\) now denote \(\sub^N\), and \(L\) for \(L^N\). So we are implicitly restricting to hierarchy \((\cP_{iN})_i\) for \(\cP\), with respect to the substitution \(\sub^N\), so the result will now follow by proving the first statement with \(N=1\) in this new notation.

Recall the notation \(\lang_n(\cP)\) from Notation \ref{not:language levels}. Also recall the mapping on patches induced by substitution, defined in Lemma \ref{lem:induced substitution on patches}. This defines maps \(\sub \colon \lang_n(\cP) \to \lang_{n+1}(\cP)\) (noting that \(\sub\) now fixes LI-classes). A patch \(p \in \lang_n(\Omega)\) (thus, \(p \in A^{V^n}\)) determines its own `sub-patches': for \(m < n\) and \(p' \in \lang_m(\Omega)\), let us write \(p' \triangleleft p\) if there exists some \(x \in E\) so that \(V^m+x \subseteq V^n\) and \(p'(u) = p(u+x)\) for all \(u \in V^m\) (i.e., a translated copy of \(p'\) appears inside \(p\)). 

For each \(n \in \N_0\) and \(p \in \lang_n(\cP)\), define \(f(p) \in \N_0 \cup \{\infty\}\) to be the minimum value \(i \in \N_0\) for which, for all \(p' \in \lang_0(\cP)\), we have \(p' \triangleleft \sub^i(p)\), or \(f(p) = \infty\) if there is no such \(i \in \N_0\). We consider \(p\) to be `generating' if \(f(p) < \infty\) and otherwise \(f(p)\) is `non-generating'. By FLC, there is some \(K \cpt E\) for which, for all \(p' \in \lang_1(\cP)\), we have \(p' = \cP[x,V^1]\) for some \(x \in K\). Take \(\nu \in \N\) large enough so that \(V^\nu \supseteq K+V^0\). Then \(q = \cP_\nu[\mathbf{0},V^1]\) has \(f(q) \leq \nu < \infty\), so \(q\) is generating, since \(\sub^\nu(q) = \cP[\mathbf{0},V^{\nu+1}]\). So there is a generating patch in \(\lang_1(\cP) =  \lang_1(\cP_\nu)\). Of course, by a similar argument, we can also find generating patches in \(\lang_0(\cP)\), but we want \(q \in \lang_1(\cP)\) in the following.

Let \(K_2 \cpt E\) be such that, for all \(p \in \lang_2(\cP)\), we have \(p = \cP[x,V^2]\) for some \(x \in K_2\). Let \(\epsilon > 0\) be such that \(V^0 + B_\epsilon \subseteq V^1\) and \(V^1 + B_\epsilon \subseteq V^2\). Choose a finite set \(S \subset K_2\) that is \(\epsilon\)-dense, that is, for each \(x \in K_2\) there exists \(s \in S\) with \(\|s-x\| \leq \epsilon\). For each \(s \in S\), let \(F(s) \in \N_0 \cup \{\infty\}\) be defined by \(F(s) \coloneqq f(\cP[s,V^1])\). Let \(S_\infty \coloneqq \{s \in S \mid F(s) = \infty\}\) and \(S_{<\infty} \coloneqq \{s \in S \mid F(s) < \infty\}\). Without loss of generality, by the previous paragraph, we may suppose that \(S_{<\infty} \neq \emptyset\) (by including a centre of the patch \(q\) in \(S\)). Let \(M \coloneqq \max \{F(s) \mid s \in S_{<\infty}\} \in \N_0\). Define \(p_i \in \lang_1(\cP)\) by \(p_i \coloneqq \cP_i[\mathbf{0},V^1]\). Increasing \(M\), if necessary, we may assume that for all \(p \in \lang_0(\cP)\), we have \(p \triangleleft \cP[\mathbf{0},V^M]\). Note that, by the definition of \(M\), for any \(p \in S_{< \infty}\) and any \(p' \in \lang_0(\cP)\), we have that \(p' \triangleleft \sub^M(p)\). Indeed, if \(\sub^m(p)\) contains all patches of \(\lang_0(\cP)\) for some \(m\), then the same is true of \(\sub^M(p) = \sub^{M-m}(\sub^m(p))\) for all \(M \geq m\).

For \(n \geq M\), let \(p_n' \coloneqq \cP_n[\mathbf{0},V^2]\). We claim that, for each \(n \geq M\) and \(p \in \lang_0(\cP)\), we have \(p \triangleleft \sigma^M(p_n')\). By our definition of \(M\), this holds for \(n = M\) (even restricting \(p_M'\) to \(\cP_M[\mathbf{0},V^0] \in \lang_0(\cP)\)), since \(\sub^M(\cP_M[\mathbf{0},V^0]) = \cP[\mathbf{0},V^M]\), which we assumed contains all \(V^0\)-patches. Given \(n > M\), write \(p_n' = \cP[x,V^2]\) for \(x \in K_2\) and take \(s \in S\) with \(\|s-x\| \leq \epsilon\). Then, defining \(q \coloneqq \cP[s,V^1]\), we have \(q \triangleleft p_n'\), since \(V^1 + (s-x) \subseteq V^2\) by our choice of \(\epsilon\).

We claim that \(s \in S_{<\infty}\). Indeed, \(q\) contains the patch \(q' \coloneqq \cP[x,V^0] = \cP_n[\mathbf{0},V^0]\), since \(V^0 + (x-s) \subseteq V^1\). Thus, \(\sub^n(q)\) contains \(\sub^n(q') = \cP[\mathbf{0},V^n]\). The latter, by assumption (as \(n > M\)), contains all \(V^0\)-patches, thus so does \(\sub^n(q) = \sub^n(\cP[s,V^1])\), hence \(s \in S_{<\infty}\). Thus, by our choice of \(M\) and definition of \(S_{<\infty}\), we have that \(\sub^M(q) = \sub^M(\cP[s,V^1])\) contains all \(V^0\)-patches, hence so does the patch \(\sub^M(p_n)\) that contains it.

So we have shown that, for all \(p \in \lang_0(\cP)\) and all \(n \geq M\), we have \(p \triangleleft \sigma^M(p_n) = \sigma^M(\cP_n[\mathbf{0},V^2]) = \cP_{n-M}[\mathbf{0},V^{M+2}]\). Then the first claim follows, taking \(Y = V^{M+2}\).

The second claim now follows quickly: take \(Y \cpt E\) as in the first (and return notation to the original substitution and expansion, so that we do not necessarily have \(N = 1\)). Take \(\ell \geq N\) with \(Y \subseteq L^{\ell-N+1}(B_1)\). Let \(i \in \N_0\) and \(p \in \lang_i(\cP)\) be arbitrary, say \(p = \cP[x,V^i]\). Let \(j \in \N_0\) be such that \(i \leq Nj < i+N\). Then \(p\) is the restriction of the patch \(\cP[x,V^{Nj}] = \sub^{Nj}(q)\) to \(V^i\), where \(q = \cP_{Nj}[L^{-Nj} x, V^0]\). Since \(\cP_{Nj} \LIs \cP\), we have that \(q \in \lang_0(\cP)\), so by the defining property of \(Y\) we may in fact write \(q = \cP_{Nj}[y,V^0]\) for some \(y \in Y\). Applying \(\sub^{Nj}\), we have \(\sub^{Nj}(q) = \cP[L^{Nj} y,V^{Nj}]\) and thus, by restricting (as established above), \(p = \cP[L^{Nj} y, V^i]\). We have \(L^{Nj}y \in L^{Nj}Y \subseteq L^{i+N-1}(L^{\ell-N+1}(B_1)) = L^{i+\ell}(B_1) \subseteq V^{i+\ell}\) (see Notation \ref{not:balls}), as required.
\end{proof}

Although we will not need the following result elsewhere in this paper, it seems useful to point out that, just as for primitive stone inflation rules, some power of substitution has fixed points:

\begin{proposition}
If \(\Omega\) is FLC, well-separated and expansive \(L\)-sub, then there exists some \(\cP \in \Omega\) and \(n \in \N\) with \(\sub^n(\cP) = \cP\). Thus, \(\Omega\) contains pseudo self-affine points for some power of \(L\), that is, points \(\cP \in \Omega\) with \(L^n \cP \LD \cP\) for some \(n \in \N\).
\end{proposition}

\begin{proof}
First, using Corollary \ref{cor:power fixing LI-classes}, we assume that \(\sub\) fixes LI-classes of \(\Omega\) (otherwise, we begin by passing to a power that does). Take any \(\cQ \in \Omega\) and, using FLC, let \(N \in \N\) be such that all \(V^1\)-patches occur with centre in \(B_N\), that is, for all \(x \in E\) there exists some \(y \in V^1\) so that \(\cQ[x,V^1] = \cQ[y,V^1]\).

For each \(n \in \N\), pick some \(\cQ_n \in \sub^{-n}(\cQ)\). Take \(x_n \in B_N\) with \(\cQ[x_n,V^1] = \cQ_n[\mathbf{0},V^1]\). Given any \(y \in E\) small enough that \(V^0 + y \subseteq V^1\) (for the second equality below), by Lemma \ref{lem:pattern properties} we have
\begin{align*}
(\cQ_n-y)[\mathbf{0},V^0] & = \cQ_n[y,V^0] = \cQ[x_n+y,V^0] = (\cQ-(x_n+y))[\mathbf{0},V^0] \\
                          & = (\sub^n(\cQ_n-L^{-n}(x_n+y)))[\mathbf{0},V^0] .
\end{align*}
For sufficiently large \(n\), we may take \(y\) sufficiently so small (as required for above) and also satisfying \(y = L^{-n}(x_n+y)\). Indeed, equivalently, \(y = (L^n - \id)^{-1}x_n\) and, since \(L\) is expansive, \(\|(L^n - \id)^{-1}\| \to 0\) as \(n \to \infty\), whilst \(\|x_n\| \leq N\) remains bounded. So we have shown that there exists \(n \in \N\) and some \(\cQ' = \cQ_n-y \in \Omega\) satisfying \(\cQ'[\mathbf{0},V^0] = (\sub^n \cQ')[\mathbf{0},V^0]\).

Define \(\cP_i \coloneqq \sub^{ni}(\cQ')\). Thus, \(\cP_0[\mathbf{0},V^0] = \cP_1[\mathbf{0},V^0]\). Applying \(\sub^n\), from Lemma \ref{lem:bigger agreement after substitution}, we have \(\cP_1[\mathbf{0},V^1] = \cP_2[\mathbf{0},V^1]\) (in fact, even \(\cP_1[\mathbf{0},V^n] = \cP_2[\mathbf{0},V^n]\)). Iterating, we have \(\cP_i[\mathbf{0},V^i] = \cP_{i+1}[\mathbf{0},V^i]\). Thus, \((\cP_i)_{i \in \N}\) is a Cauchy sequence, so \(\cP_i \to \cP\) as \(i \to \infty\), by compactness (or completeness) of \(\Omega\). By continuity of \(\sub^n\) and \(\Omega\) being Hausdorff, it follows that
\[
\sub^n(\cP) = \sub^n \lim \cP_i = \lim \sub^n \cP_i = \lim \cP_{i+1} = \cP ,
\]
so \(\sub^n(\cP) = \cP\), as required. Hence, \(L^n \cP \LD \cP\), so \(\Omega\) contains pseudo self-affine points.
\end{proof}

\section{Recognisability}
\label{sec:recognisability}

Having built up the necessary framework, we now state and prove our main recognisability results. In our geometric setting, a useful notion of recognisability in the presence of periodic and non-periodic points is the following:

\begin{definition}
Let \(\Omega\) be an \(L\)-sub pattern space, with associated substitution \(\sub \colon \Omega \to \Omega\). We say that \(\sub\) satisfies \textbf{unique composition modulo translation}, or \UC, for short, if for all \(\cU\), \(\cV \in \Omega\) with \(\sub(\cU) = \sub(\cV)\) we have \(\cU = \cV + x\) for some \(x \in E\).
\end{definition}

Thus, \UC\ is simply the condition that all pre-images under substitution are translation equivalent. Our main theorem is that it always holds under reasonable circumstances:

\begin{theorem}\label{thm:recognisability}
For a compact Hausdorff expansive \(L\)-sub pattern space \(\Omega\), the substitution map \(\sub \colon \Omega \to \Omega\) satisfies \UC.
\end{theorem}

We will soon see how this relates to the degree of injectivity of substitution. Such a result is already known in the case of primitive stone inflation tilings \cite{Sol98}. For non-primitive admissible stone inflations, it is also known that substitution is injective when all elements of the substitution tiling space are non-periodic and that non-periodic elements have unique pre-images when there are also periodic elements, provided a `non-periodic border condition' holds \cite{CS11}. Such results of course also apply when an MLD equivalence to such a case can be established. The above result extends these by already applying to general types of patterns (including non-return discrete ones, allowing for non-relatively dense point sets), with a slightly broader notion of being substitutional (even for tilings) and needing no further conditions such as the non-periodic border forcing condition. Moreover, as we will see now, it directly leads to a formula for the size of all fibres of substitution, not just over the non-periodic elements.

Throughout the remainder of this section, we assume that \(\Omega\) satisfies the hypotheses of Theorem \ref{thm:recognisability}. Before proving it, we deduce some simple consequences of \UC.

\begin{proposition}\label{prop:number of pre-images}
For all \(\cP \in \Omega\) and all \(\cP' \in \sub^{-1}(\cP)\) we have
\[
\# \sub^{-1}(\cP) = [\cK_{\cP} : L \cK_{\cP'}] < \infty .
\]
In particular, if \(\cP\) is non-periodic, it has a unique pre-image under substitution.
\end{proposition}

\begin{proof}
We have \(\sub^{-1}(\cP) = \{\cP' + x \mid x \in L^{-1} \cK_{\cP} \}\). Indeed, by \UC, \(\cQ \in \sub^{-1}(\cP)\) if and only if \(\cQ = \cP' + x\) for some \(x \in E\). And, for such a translate, \(\cP' + x \in \sub^{-1}(\cP)\) if and only if \(\cP = \sub(\cP'+x) = \sub(\cP')+Lx = \cP+Lx\), that is, \(Lx \in \cK_{\cP}\), as required. Now, given two such translates, we have \(\cP'+x = \cP' + y\) if and only if \(\cP' + (y-x) = \cP'\), that is, \(y-x \in \cK_{\cP'}\). Thus, we may identify \(\sub^{-1}(\cP)\) with the quotient group \(L^{-1}\cK_{\cP} / \cK_{\cP'} \cong \cK_{\cP} / L \cK_{\cP'}\). This is finite, by Corollary \ref{lem:groups of periods FI}. If \(\cP\) is non-periodic, we have that \(\cK_{\cP}\) is the trivial group, so this index is equal to \(1\).
\end{proof}

For the following, recall that \(\cK_{\cP}\) is discretely non-periodic (Definition \ref{def:discretely periodic}) if \(\cK_{\cP}\) has trivial discrete component or, equivalently, \(\cK_{\cP}\) is connected.

\begin{corollary}\label{cor:uniqueness of pre-images}
If \(\cP \in \Omega\) is discretely non-periodic then \(\# \sub^{-1}(\cP) = 1\). Conversely, there exists some \(N \in \N\) so that \(\sub^N\) fixes LI-classes, for which \(\# \sub^{-N}(\cP) = [\cK_{\cP} : L^N \cK_{\cP}]\) and thus, if \(\# \sub^{-N}(\cP) = 1\), then \(\cP\) is discrete non-periodic.
\end{corollary}

\begin{proof}
If \(\cP\) is discretely non-periodic, \(\cK_{\cP}\) is connected and thus some real subspace of \(E\). Given \(\cP' \in \sub^{-1}(\cP)\), the connected component of \(L\cK_{\cP'}\) at the origin must be all of \(V\), or else \([\cK_{\cP} : L \cK_{\cP'}] = \infty\), contradicting Proposition \ref{prop:number of pre-images}. But then \(\cK_{\cP} = L\cK_{\cP'}\) and this index is \(1\), so that \(\sub^{-1}(\cP) = \{\cP'\}\) by the same result.

Conversely, by Corollary \ref{cor:power fixing LI-classes}, there exists the stated power \(N \in \N\) of substitution fixing LI-classes. By applying Proposition \ref{prop:number of pre-images} with substitution \(\sub^N\), and any \(\cP \in \Omega\), \(\cP' \in \sub^{-N}(\cP)\), we obtain
\[
\# \sub^{-N}(\cP) = [\cK_{\cP} : L^N \cK_{\cP'}] = [\cK_{\cP} : L^N \cK_{\cP}] ,
\]
where the second equality follows from \(\cP \LIs \cP'\) and thus \(\cK_{\cP} = \cK_{\cP'}\) (Lemma \ref{lem:periods under LI and LD}). By Lemma \ref{lem:index of closed subgroups}, this index is equal to \(1\) if and only if \(\cK_{\cP}\) has trivial discrete component, that is, \(\cP\) is discretely non-periodic, as required.
\end{proof}

In the above result, recall that being discretely periodic is equivalent to being periodic in the return discrete cases (e.g., from tilings or Delone sets). So, in this case, the above says that non-periodic points have unique pre-images, whilst periodic points have multiple pre-images at least under some power of the substitution. This power cannot be dropped in general, even in the case of hulls of tilings, see Example \ref{exp:subs not LI}. We summarise this with the following recognisability dichotomy:

\begin{corollary}\label{cor:injectivity dichotomy}
Substitution \(\sub \colon \Omega \to \Omega\) is injective if and only if each \(\cP \in \Omega\) is discretely non-periodic. In the return discrete case, this is equivalent to each \(\cP \in \Omega\) being non-periodic. When \(\sub \colon \Omega \to \Omega\) is injective, the inverse \(S^{-1} \colon \Omega \LDmap L \Omega\) of subdivision is an LD map.
\end{corollary}

\begin{proof}
By Corollary \ref{cor:uniqueness of pre-images}, if each \(\cP\) is discretely non-periodic then each \(\# \sub^{-1}(\cP) = 1\), that is, \(\sub\) is injective. Conversely, if \(\sub\) is injective then so is \(\sub^N\), so each \(\cP\) is discretely non-periodic by the same result. In the return discrete case any group of periods \(\cK_{\cP}\) of a pattern \(\cP \in \Omega\) is discrete (Proposition \ref{prop:return discrete from discrete periods}), so discrete non-periodicity is equivalent to non-periodicity, as required. Finally, suppose that \(\sub \coloneqq S \circ L\) is injective, equivalently \(S \colon L\Omega \LDmap \Omega\) is injective (and it is also surjective, by definition of an \(L\)-sub pattern space). Then \(S^{-1}\) is also an LD map, by Proposition \ref{prop:inverse of LD is LD}.
\end{proof}

In the case of a minimal hull \(\Omega = \Omega_{\cP}\) of an FLC tiling, all elements of \(\Omega\) are locally isomorphic (Proposition \ref{prop:repetitive <=> minimal}), and thus all have the same groups of periods, so that \(\Omega\) consists of only non-periodic elements if and only if \(\cP\) is non-periodic. Thus, the well-known unique composition result of Solomyak \cite{Sol07} in the case of an FLC and primitive stone inflation follows from the above. The following result gives a more general result for an individual return discrete and substitutional pattern, that need not be repetitive:

\begin{corollary}\label{cor:unique decomposition for individual patterns}
Let \(\cP\) be an FLC and return discrete pattern in the Euclidean space \(E\) and \(L \colon E \to E\) a linear expansion. Suppose that \(\cP\) is \(L\)-sub, that is, there is a pattern \(\cP'\) for which \(\cP' \LI \cP\) and \(L\cP' \LD \cP\). Then \(\cP' \LIs \cP\). We have that \(\cP \LD L\cP'\), equivalently \(\sub \colon \Omega_{\cP} \to \Omega_{\cP}\) is injective, if and only if \(\cP\) is aperiodic, that is, \(\Omega_{\cP}\) contains no periodic elements.
\end{corollary}

\begin{proof}
Since \(\cP\) is return discrete, it is also well-separated (Lemma \ref{lem:return discrete => well-separated}). Then \(\Omega = \Omega_{\cP}\) is \(L\)-sub, with substitution map satisfying \(\sub(\cP') = \cP\), by Theorem \ref{thm:predecessor LI}. The same result gives \(\cP' \LIs \cP\). We also have \(\Omega = \Omega_{\cP}\) is compact Hausdorff, by Proposition \ref{prop:well-sep <=> HD}, so \(\sub\) is injective if and only if \(\cP\) is aperiodic by Corollary \ref{cor:injectivity dichotomy}.
\end{proof}

\subsection{Proof of Theorem \ref{thm:recognisability}}
In this section we prove our main recognisability result. As already stated, we assume throughout that \(\Omega\) is compact Hausdorff and expansive \(L\)-sub. We will in fact prove instead the following `local' version of \UC:

\begin{lemma}\label{lem:local recognisability}
Given \(\cP \in \Omega\), \(\cP' \in \sub^{-1}(\cP)\) and \(f \colon \R_{\geq 0} \to \R_{\geq 0}\), define the following property:
\begin{itemize}
	\item[(\(\mathrm{P}_f\))] for all \(x \in E\) and \(R \geq 0\),
	\[
	\text{if} \ \cP[\mathbf{0},f(R)] = \cP[x,f(R)] \ \text{then there exists } g \in \cK_{\cP} \ \text{with} \ \cP'[\mathbf{0},R] = \cP'[L^{-1}(x+g),R] . 
	\]
\end{itemize}
If, for all \(\cP \in \Omega\) and \(\cP' \in \sub^{-1}(\cP)\), there exists some \(f\) for which \(\mathrm{P}_f\) holds for this pair, then \(\sub\) satisfies \UC.
\end{lemma}

\begin{proof}
Let \(\cP\), \(\cP' \in \sub^{-1}(\Omega)\) be arbitrary and suppose \(\mathrm{P}_f\) holds. Take any \(\cU \in \Omega\) for which we also have \(\sub(\cU) = \cP\). Set an arbitrary sequence \(R_n \to \infty\) (e.g., \(R_n = n\)). By Corollary \ref{cor:power fixing LI-classes}, \(\sub\) acts bijectively on LI-classes, so we must have that \(\cU \LIs \cP'\). In particular,
\begin{equation} \label{eq:lem-UC}
\cU[\mathbf{0},R_n'] = \cP'[L^{-1} x_n,R_n']
\end{equation}
for some \(x_n \in E\), where \(R_n' \coloneqq \max\{R_n,f(R_n)\}\). Substituting and using \(\cP = \sub(\cP') = \sub(\cU)\), we have
\[
\cP[\mathbf{0},R_n'] = \sub(\cU)[\mathbf{0},R_n'] = \sub(\cP')[x_n,R_n'] = \cP[x_n,R_n']
\]
for sufficiently large \(R_n' \geq R_n \to \infty\) since, by Lemma \ref{lem:bigger agreement after substitution}, substitution only increases size of agreement of patches. By \(\mathrm{P}_f\), we have \(\cP'[\mathbf{0},R_n] = \cP'[L^{-1}(x_n+g_n),R_n] = p\) for some \(g_n \in \cK = \cK_{\cP}\). Let \(\Gamma\) be a set of coset representatives of \(L\cK_{\cP'}\) in \(\cK\). Since \(p = \cP'[L^{-1}(x_n+(g_n+Lh)),R_n]\) for any period \(h \in \cK_{\cP'}\), we may choose each \(g_n \in \Gamma\). Moreover, \(\Gamma\) is finite by Lemma \ref{lem:groups of periods FI}. Thus, by restricting to a subsequence, we may assume that each \(g_n = g\) for some \(g \in \Gamma \subseteq \cK\). Hence, shifting this equality of patches by \(-L^{-1} g\) and using Equation \ref{eq:lem-UC}, we have
\[
\cU[\mathbf{0},R_n - \delta] = \cP'[L^{-1}x_n,R_n - \delta] = \cP'[-L^{-1}g,R_n - \delta] = (\cP'+L^{-1}g)[\mathbf{0},R_n-\delta],
\]
where \(\delta\) is chosen with \(L^{-1}\Gamma \subseteq B_\delta\), which we subtract from the radius \(R_n\) of patch agreement when shifting by \(-L^{-1}g\), which applies once \(R_n \geq \delta\), using Lemma \ref{lem:pattern properties} (3). Since \(\cU[\mathbf{0},R_n - \delta] = (\cP + L^{-1}g)[\mathbf{0},R_n - \delta]\) for \(R_n - \delta \to \infty\), we have that \(\cU = \cP' + L^{-1}g\). So the arbitrary pre-images \(\cP'\), \(\cU \in \sub^{-1}(\cP)\) are related by a translation and hence \UC\ holds, as required.
\end{proof}

The above result is an analogue of \cite[Lemma 2.5]{Sol98}, although modified to fix consideration at the origin. As well as simplifying matters (by removing one variable), this is crucial: the two-pointed version (where we replace \(\mathbf{0}\) with an arbitrary \(y \in E\) and \(L^{-1} y \in E\) in the left- and right-hand equations of Property \(\mathrm{P}_f\)) fails in the non-repetitive case. Moreover, its statement is natural in light of the origin-centric Proposition \ref{prop:full patch at origin}, needed later. Note that, in the aperiodic case where each \(\cK_{\cP}\) is trivial, condition \(\mathrm{P}_f\) simply says the patches at \(\mathbf{0}\) and \(x\) can only agree to a large radius in \(\cP\) if they also agree in \(L\cP'\) to a relatively large radius.

One may (with some ease, at least in the return discrete case) prove a converse: if \UC\ is satisfied, \(\mathrm{P}_f\) holds for some \(f\), demonstrating necessity as well as sufficiency of some \(\mathrm{P}_f\) for \UC. We do not provide this proof, since we can instead show it simply always holds:

\begin{proposition}\label{prop:L-sub=>Pf}
There exists some \(f \colon \R_{\geq 0} \to \R_{\geq 0}\) for which Property \(\mathrm{P}_f\) of Lemma \ref{lem:local recognisability} holds for all \(\cP \in \Omega\) and \(\cP' \in \sub^{-1}(\cP)\).
\end{proposition}

\begin{proof}
Let \(\cP \in \Omega\) and \(\cP' \in \sub^{-1}(\cP)\) be arbitrary and define \(\cK \coloneqq \cK_{\cP}\). Choose a corresponding hierarchy \((\cP_i)_{i=0}^\infty\), that is, a sequence of patterns \(\cP_i \in \Omega\) with \(\sub(\cP_{i+1}) = \cP_i\), \(\cP_0 = \cP\) and \(\cP_1 = \cP'\). Throughout this proof, to simplify notation, for \(j \in \N_0\) we denote \(\cP_i[x,j] \coloneqq \cP_i[x,V^j]\) (instead of our usual notation \(\cP_i[x,j] = \cP_i[x,B_j]\)). We claim that, for sufficiently large \(k \in \N_0\), if \(\cP[\mathbf{0},k+2] = \cP[x,k+2]\) for some \(x \in E\) then \(\cP'[\mathbf{0},k] = \cP'[x^{-1} + L^{-1} h,k]\) for some \(h \in \cK\). In this case, clearly \(\mathrm{P}_f\) holds (and even independently of \(\cP\) and \(\cP'\)) for a certain function \(f\) (whose explicit description is not important and is neater in terms of \(V^k\)-patches rather than \(r\)-patches).

Let \(\ell \in \N\) be as given from Proposition \ref{prop:full patch at origin}. Let \(\epsilon > 0\) with \(\epsilon \leq (2(\overline{\lambda})^\ell)^{-1}\), where we take \(\overline{\lambda} > 0\) such that \(LB \subseteq \overline{\lambda} B\) for \(B \coloneqq B_1\). We may also take \(\epsilon > 0\) sufficiently small that \(V^k + L^k(\epsilon B) \subseteq V^{k+1}\) for all \(k \in \N_0\), since
\begin{align*}
V^{k+1} &= L^{k+1}(B) + \kappa B \supseteq \lambda L^k(B) + \kappa B = (L^k(B) + (\lambda - 1)L^k(B)) + \kappa B\\
        &= V^k + L^k((\lambda - 1)B) \supseteq V^k + L^k(\epsilon B) ,
\end{align*}
provided \(0 < \epsilon \leq \lambda - 1\). In the second equality, we use that for all \(\alpha\), \(\beta > 0\),
\[
\alpha L^k(B) + \beta L^k(B) = L^k(\alpha B + \beta B) = L^k((\alpha + \beta)B) = (\alpha + \beta)L^k(B) .
\]
By expansivity, we also have \(V^k + L^m(\epsilon B) \subseteq V^{k+1}\) for other \(m \leq k\).

From compactness of \(\Omega\), we may apply FLC (Definition \ref{def:FLC pattern space}) to patches with shape \(U = V^3\), yielding patterns \(\cQ_1\), \ldots, \(\cQ_n \in \Omega\) and \(K \cpt E\) for which, for all \(\cQ \in \Omega\) and \(x \in E\), there exists some \(i \in \{1,\ldots,n\}\) and \(x \in K\) with \(\cQ[x,3] = \cQ_i[y,3]\). Let \(K' = K \times \{1,\ldots,n\}\) denote the disjoint union of \(n\) copies of \(K\). By compactness of \(K' \times K'\), there is some \(M \in \N\) for which, for any \(M\) pairs \(((a_i,m_i),(b_i,n_i)) \in K' \times K'\), we may find \(i \neq j\) with \(m_i = m_j\), \(n_i = n_j\), \(\|a_i - a_j\| \leq \epsilon\) and \(\|b_i - b_j\| \leq \epsilon\).

For a contradiction, suppose that there are infinitely many \(k \in \N_0\) and \(x_k \in E\) with \(\cP[\mathbf{0},k+2] = \cP[x_k,k+2]\) yet \(\cP'[\mathbf{0},k] \neq \cP'[L^{-1}(x_k + g),k]\) for all \(g \in \cK\). For each such \(k\), consider the patches \(\cP_k[\mathbf{0},3]\) and \(\cP_k[L^{-k}x_k,3]\). From FLC, we may write \(\cP_k[\mathbf{0},3] = \cQ_{m_k}[a_k,3]\) and \(\cP_k[L^{-k} x_k,3] = \cQ_{n_k}[b_k,3]\) for \(a_k\), \(b_k \in K\) and \(m_k\), \(n_k \in \{1,\ldots,n\}\). By considering \(M\) different values of \(k \geq \ell\), we obtain corresponding pairs \((a_k,m_k) \in K'\) and \((b_k,n_k) \in K'\). By our choice of \(M\), there exist indices \(j > i \geq \ell\) for which \(m_i = m_j = m\), \(n_i = n_j = n\), \(\|a_i - a_j\| \leq \epsilon\) and \(\|b_i - b_j\| \leq \epsilon\). In summary,
\[
\cP_i[\mathbf{0},3] = \cQ_m[a_i,3]\ , \ \cP_i[L^{-i} x_i,3] = \cQ_n[b_i,3]\ , \ \cP_j[\mathbf{0},3] = \cQ_m[a_j,3]\ , \ \cP_j[L^{-j} x_j,3] = \cQ_n[b_j,3].
\]
By shifting the third equation above by \(u \coloneqq a_i - a_j \in \epsilon B\) and the fourth by \(v \coloneqq b_i - b_j \in \epsilon B\) (using Lemma \ref{lem:pattern properties} (3), with \(V^2 + \epsilon B \subseteq V^3\)), and combining with the first and second equations, respectively (each restricted to \(V^2\)), we obtain
\begin{equation}\label{eq:UC1}
\cP_i[\mathbf{0},2] = \cP_j[u,2]\ , \ \cP_i[L^{-i}x,2] = \cP_j[L^{-j}y + v,2] ,
\end{equation}
where we denote \(x \coloneqq x_i\) and \(y \coloneqq x_j\). We recall that, by definition of \(x_i\) and \(x_j\),
\begin{equation}\label{eq:UC2}
\cP[\mathbf{0},i+2] = \cP[x,i+2]\ , \ \cP[\mathbf{0},j+2] = \cP[y,j+2]
\end{equation}
but, for all \(g \in \cK\),
\begin{equation}\label{eq:UC3}
\cP'[\mathbf{0},i] \neq \cP'[L^{-1}(x+g),i]\ , \ \cP'[\mathbf{0},j] \neq \cP'[L^{-1}(y+g),j] .
\end{equation}
We will construct some \(h \in \cK\) contradicting the right-hand non-equality for \(g=-h \in \cK\).

Begin by substituting Equation \ref{eq:UC1} \(i\) times, by Lemma \ref{lem:bigger agreement after substitution}, and denoting \(\eta \coloneqq j-i \in \N\), to obtain
\begin{equation}\label{eq:UC4}
\cP[\mathbf{0},i+2] = \cP_\eta[L^i u,i+2]\ , \ \cP[x,i+2] = \cP_\eta[L^{-\eta} y + L^i v,i+2] .
\end{equation}
Combining this with Equation \ref{eq:UC2} gives
\begin{equation}\label{eq:UC5}
\cP_\eta[L^i u,i+2] = \cP[\mathbf{0},i+2] = \cP[x,i+2] = \cP_\eta[L^{-\eta} y + L^i v,i+2] .
\end{equation}
Substituting this a further \(\eta\) times gives
\[
\cP[L^j u,j+2] = \cP[y + L^j v,j+2] .
\]
Recall that \(u \in \epsilon B\) and that, for all \(j \in \N_0\), we have \(V^{j+2} \supseteq V^{j+1} + L^j(\epsilon B)\). Thus, we may shift the above equation by \(L^j v \in L^j(\epsilon B)\) to obtain
\begin{equation}
\cP[L^j(u-v),j+1] = \cP[y,j+1] = \cP[\mathbf{0},j+1] ,
\end{equation}
where the latter is given by restricting the right-hand side of Equation \ref{eq:UC2} to \(V^{j+1}\). The above shows that \(h \coloneqq L^j(u-v)\) is a `short' return vector to the relatively large patch of shape \(V^j\) (in fact, even \(V^{j+1}\)). More precisely, we have \(h \in V^{j-\ell}\) since, by assumption, \(h = L^j(v - u) \in L^j(2\epsilon B)\), so
\[
h \in L^{j-\ell}L^\ell(2\epsilon B) \subseteq L^{j-\ell}((\overline{\lambda})^\ell 2\epsilon B) \subseteq L^{j-\ell} B \subseteq L^{j-\ell} B + \kappa B = V^{j-\ell} ,
\]
where we recall, for the second inclusion, that we defined \(\epsilon \leq (2(\overline{\lambda})^\ell)^{-1}\). In the first, we use that \(L(B) \subseteq \overline{\lambda}B\), hence \(L(\mu B) \subseteq \overline{\lambda}(\mu B)\) for any \(\mu \geq 0\) (here, \(\mu = 2\epsilon\)) and thus, by induction, \(L^\ell(\mu B) \subseteq (\overline{\lambda})^\ell \mu B\).

Summarising, we have \(h = L^j(u-v) \in V^{j-\ell}\) for which \(\cP[h,j+1] = \cP[\mathbf{0},j+1]\). Recall that \(\ell\) was chosen from Proposition \ref{prop:full patch at origin} so that a centre of every \(V^n\)-patch occurs in \(V^{n+\ell}\), that is, given \(p \in \lang_n(\cP)\), we have that \(p = \cP[x,V^n]\) for some \(x \in V^{n+\ell}\). Taking \(n = j-\ell \in \N_0\) here, since \(h \in V^{j-\ell}\), it follows from Lemma \ref{lem:close repeat => period} (in short: a \(U\)-short return vectors of a patch containing centres of all \(U\)-patches must be a period) that \(h \in \cK\).

Thus, it remains to show that \(\cP'[\mathbf{0},j] = \cP'[L^{-1}(y-h),j]\), contradicting the assumed non-equality of Equation \ref{eq:UC3}. But this may be done similarly to earlier, just shifting down \(j-1\) rather than \(j\) levels. Explicitly, from Equation \ref{eq:UC5}, substitute a further \(\eta-1\) times to obtain \(\cP_1[L^{j-1} u,j+1] = \cP_1[L^{-1} y + L^{j-1} v,j+1]\). Just as before, due to our choice of small \(\epsilon\), we have \(V^{j+1} \supseteq V^j + L^{j-1}(\epsilon B)\), so we may shift this equality of patches by \(-L^{j-1} u \in L^{j-1}(\epsilon B)\) to obtain \(\cP_1[\mathbf{0},j] = \cP_1[L^{-1} y + L^{j-1}(v-u),j]\). Since \(L^{j-1}(v-u) = L^{-1}(L^j(v-u)) = L^{-1}(-h)\), the result follows.
\end{proof}

\begin{remark}
Notice that, if \(\cP\) is non-periodic, then \(h=\mathbf{0}\) and the last few steps are simpler; in particular, it would have been sufficient to start with agreement of just \(V^{k+1}\)-patches in \(\cP\) at the start (rather than \(V^{k+2}\)). Also note that there there is considerable room to work with, in the sense that property \(\mathrm{P}_f\) was shown to always hold and with a relatively slow-growing function \(f\), independently of \(\cP\) and \(\cP'\). Through such considerations, it is possible to replace linear expansivity of \(L\) with a weaker expansivity condition, something we will return to in future work.
\end{remark}

The main Theorem \ref{thm:recognisability} is now a corollary of the above two results: for all \(\cP \in \Omega\) and \(\cP' \in \sub^{-1}(\cP)\), we have that \(\mathrm{P}_f\) is satisfied for some \(f\) by Proposition \ref{prop:L-sub=>Pf} and hence \UC\ holds by Lemma \ref{lem:local recognisability}.

\section{Examples}\label{sec:examples}

In this section we give examples demonstrating the range of patterns that our main recognisability results may be applied to.

\begin{example}[Half and half patterns]
Here we demonstrate our recognisability results on two simple non-minimal patterns, the first return discrete and the second not. Take the tiling of \(E = \R^1\) of interval \([n,n+1]\) tiles, for \(n \in \Z\), coloured white for \(n < 0\) and black for \(n \geq 0\) (see \cite[Figure 1.6]{Sad08}), with associated pattern \(\cT\) (see Section \ref{sec:patterns from tilings}). Then \(\cT\) is \(L\)-sub with respect to any inflation \(Lx = kx\), for \(k \in \N\), with itself as predecessor. Subdivision replaces an inflated tile \(Lx\) with \(k\) unit-length tiles of the same colour. This extends over the \(L\)-sub hull \(\Omega = \Omega_{\cT}\), defining the LD map \(S \colon L\Omega \LDmap \Omega\). Take \(k \geq 2\), so that \(L\) is expansive. The induced substitution map \(\sub = S \circ L\) preserves LI-classes so, by Proposition \ref{prop:number of pre-images} (and Lemma \ref{lem:periods under LI and LD}), each \(\cQ \in \Omega\) has exactly \([\cK_{\cQ},L\cK_{\cQ}]\) pre-images under \(\sub\). Explicitly, the hull \(\Omega\) consists of the orbit \(\cT+E\), which we may identify with \(E\) (where substitution acts by \(\sub(\cT+x) = \cT+Lx\)), along with two other orbits, of the periodic tiling \(\cB\) of all black tiles and \(\cW\) of all white tiles. The tiling space is `slinky-like', with two embedded circles (the orbits of \(\cB\) and \(\cW\)), with the other orbit \(\cP+x\) spiralling towards \(\cW+E\) as \(x \to \infty\), and towards \(\cB+E\) as \(x \to -\infty\). We can also directly inspect that \(\sub^{-1}(\cW)\) (and similarly for \(\cB\), and translates) contains \(k = [\cK_{\cW}:L\cK_{\cW}]\) points, whereas \(\sub \colon \cT+E \to \cT+E\) is injective. In particular, \(\sub\) is not injective since \(\Omega\) contains periodic points (and we have return discreteness), as stated by the general dichotomy result Corollary \ref{cor:injectivity dichotomy}.

We may consider a non-discrete analogue of the above, the pattern \(\cP \colon \R^1 \to \{B,W\}\), where \(\cP[x] = W\) if \(x \leq 0\) and \(\cP[x] = B\) if \(x > 0\). In this case, \(\Omega = \Omega_{\cP}\) consists, again, of the orbit \(\cP+E\) but now just \emph{two} other patterns (rather than two and their infinitely many translates), the pattern \(\cB'\) whose points are all labelled \(B\), and \(\cW'\) whose points are all labelled \(W\), each of which has periods \(\cK = E\). Again, substitution preserves LI-classes and, by Proposition \ref{prop:number of pre-images}, \(\sub\) is injective since each \([\cK:L\cK] = 1\) in this case (or, by Corollary \ref{cor:injectivity dichotomy}, \(\sub\) is injective since no \(\cK_{\cQ}\) has non-trivial discrete component). Note that \(\Omega \cong [0,1]\), with interior identified with \(\cP+E\), and end points \(\cB'\) and \(\cW'\). This is unusual in the context of Aperiodic Order, where the hull is usually a torus bundle with discrete fibres, see \cite{SW03}. However, this does not apply here, as we do not have return discreteness (and so no MLD relation to tilings). Nonetheless, it is sufficiently well behaved so as to not be excluded from our main recognisability results.
\end{example}

\begin{example}[\(L\)-sub patterns with unusual hulls]
\label{exp:unusual hulls}
As we saw above, dropping return discreteness allows for less standard hulls. Here we define some other, similarly unusual \(L\)-sub hulls of patterns of \(E = \R^2\) labelled in \(S = \{0,1,2\}\):
\begin{enumerate}
	\item \(\cP_1[\mathbf{0}] = 1\) and \(\cP_1[v] = 0\) otherwise;
	\item \(\cP_2[(x,y)] = 1\) if \(x \in \Z\) and \(y \geq 0\), with \(\cP_2[v] = 0\) otherwise;
	\item \(\cP_3[(x,y)] = 1\) if \(x \in \Z\) and \(y = 0\), with \(\cP_3[v] = 0\) otherwise;
	\item \(\cP_4[(x,y)] = 1\) if \(x \in \Z\) and \(y \geq 0\), \(\cP_4[(x,y)] = 2\) if \(x \in \Z\) and \(y < 0\), and \(\cP_4[v] = 0\) otherwise.
\end{enumerate}
Then \(\cP_1\) decorates \(\R^2\) with a single mark at the origin, \(\cP_2\) marks the upper-half plane periodically with half-vertical lines, \(\cP_3\) marks the \(x\)-axis periodically with points and \(\cP_4\) defines the same markings as \(\cP_2\), as well as similar lines in the lower half-plane using a different colour.

We leave it to the reader to identify that \(\Omega_{\cP_1} \cong S^2\) (the \(2\)-dimensional sphere), \(\Omega_{\cP_2} \cong D^2\) (the closed \(2\)-dimensional disc), \(\Omega_{\cP_3}\) is homeomorphic to the quotient of \(S^2\) with poles pinched to a single point and \(\Omega_{\cP_4} \cong S^1 \times [0,1]\), the cylinder. In each case, the patterns are \(L\)-sub for \(L(x) = 2x\) with themselves as predecessor. The associated substitution map is injective for \(\cP_1\) and not for the others. Indeed, the only periodic element of \(\Omega_{\cP_1}\) is the void pattern (see Example \ref{exp:trivial patterns}), which is not \emph{discretely} periodic, whilst \(\cP_2\), \(\cP_3\) and \(\cP_4\) have periods \(\cK = \Z \times \{0\}\), each with \(2 = [\cK : L\cK]\) pre-images in their hulls.
\end{example}

\begin{example}[Trivial patterns]\label{exp:trivial patterns}
The `void' pattern \(\cP \colon E \to \{\emptyset\}\) defined by \(\cP[x] \coloneqq \emptyset\) can be thought of as a `point set' that is devoid of any points; it occurs in any hull of a non-relatively dense point set. Its hull is \(\Omega_{\cP} = \{\cP\}\), a single-point space. We have that \(\cP\) is well-separated, FLC and \(L\)-sub (for any linear automorphism \(L\)), and recognisability still applies: substitution \(\sub \colon \Omega_{\cP} \to \Omega_{\cP}\) is (obviously) injective, and indeed \([\cK : L\cK] = [E : LE] = [E : E] = 1\).

In the other extreme, consider a pattern with \(\cP[x] = \cP[y]\) if and only if \(x = y\). Then \(\cP\) is well-separated, so its hull \(\Omega\) is Hausdorff, but it is not FLC so \(\Omega\) is not compact. Indeed, \(\Omega = \cP + E  \cong E\). Up to MLD, we can take \(\cP \in E^E\) with \(\cP = \id\) i.e., each \(x \in E\) is labelled by itself. Alternatively, up to MLD, it may be given as a periodic tessellation of unit (hyper)cubes but where each tile is assigned a unique label. Clearly, \(\cP\) is \(L\)-sub for any linear automorphism \(L\), all elements in \(\Omega\) are non-periodic and substitution \(\sub \colon \Omega \to \Omega\) may be identified with \(L \colon E \to E\). It is not FLC though, so is not covered by our recognisability results.
\end{example}

\begin{example}[Fibonacci bar codes]\label{exp:Fibonacci bar code}
Let \(\Omega_\mathrm{fib}\) be the hull of Fibonacci tilings, whose tiles, \(a\) and \(b\), have lengths \(\varphi\) (the golden mean) and \(1\), respectively, generated by the substitution \(a \mapsto ab\), \(b \mapsto a\). For each tiling of \(\Omega_\mathrm{fib}\), define the 2d pattern \(\cP \colon \R^2 \to \{B,W\}\) by \(\cP[(x,y)] = B\) if \(x\) belongs to an \(a\) tile and \(\cP[(x,y)] = W\) otherwise. One may think of this as colouring the plane with a `vertical bar code', which is aperiodic (following a Fibonacci tiling) in the \(x\) direction.

The above defines a new pattern space \(\Omega\) of \(2\)-dimensional patterns, with \(\Omega \cong \Omega_\mathrm{fib}\) and each element with periods \(\cK = \{0\} \times \R\). It is also naturally \(L\)-sub (where expansion is by the golden ratio in the \(x\)-direction) and injectivity of \(\sub\) now follows not from aperiodicity, but rather that each \(\cK\) has trivial discrete component.

One may modify this, by additionally marking the \(x\)-axis of the bar-code patterns (and also taking all their translates, and limiting patterns with no such marking). The result will be a non-minimal pattern space \(\Omega' \cong \Omega_\mathrm{fib} \times S^1\), with \(\cK_{\cP} = \{\mathbf{0}\}\) for those patterns with a horizontal marking, and \(\cK_{\cP} = \{0\} \times \R\) otherwise (i.e., for \(\cP \in \Omega\) as above). If we instead added a decoration of horizontal lines periodically repeating with unit gaps (commensurate with the expansion in the \(y\)-direction), we still obtain a tiling space \(\Omega'' \cong \Omega_\mathrm{fib} \times S^1\) but it is now minimal and substitution is no longer injective.
\end{example}

\begin{example}[Amalgams]
Consider the bi-infinite fixed point \(\sub^\infty(b) | \sub^\infty(a) = \cdots baab | abba \cdots\) of the symbolic Thue--Morse substitution, and its geometric realisation as a tiling of (marked) unit interval tiles, where \(|\) indicates the location of the origin. Delete all tiles to the left of the origin. This defines an unusual pattern, with \(\cP[x] = \cP[y]\) whenever the tiles lying over \(x\) agree exactly with the tiles lying over \(y\), up to a translation from \(x\) to \(y\). For example, \(\cP[x] = \cP[y]\) for all \(x\), \(y < 0\) and \(\cP[x] \neq \cP[y]\) for any \(x < 0\) and \(y \geq 0\).

This is \(L\)-sub with respect to \(L(x) = 2x\) and \(\cP = \cP'\) its own predecessor. Substitution is the usual one for Thue--Morse, but with void mapping to void. It is FLC and well-separated, so has compact Hausdorff hull \(\Omega\). Clearly, the standard Thue--Morse tiling space \(\Omega_\mathrm{TM}\) is a sub-pattern space, as they are locally indistinguishable from \(\cP\). We also have the `empty' pattern \(\emptyset \in \Omega\), as described in Example \ref{exp:trivial patterns}. These, along with the translational orbit of \(\cP\), give all patterns locally indistinguishable from \(\cP\), so that \(\Omega\) is the (disjoint) union \(\Omega = \Omega_\mathrm{TM} \cup (\cP + \R) \cup \{\emptyset\}\).

We may think of this as the Thue--Morse tiling space `with a stick': the orbit of \(\cP\) gives an extra copy of \(\R\) that spirals towards \(\Omega_\mathrm{TM}\), as we translate in one direction, whilst in the other direction it converges to \(\emptyset \in \Omega\). Since \(\cP\) is FLC, Corollary \ref{cor:injectivity dichotomy} applies: no \(\cQ \in \Omega\) is discretely periodic, so \(\sub \colon \Omega \to \Omega\) is injective. Indeed, directly, the restriction of \(\sub\) to the minimal component \(\Omega_\mathrm{TM}\) is injective (as can also be checked directly).

We could define a discrete analogue, where we continue with a periodic tiling on the left-hand side of \(\cP\). The hull of such a pattern now contains discretely periodic points and substitution is non-injective (the hull being a union of \(\Omega_\mathrm{TM}\), a circle \(S^1\) for the periodic tilings, where substitution is non-injective, and a leaf \(\cP+\R\) spiralling towards \(\Omega_\mathrm{TM}\) in one direction and to \(S^1\) in the other). Our theorems cover other such `amalgams', as already seen in some other examples above.
\end{example}

\begin{example}[Symbolic substitution with some periodic elements having unique pre-images]\label{exp:subs not LI}
The following shows we cannot simplify Corollary \ref{cor:uniqueness of pre-images} to uniqueness of pre-images being equivalent to discrete non-periodicity: discretely periodic points can have unique pre-images under low powers of substitution (even though they eventually have multiple pre-images under higher powers).

As a very simple example, consider the standard periodic tiling of \(E = \R^1\) by unit interval tiles, with the origin lying on a tile boundary. Consider two patterns \(\cU\) and \(\cV\) defined by certain labellings of this: \(\cU\) is given by labelling all tiles by an \(A\), and \(\cV\) is given by repeating the block \(CBBBC\). Consider the pattern space \(\Omega = \Omega_{\cU} \cup \Omega_{\cV} \cong S^1 \sqcup S^1\). Then it is \(L\)-sub, with respect to inflation map \(L(x) = 5x\) and subdivision defined by the constant length substitution
\[
A \mapsto CBBBC, \ B \mapsto AAAAA, \ C \mapsto AAAAA .
\]
Since \(\cK_{\cU} = \Z\) and \(\cK_{\cV} = 5\Z\), one may think of the circle \(\Omega_{\cV}\) as five times longer than that of \(\Omega_{\cU}\). Substitution \(\sub \colon \Omega \to \Omega\) wraps \(\Omega_{\cV}\) 25 times around \(\Omega_{\cU}\), whereas it maps \(\Omega_{\cU}\) injectively onto \(\Omega_{\cV}\). In particular, \(\sub^{-1}(\cV) = \{\cU\}\) and \(\# \sub^{-1}(\cU) = 1\), despite \(\cU\) being (discretely) periodic. Passing to the second power, however, \(\sub\) preserves LI-classes and every element has 25 pre-images.

Such behaviour can even happen for hulls of individual patterns. For instance, extend the above substitution to include new letters \(S_1\) and \(S_2\), where we let \(S_1 \mapsto S_2 A S_1 B S_2\) and \(S_2 \mapsto S_1 B S_2 A S_1\). Centring an \(S_1\) tile at the origin and iterating, we generate a fixed point tiling \(\cT\) corresponding to the following symbolic sequence (with origin located at the centre of the under-barred tile):
\[
\overbrace{\cdots \overbrace{AAAAA}^{\sub(B)} \overbrace{AAAAA}^{\sub(C)}}^{\sub^2(A)} \overbrace{\overbrace{S_1 B S_2 A S_1}^{\sub(S_2)} \overbrace{CBBBC}^{\sub(A)} \overbrace{S_2 A \underline{S_1} B S_2}^{\sub(S_1)} \overbrace{AAAAA}^{\sub(B)} \overbrace{S_1 B S_2 A S_1}^{\sub(S_2)}}^{\sub^2(S_1)} \overbrace{\overbrace{CBBBC}^{\sub(A)} \overbrace{CBBBC}^{\sub(A)} \cdots}^{\sub^2(B)}
\]
Clearly \(\cT\) is \(L\)-sub, for expansion \(L(x) = 5x\), with \(L\cT' \LD \cT\) defined by the substitution and \(\cT = \cT'\) its own predecessor. Its hull \(\Omega_{\cT}\) contains the hulls \(\Omega_{\cU} \cup \Omega_{\cV}\) above, where the discretely periodic translates of \(\cV\) still have unique pre-images.
\end{example}

\begin{example}[Variations on the Penrose tilings]\label{exp:penroses}
Any Penrose tiling (by kites and darts, or rhombs), or hull of Penrose tilings, is \(L\)-sub. The property of \(L\)-sub is invariant under MLD equivalence (see Proposition \ref{prop:MLD invariance of L-sub}, for instance). Thus, the framework would equally well apply when starting from other MLD representations, such as by Delone sets of vertices or by the coverings of Gummelt decagons, see \cite{Gum96} and Figure \ref{fig:gummelt}.

There is a related substitution \(\sub\), with inflation the golden mean squared, partially subdividing a pentagon \(p\), see \cite{Pen79} and Figure \ref{fig:several subs}. To define his famous tilings, Penrose found a systematic way of filling the gaps. Instead, let us consider the patterns generated that keep the gaps (we thank Edmund Harriss for pointing out this example as one covered by our setting). For instance, we may start with an origin-centred pentagon and iteratively apply substitution to it (and rotate 180 degrees at each step), giving a nested sequence \(P_1 \subset P_2 \subset P_3 \subset \cdots\) of patches of tiles, with union a pattern \(\cP\) with hierarchical structure. Indeed, \(L\cP'\) subdivides to \(\cP\), where \(\cP'\) is the (locally isomorphic) 180 degree rotation of \(\cP\). Since \(\cP\) has arbitrarily large holes, the void pattern is an element of its hull. It is not too hard to see that if \(\cQ \in \Omega_{\cP}\) contains a tile, it cannot be periodic, so the void pattern is the only non-periodic element of \(\Omega_{\cP}\). Thus (e.g., by Corollary, \ref{cor:injectivity dichotomy}) \(\sub \colon \Omega \to \Omega\) is injective. Indeed, it is visually clear that small sub-patches of \(\cP\) can be grouped, using a local rule, into supertiles; the inverse \(S^{-1} \colon \Omega \LDmap L\Omega\) of subdivision is LD, as claimed by Corollary \ref{cor:injectivity dichotomy}.
\end{example}

\subsection{\(L\)-sub pattern spaces from iterated inflate, replace rules}
\label{L-sub pattern spaces from iterated inflate, replace rules}

Most examples of substitutional spaces of tilings \cite{AP98} or Delone sets \cite{LW03} start with a substitution rule on a finite set of `atoms', such as a finite set of prototiles or coloured points. This is then iterated, to generate a `language' for patterns admitted by the substitution. Our approach has been, instead, to start with the substitution map on the space of patterns. Here, we briefly justify why this covers the former type of construction.

For simplicity, let us focus on a tiling space generated by an FLC substitution rule \(\sub\). This means that there are only finitely many so-called prototiles, and only finitely many ways that they may meet in generated tilings, modulo translation. Let \(\Omega\) be any `large' FLC and well-separated pattern space to which \(\sub\) restricts (it may aid the reader to see the next section on digit substitutions). For instance, in the case of a 1d substitution coming from a symbolic substitution rule, \(\Omega\) could be a suspension of the full shift. The case of higher dimensional tilings is only slightly more complicated, one may consider the space of \emph{all} tilings (with tiles translates from the prototile set) with some restriction on local patches (e.g., the way tiles can meet) coming from the FLC condition on substitution; by imposing this restriction for sufficiently large patches, one may ensure that \(\sub \colon \Omega \to \Omega\) is well-defined (i.e., that substitutions of tilings remain tilings without any partial overlaps), even for non-stone tile inflations, when these may be iterated indefinitely when starting with single prototiles.

We may then restrict \(\Omega\) to its subset \(\sH\) of hierarchical elements (see Definition \ref{def:hierarchy}). Equivalently, by Proposition \ref{prop:restriction to hierarchical elements}, we restrict substitution to its eventual range, yielding a surjective substitution map to which our main recognisability results apply.

Now, the standard construction of the space of tilings admitted by \(\sub\) is to consider the space \(\Omega_\sub\) of tilings whose finite sub-patches are each in some `supertile' \(\sub^k(t)\), for some translated prototile \(t\) and \(k \in \N_0\). But all such elements are necessarily hierarchical, so our results cover them (in fact, one often has a strict inclusion \(\Omega_\sub \subset \sH\), so even more patterns than usual are covered in this case). Indeed, \(\Omega_\sub \subseteq \Omega\) and by a standard Cantor diagonalisation argument \(\sub \colon \Omega_\sub \to \Omega_\sub\) is surjective. Similar comments apply to the setting of substitution Delone (multi)sets.

\subsection{Digit substitutions}
\label{sec:digit subs}

We conclude by demonstrating recognisability on all digit substitutions, a term used in \cite{FM22,Vin00}, although they are also variously called multidimensional constant shape substitutions \cite{Cab23, CL25} or lattice substitution systems \cite{LMS03, FS07}. These may be considered as multi-dimensional analogues of constant length substitutions from the symbolic setting.

Let \(E = \R^d\) and \(L \colon E \to E\) be a linear expansion with \(L\Z^d < \Z^d\). That is, \(L(x) = Mx\) for an integer \(d \times d\) matrix, with all eigenvalues strictly greater than one in modulus. Choose a finite `alphabet' \(A\) and `shape' \(F\) of the substitution, which is a fundamental domain \(F \subset \Z^d\) of coset representatives of \(L\Z^d < \Z^d\) i.e., we have a disjoint union \(\Z^d = \bigcup_{n \in \Z^d} (F + L(n))\). In particular, \(\#F = \det(M)\). A \textbf{digit substitution} is simply a function \(\sub \colon A \to A^F\), see the top of Figure \ref{fig:digit sub}.

\begin{figure}
\begin{center}
\includegraphics[width=0.8\textwidth]{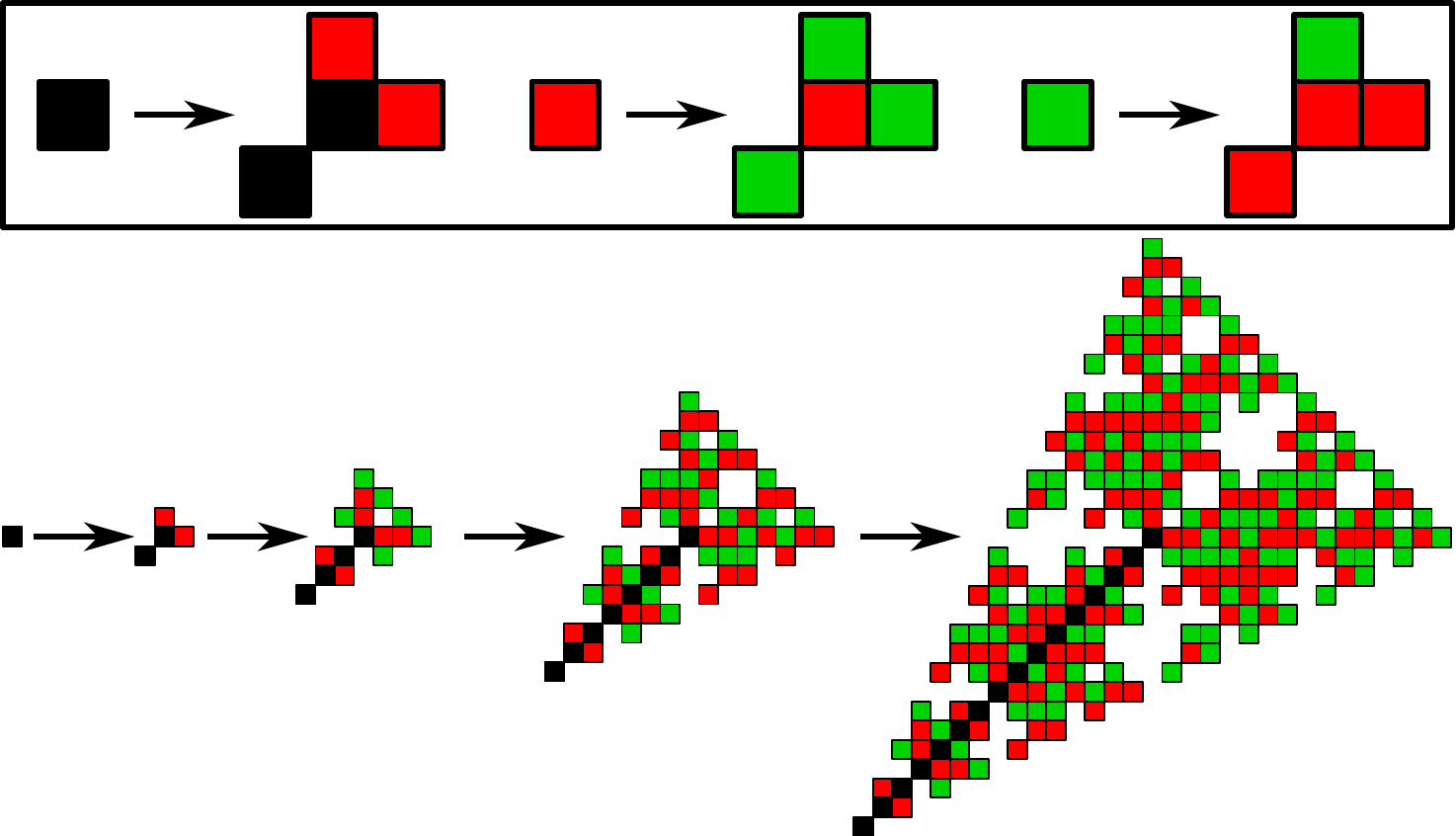}
\caption{
Example of a \(d=2\) dimensional digit substitution with 3-letter alphabet (with elements identified with different colours, above) and first four iterations of it defining supertiles (below).
\label{fig:digit sub}
}
\end{center}
\end{figure}

A digit substitution naturally extends to a substitution map \(\sub \colon A^{\Z^d} \to A^{\Z^d}\) on the full \(\Z^d\)-shift, and to its suspension, the \(\R^d\)-dynamical system \(\Omega\). More precisely, given any \(w \in A^{\Z^d}\), identify it with the pattern \(\cP \colon \R^d \to A \cup \{\emptyset\}\), where \(\cP[x] = w(x)\) when \(x \in \Z^d\) and \(\cP[x] = \emptyset\) otherwise. Let \(\Omega\) consist of all such patterns, and all their translates under \(\R^d\). Given \(\cP \in \Omega\), its inflation \(L\cP \in L\Omega\) is given by \((L\cP)[Lx] = \cP[x]\), a colouring of some translate of \(L \Z^d\) (namely, the translate \(\Z^d + Lx\), when \(\cP\) is a colouring of \(\Z^d + x\)). The original digit substitution naturally defines an LD map \(S \colon L \Omega \to \Omega\). Explicitly, given \(L\cP \in L\Omega\) and \(x \in \R^d\), if \((L\cP)[x-f] = a \in A\) for a (necessarily unique) \(f \in F\), then we let \(S(L\cP)[x] = (\sub(a))(f)\), and \(S(L\cP)[x] = \emptyset\) otherwise. Since this definition only depends on the patch of shape \(-F\) at \(x\) in \(L\cP\), this description makes it clear that \(S\) is indeed an LD map. Note that we are implicitly thinking of our patterns \(\cP \in \Omega\) as colourings of translates of \(\Z^d\), although an equivalent perspective (modulo MLD equivalence) is that each such pattern is defined by an arbitrary colouring of a standard unit hypercube tessellation.

We may now, as usual, let \(\sub \coloneqq S \circ L \colon \Omega \to \Omega\). It is easy to show that \(\Omega\) is FLC and return discrete. The set \(\sH \subseteq \Omega\) of hierarchical elements is the eventual range of \(\sub\), see Proposition \ref{prop:restriction to hierarchical elements}, to which \(\sub\) restricts as a surjective map, making \(\sH\) an \(L\)-sub pattern space. By Corollary \ref{cor:uniqueness of pre-images}, non-periodic elements have unique pre-images and there is a power \(N \in \N\) of substitution (one fixing LI-classes) so that periodic elements have multiple pre-images under \(\sub^N\). By Proposition \ref{prop:number of pre-images}, the number of pre-images of any given \(\cP \in \sH\) is given by \([\cK_{\cP} : L\cK_{\cP'}]\), for arbitrary \(\cP' \in \sub^{-1}(\cP)\) (all of which are translation equivalent).

As in the previous subsection, it is in fact more typical to consider the space \(\Omega_\sub\) (or rather its restriction to the symbolic version in \(A^{\Z^d}\)) of elements whose finite patches are all contained in translates of supertiles (given by iteratively substituting a single letter, as depicted in the bottom of Figure \ref{fig:digit sub}), see \cite{Cab23} for more details. However, as before, we always at least have \(\Omega_\sub \subseteq \sH\), so all of the above results apply to this smaller \(L\)-sub pattern space too. Injectivity of \(\sub\) on such a geometric hull is equivalent to the substitution on \(A^{\Z^d}\) being recognisable in the symbolic sense of \cite[Definition 3.4]{Cab23}.

\bibliography{references}
\bibliographystyle{alpha}

\end{document}